\documentclass[11pt]{amsart}
\usepackage{fullpage,amssymb,amsmath,verbatim,amstext,eucal}
\usepackage[all]{xy}
\newtheorem{lemma}{Lemma}[section]
 \newtheorem{proposition}[lemma]{Proposition}
\newtheorem{theorem}[lemma]{Theorem}

\newtheorem{corollary}[lemma]{Corollary}

\newtheorem{question}[lemma]{Question}

\theoremstyle{definition}
\newtheorem{definition}[lemma]{Definition}
\newtheorem{example}[lemma]{Example}

\newenvironment{remark}[1]{\refstepcounter{lemma}%
\vskip 5pt \par\noindent {\bf #1\ \thelemma .}}{\vskip 5pt \par}

\newenvironment{remark*}[1]{\par \vskip 5pt \noindent 
{\bf #1.}}{\vskip 5pt \par}

\newcommand{\Aut}{\operatorname{Aut}}

\newcommand{\bh}{\ensuremath{{\mathcal B}({\mathcal H})}}

\newcommand{\cstar}{\hbox{$C^*$}}
\newcommand{\cstaralg}{$C^*$-algebra}

\providecommand{\dual}[1]{\ensuremath{#1^{\#}}}
\providecommand{\ddual}[1]{\ensuremath{#1^{\#\#}}}
\newcommand{\dist}{\operatorname{dist}}

\newcommand{\dom}{\operatorname{dom}}
\newcommand{\ds}{\displaystyle}
\newcommand{\dstext}[1]{\quad\text{#1}\quad}
\newcommand{\eps}{\ensuremath{\varepsilon}}

\newcommand{\innerprod}[1]{\left\langle #1\right\rangle}

\newcommand{\join}{\bigvee}

\newcommand{\norm}[1]{\left\|{#1}\right\|}

\newcommand{\oinv}{\text{Inv}_{\mathcal{O}}}

\providecommand{\qed}%
{\hfill \vrule height5pt width4pt depth1pt \vspace{+2.00ex}}
\newcommand{\rad}{\operatorname{Rad}}

\newcommand{\ran}{\operatorname{range}}

\newcommand{\spn}{\operatorname{span}}
\newcommand{\supp}{\operatorname{supp}}
\newcommand{\tr}{{\rm tr}}

\newcommand{\unit}[1]{#1^{(\circ)}}

\newcommand{\bbC}{{\mathbb{C}}}

\newcommand{\bbE}{{\mathbb{E}}}
\newcommand{\bbF}{{\mathbb{F}}}

\newcommand{\bbN}{{\mathbb{N}}}

\newcommand{\bbT}{{\mathbb{T}}}

\newcommand{\bbZ}{{\mathbb{Z}}}

  \newcommand{\A}{{\mathcal{A}}}
  
  \newcommand{\B}{{\mathcal{B}}}
  \newcommand{\C}{{\mathcal{C}}}
  
  \newcommand{\D}{{\mathcal{D}}}
  
  \newcommand{\E}{{\mathcal{E}}}
  \newcommand{\F}{{\mathcal{F}}}
  \newcommand{\G}{{\mathcal{G}}}
\renewcommand{\H}{{\mathcal{H}}}
  \newcommand{\I}{{\mathcal{I}}}
  \newcommand{\J}{{\mathcal{J}}}
  \newcommand{\K}{{\mathcal{K}}} 
\renewcommand{\L}{{\mathcal{L}}}
  \newcommand{\M}{{\mathcal{M}}}
  \newcommand{\N}{{\mathcal{N}}}

\renewcommand{\P}{{\mathcal{P}}}

\renewcommand{\S}{{\mathcal{S}}}
  \newcommand{\T}{{\mathcal{T}}}
  \newcommand{\U}{{\mathcal{U}}}
  
  \newcommand{\W}{{\mathcal{W}}}
  \newcommand{\X}{{\mathcal{X}}}
  \newcommand{\Y}{{\mathcal{Y}}}

\newcommand{\fB}{{\mathfrak{B}}}
\newcommand{\fC}{{\mathfrak{C}}}

\newcommand{\fF}{{\mathfrak{F}}}

\newcommand{\fI}{{\mathfrak{I}}}
\newcommand{\fJ}{{\mathfrak{J}}}

\newcommand{\fM}{{\mathfrak{M}}}

\newcommand{\fO}{{\mathfrak{O}}}

\newcommand{\fR}{{\mathfrak{R}}}
\newcommand{\fr}{{\mathfrak{r}}}
\newcommand{\fS}{{\mathfrak{S}}}
\newcommand{\fs}{{\mathfrak{s}}}
\newcommand{\fT}{{\mathfrak{T}}}
\newcommand{\fU}{{\mathfrak{U}}}

\providecommand{\cstardiag}{\text{$C^*$-diagonal}}
\newcommand{\prt}[1]{\par \textit{Statement~(\ref{#1}).}}
\providecommand{\oinv}{\text{Inv}_{\mathcal{O}}}
\newcommand{\inter}{\I}

\newcommand{\fix}[1]{\operatorname{fix} #1}
\newcommand{\ind}{\operatorname{ind}}
\newcommand{\Mod}{\text{Mod}}
\newcommand{\homeo}{h}
\newcommand{\unistex}{\fU}
\newcommand{\tet}{\lambda}
\newcommand{\vngp}{\A}
\newcommand{\dstab}{\ensuremath{\D\text{-stab}}}
\newcommand{\ha}[1]{\rho_{(#1)}}
\newcommand{\coexp}{\bbE}
\newcommand{\emb}{w}
\newcommand{\vdb}{\delta}
\newcommand{\tog}{\text{pHom}_1}
\newcommand{\phom}{pre-homomorphism}
\newcommand{\ce}{\E_c}
\newcommand{\ceo}{\E^1_c}
\newcommand{\ceoF}{\E_F^1}
\newcommand{\fRF}{\fR_F}
\newcommand{\spmatrix}[1]{\bigl(\begin{smallmatrix} #1\end{smallmatrix}\bigr)}
\newcommand{\unitspace}[1]{#1^{(\circ)}}
\newcommand{\eval}{\eta}
\newcommand{\bist}{\fB}

\begin{document}

\title{Structure for Regular Inclusions}
\author[D.R. Pitts]{David R. Pitts}

\thanks{The author is grateful for the support of the University of
  Nebraska's NSF ADVANCE grant
  \#0811250 in the completion of this paper.} 

\address{Dept. of Mathematics\\
University of Nebraska-Lincoln\\ Lincoln, NE\\ 68588-0130}
\email{dpitts2@math.unl.edu}

\dedicatory{Dedicated to the memory of William B. Arveson}    


\begin{abstract}
  We study pairs $(\C,\D)$ of unital \cstaralg s where $\D$ is an
abelian \cstar-subalgebra of $\C$ which is regular in $\C$ in the
sense that the span of $\{v\in\C: v\D v^*\cup v^*\D v\subseteq \D\}$
is dense in $\C$.  When $\D$ is a MASA in $\C$, we prove the existence
and uniqueness of a completely positive unital map $E$ of $\C$ into
the injective envelope $I(\D)$ of $\D$ whose restriction to $\D$ is
the identity on $\D$.  We show that the left kernel of $E$,
$\L(\C,\D)$, is the unique closed two-sided ideal of $\C$ maximal with
respect to having trivial intersection with $\D$.  When $\L(\C,\D)=0$,
we show the MASA $\D$ norms $\C$ in the sense of Pop-Sinclair-Smith.
We apply these results to significantly extend existing results in the
literature on
isometric isomorphisms of norm-closed subalgebras which lie between
$\D$ and $\C$.  

The map $E$ can be used as a substitute for a conditional expectation
in the construction of coordinates for $\C$ relative to $\D$.  We show
that coordinate constructions of Kumjian and Renault which relied upon
the existence of a faithful conditional expectation may partially be
extended to settings where no conditional expectation exists.

  As an example, we consider the situation in which $\C$ is the
reduced crossed product of a unital abelian \cstaralg\ $\D$ by an
arbitrary discrete group $\Gamma$ acting as automorphisms of $\D$.  We
characterize when the relative commutant $\D^c$ of $\D$ in $\C$ is
abelian in terms of the dynamics of the action of $\Gamma$ and show
that when $\D^c$ is abelian, $\L(\C,\D^c)=(0)$.  This setting produces
examples where no conditional expectation of $\C$ onto $\D^c$ exists.

In general, pure states of $\D$ do not extend uniquely to states on
$\C$.  However, when $\C$ is separable, and $\D$ is a regular MASA in
$\C$, we show the set of pure states on $\D$ with unique state
extensions to $\C$ is dense in $\D$.  We introduce a new class of well
behaved state extensions, the compatible states; we identify
compatible states when $\D$ is a MASA in $\C$ in terms of groups
constructed from local dynamics near an element $\rho\in\hat{\D}$.
  
A particularly nice class of regular inclusions is the class of
\cstardiag s; each pair in this class has the extension property, and
Kumjian has shown that coordinate systems for \cstardiag s are
particularly well behaved.  We show that the pair $(\C,\D)$ regularly
embeds into a \cstardiag\ precisely when the intersection of the left
kernels of the compatible states is trivial.
 \end{abstract}

\maketitle

\tableofcontents

\section{Introduction, Background and Preliminaries}
In this paper, we investigate the structure of the class of regular
inclusions.       

\begin{definition}\label{mainlots}  An \textit{inclusion} is a pair $(\C,\D)$ of
  \cstar-algebras with $\D$ abelian, and $\D\subseteq \C$.  When $\C$
  has a unit $I$, we always assume that $I\in\D$. 
For any inclusion,  let 
$$\N(\C,\D):=\{v\in \C: v^*\D v\cup v\D v^*\subseteq \D\};$$ elements
of $\N(\C,\D)$ are called \textit{normalizers}.  If $(\C_i,\D_i)$
($i=1,2$) are inclusions, and $\theta:\C_1\rightarrow\C_2$ is a
$*$-homomorphism, we will say that $\theta$ is a \textit{regular
  $*$-homomorphism} if 
$$\theta(\N(\C_1,\D_1))\subseteq \N(\C_2,\D_2).$$  When $\theta$ is
regular and one-to-one, we will say that $(\C_1,\D_1)$ \textit{regularly embeds}
into $(\C_2,\D_2)$.  

The inclusion
$(\C,\D)$ is a
\newcounter{nul}
\begin{list}{}{\usecounter{nul}%
\setlength{\labelwidth}{10em}\setlength{\leftmargin}{8.5em}}
\item[\textit{regular inclusion}] if $\spn\N(\C,\D)$ is norm-dense in
$\C$;
\item[\textit{MASA inclusion}] if $\D$ is a MASA in $\C$;
\item[\textit{EP inclusion}] if $\D$ has the extension property relative to $\C$, that
  is, every pure state $\sigma$ on $\D$ has a unique extension to a
  state on $\C$ and no pure state of $\C$ annihilates $\D$;
\item[\textit{Cartan inclusion}] if $(\C,\D)$ is a regular MASA
  inclusion
 and there exists a faithful
  conditional expectation $E:\C\rightarrow\D$;
\item[\textit{\cstar-diagonal}] if  $(\C,\D)$ is a Cartan inclusion
  and also an EP-inclusion. 
\end{list}
\end{definition}

Examples of regular inclusions are commonplace in the theory of
\cstaralg s: any MASA $\D$ in a finite dimensional \cstaralg\ $\C$
yields a \cstardiag\ $(\C,\D)$; the categogy of \cstar-diagonals and
regular $*$-monomorphisms is closed under inductive
limits~\cite[Theorem~4.23]{DonsigPittsCoSyBoIs}; when a discrete group
$\Gamma$ acts topologically freely on a compact Hausdorff space $X$,
the reduced crossed product $C(X)\rtimes_r\Gamma$ together with the
canonical embedding of $C(X)$ yields a Cartan pair $(C(X)\rtimes_r
\Gamma,C(X))$~\cite{RenaultCaSuC*Al}; when $\C$ is the \cstaralg\ of a
directed graph and $\D$ is the \cstar-subalgebra generated by the
range and source projections of the partial isometries corresponding
to the edges of the graph, the pair $(\C,\D)$ is a regular inclusion.
Other examples arise from certain constructions in the theory of
groupoid \cstaralg s or from \cstaralg s constructed from
combinatoral data.

In this paper, we present a number of structural results for regular
inclusions.  Our main results include: Theorem~\ref{uniquecpmap}, which
establishes the existence and uniqueness of a psuedo-expectation for a
regular MASA inclusion; Theorem~\ref{Eideal}, which shows that the left
kernel $\L(\C,\D))$ of the pseudo-expectation is a two-sided ideal in
$\C$ which is maximal with respect to being diagaonal disjoint;
Theorem~\ref{embedCdiag}, which characterizes when a regular inclusion
can be regularly embedded into a \cstar-diagonal;
Theorem~\ref{inctoid} which shows how constructions of Kumjian and
Renault may be used to produce a twist associated to a regular
inclusion; Theorem~\ref{norming}, which shows that for a regular MASA
inclusion with $\L(\C,\D)=(0)$, $\D$ norms $\C$ in the sense of
Pop-Sinclair-Smith; and Theorem~\ref{niceiso}, which gives conditions
on which an isometric isomorphism of a subalgebra $\A$ of $\C$
containing $\D$ can be extended to a $*$ isomorphism of the \cstaralg\
generated by $\A$.  We turn now to some background.

In a landmark paper, J. Feldman and
C. Moore~\cite{FeldmanMooreErEqReII} considered pairs $(\M,\D)$
consisting of a (separably acting) von Neumann algebra $\M$ containing
a MASA $\D\simeq L^\infty(X,\mu)$ such that the set of normalizing
unitaries, $\{u\in M: \text{$u$ is unitary and } u\D u^*=\D\}$, has
$\sigma$-weakly dense span in $\M$ and there exists a faithful normal
conditional expectation $E:\M\rightarrow \D$.  In this context,
Feldman and Moore showed that there is a Borel equivalence relation
$R$ on $X$ and a cocycle $c$ such that the pair $(M,\D)$ can be
identified with a von-Neumann algebra arising from certain Borel
functions on $R$.  In a heuristic sense, their construction may be
viewed as the left regular representation of the equivalence relation
where the multiplication is twisted by the cocycle $c$.  This result
may be viewed as a means of coordinatizing the von Neumann algebra
along $\D$.

Work on a \cstar-algebraic version of the Feldman-Moore result was studied by
Kumjian in~\cite{KumjianOnC*Di}.   
In that article, Kumjian introduced the notion of a
\cstar-diagonals as well as the notion of regularity.  (We should
mention that the axioms for
a \cstar-diagonal given in~\cite{KumjianOnC*Di}, while equivalent to
the axioms given in Definition~\ref{mainlots},   do not explicitly
mention the extension property.  In the sequel, we will have considerable
interest in the extension property or its failure, which is why we use
the axioms given in Definition~\ref{mainlots}.)    Kumjian showed
that if $(\C,\D)$ is a \cstar-diagonal, with $\C$ separable and
$\hat{\D}$ second countable, then it can be coordinatized
via a twisted groupoid over a topological equivalence relation.  This
provided a very satisfying parallel to the von Neumann algebraic
context.

The requirement of the 
extension property in the axioms for a \cstardiag\ is  at times too
stringent, which is one of the advantages to Cartan inclusions, which 
need not have the extension property.    For example, let
$\H=\ell^2(\bbN)$, with the usual orthonormal basis $\{e_n\}$, and let
 $S$ be the unilateral shift, $Se_n=e_{n+1}$.  
Let $\C:=C^*(S)$ be the Toeplitz algebra, and let 
$\D=C^*(\{S^nS^{*n}: n\geq
0\})$.  Routine arguments show this is a Cartan inclusion, but the state
$\rho_\infty(T)=\lim_{n\rightarrow\infty}\innerprod{Te_n,e_n}$ on
  $\D$ fails to have a unique extension to a state on $\C$,

Cartan inclusions were 
introduced by   
Renault in~\cite{RenaultCaSuC*Al}, where he showed
that if $(\C,\D)$ is a Cartan inclusion (again with the separability
and second countability hypotheses), then there is a satisfactory
coordinatization of the pair $(\C,\D)$ via a twisted groupoid.
In this paper, Renault makes a very convincing case that Cartan
inclusions are the appropriate analog of the Feldman-Moore setting in 
 the \cstar-context.

Let $(\C,\D)$ be an inclusion.  A conditional expectation
$E:\C\rightarrow \D$ gives a preferred class,
$\{\rho\circ E: \rho\in\hat{\D}\}$, of
extensions of pure states on $\D$ to states on $\C$, and when the
expectation $E$ is unique, this class may be used for construction of
coordinates. 
Indeed, when $(\C,\D)$ is a Cartan
inclusion (or \cstardiag),   elements of
the twisted groupoid arise from ordered pairs $[v,\rho]$, where $v\in
\N(\C,\D)$ and $\rho\in \hat{\D}$ with $\rho(v^*v)\neq 0$.  
Such a pair determines a linear
functional of norm one on $\C$ by the rule,
\begin{equation}\label{vrhointro}
[v,\rho](x)=\frac{\rho(E(v^*x))}{\rho(v^*v)^{1/2}}  \qquad (x\in\C).
\end{equation}
The work of Feldman-Moore also makes essential use
of conditional expectations.

By~\cite[Theorem~3.4]{AndersonExReReStC*Al}, any EP inclusion
$(\C,\D)$ has a unique conditional expectation $E:\C\rightarrow \D$.
Thus for a \cstar-diagonal, the extension property guarantees the the
uniqueness of the expectation expectation $E:\C\rightarrow \D$.
Interestingly, the extension property is not necessary to guarantee  uniqueness
of expectations.  Indeed, Renault showed that when $(\C,\D)$ is a
Cartan inclusion, with expectation $E:\C\rightarrow \D$, then $E$ is
the unique conditional expectation of $\C$ onto $\D$.

What other regular inclusions $(\C,\D)$ have unique expectations of
$\C$ onto $\D$?  While we do
not know the answer to this question in general, we give a positive
result along these lines in Theorem~\ref{inMASAincl} below: this
result shows that when $(\C,\D)$ is a regular MASA inclusion with $\D$
 an injective \cstaralg, then $(\C,\D)$ is an EP inclusion, and therefore
has a unique expectation.   Section~\ref{secdym}  also contains 
results on dynamics of regular inclusions and introduces the
notion of a \textit{quasi-free action}.  
Theorem~\ref{freeaction} characterizes the extension property for
regular MASA inclusions in terms of the dynamics of regular
inclusions:  the regular MASA inclusion $(\C,\D)$ is an EP-inclusion
if and only if the $*$-semigroup $\N(\C,\D)$ acts quasi-freely on $\hat{\D}$.

Unfortunately, conditional expectations do not always exist, even when
$(\C,\D)$ is a regular MASA inclusion, and the \cstaralg s involved are 
well-behaved.  
Here is a simple example, which is a special case of the far more
general setting considered in Section~\ref{secCP}.
\begin{example}\label{noce}
Let $X$ be a connected, compact Hausdorff space, and let
$\alpha:X\rightarrow X$ be a homeomorphism such that $\alpha^2$ is the
identity map on $X$. Let $F^\circ$ be the interior of the set of fixed points for
$\alpha$; we assume that $F^\circ$ is neither empty nor all of $X$.
(For a concrete example, take $X=\{z\in \bbC: |z|\leq 1 \text{ and }
\Re(z)\, \Im(z)=0\}$ and let $\alpha(z)=\overline{z}$.)   Define
$\theta:C(X)\rightarrow C(X)$ by $\theta(f)=f\circ\alpha^{-1}$, and
set
$$\C:=\left\{\begin{pmatrix}f_0&f_1\\
    \theta(f_1)&\theta(f_0)\end{pmatrix}: f_0, f_1\in C(X)\right\}
\dstext{and}
\D:= \left\{\begin{pmatrix}f_0&0\\
    0&\theta(f_0)\end{pmatrix}: f_0\in C(X)\right\}.$$
Then $\C$ is a \cstar-subalgebra of $M_2(C(X))$, and  $(\C,\D)$ is a
regular inclusion.  ($\C$ may be regarded as
$C(X)\rtimes (\bbZ/2\bbZ)$.)

A calculation shows that the relative commutant $\D^c$ of $\D$ in $\C$
is
$$\D^c
=\left\{\begin{pmatrix}f_0&f_1\\
    \theta(f_1)&\theta(f_0)\end{pmatrix}\in \C : \supp(f_1)\subseteq
  F^\circ\right\}.$$ As $F^\circ\notin\{\emptyset, X\}$, we have
$\D\subsetneq\D^c\subsetneq \C$, and another calculation shows $\D^c$
is abelian.  Since $\spmatrix{0&1\\1&0}
\in\N(\C,\D^c)$, it follows that $(\C,\D^c)$ is a regular MASA
inclusion.

Suppose $E: \C\rightarrow \D^c$ is a conditional expectation.  Then
for some $f_0, f_1\in C(X)$ with $\supp(f_1)\subseteq F^\circ$, we
have $E\spmatrix{0&I\\ I&0}=\spmatrix{f_0& f_1\\
  \theta(f_1)&\theta(f_0)}$.    
 Notice 
$\theta(f_1)=f_1$ and  $\spmatrix{0&f_1\\ f_1 &0}\in \D^c$.   We have  
$$\begin{pmatrix}f_1^2& f_1\theta(f_0)\\ f_1 f_0& f_1^2\end{pmatrix}
= \begin{pmatrix}
0&f_1\\ f_1 &0\end{pmatrix} 
E\begin{pmatrix}0&I\\ I&0\end{pmatrix}
=E\left(\begin{pmatrix}
0&f_1\\ f_1 &0\end{pmatrix}
\begin{pmatrix}0&I\\ I&0\end{pmatrix} \right)=\begin{pmatrix}
f_1&0\\ 0&f_1\end{pmatrix},$$ so $f_1^2=f_1$.
As $X$ is connected, this yields $f_1=0$ or $f_1=I$.  But
$\supp(f_1)\subseteq F^\circ\neq X$,  so $f_1=0$.  Thus
$E(\spmatrix{0&I\\ I&0})=\spmatrix{f_0& 0\\ 0&\theta(f_0)}$.  

Now if $g_1\in \D$ is such that $\supp(g_1)\subseteq F^\circ$, then
$\theta(g_1)=g_1$, so $\spmatrix{0&g_1\\ g_1 & 0}\in \D^c$.  Thus,
\begin{align*}\begin{pmatrix}
0& g_1\\ g_1 &0\end{pmatrix}
&=E\left(\begin{pmatrix}
 g_1& 0\\ 0 &g_1\end{pmatrix} \begin{pmatrix} 0&I\\
 I&0\end{pmatrix}\right)=
\begin{pmatrix}g_1&0\\ 0& g_1
\end{pmatrix}
E\begin{pmatrix} 0 & I\\ I& 0
\end{pmatrix}
= \begin{pmatrix}g_1&0\\ 0& g_1
\end{pmatrix}
\begin{pmatrix}
f_0&0\\ 0&\theta(f_0)
\end{pmatrix}\\
&=\begin{pmatrix} g_1f_0& 0\\ 0 &\theta(g_1f_0)
\end{pmatrix}.
\end{align*}
Hence $g_1=0$ for every such $g_1$.  This
implies that $F^\circ=\emptyset$, contrary to hypothesis.  Hence no
conditional expectation of $\C$ onto $\D^c$ exists.
\end{example}

One of the goals of this paper
 is to show that even though conditional expectations
may fail to exist for a regular MASA inclusion, there is a map which
which may be used 
as a replacement.  Here is the relevant definition.
\begin{definition} Let $(\C,\D)$ be an inclusion and let
  $(I(\D),\iota)$ be an injective envelope for $\D$.  A
  \textit{pseudo-conditional expectation for 
    $\iota$}, or more simply, a
  \textit{pseudo-expectation for $\iota$},  is a unital
  completely positive map $E:\C\rightarrow I(\D)$ such that
  $E|_\D=\iota$.      When the context is clear, we sometimes drop the
  reference to $\iota$ and simply call $E$ a pseudo-expectation.
\end{definition}
The existence of pseudo-expectations follows
immediately from the injectivity of $I(\D)$.  In general, the
pseudo-expectation need not be unique.

However, in Section~\ref{MainExpect} below, we show that for any
regular MASA inclusion $(\C,\D)$, there is always a unique
pseudo-expectation $E:\C\rightarrow I(\D)$, see
Theorem~\ref{uniquecpmap}.  Let $\Mod(\C,\D)$ be the family of all
states on $\C$ which restrict to elements of $\hat{\D}$.  The family
of states, $\fS_s(\C,\D):=\{\rho\circ E: \rho\in \widehat{I(\D)}\}$
covers $\hat{\D}$ in the sense that the restriction map,
$\fS_s(\C,\D)\ni \rho\mapsto \rho|_\D\in \hat{\D}$, is onto.  Interestingly,
$\fS_s(\C,\D)$ is the unique minimal closed subset of $\Mod(\C,\D)$
which covers $\hat{\D}$, see Theorem~\ref{minonto}.  We also show that
$\fS_s(\C,\D)$ is closely related to the extension property.  When
$(\C,\D)$ is ``countably generated,'' Theorem~\ref{minonto} also shows
that $\fS_s(\C,\D)$ is the closure of all states in $\Mod(\C,\D)$
whose restrictions to $\D$ extend uniquely to $\C$.

For a regular MASA inclusion, the intersection of the left kernels of
the states in $\fS_s(\C,\D)$ is the left kernel of the
pseudo-expectation $E$.  Let $\L(\C,\D)$ be the left kernel of $E$.  
Theorem~\ref{Eideal} shows that $\L(\C,\D)$ is an ideal of $\C$, and
moreover, is the unique ideal of $\C$  which is maximal with
respect to the property of having trivial intersection with $\D$.
When the pseudo-expectation takes values in $\D$ (rather than
$I(\D)$), the 
 ideal $\L(\C,\D)$ may be viewed as a measure of the failure
  of the inclusion to be Cartan in Renault's sense.    

We define a regular MASA inclusion to be a \textit{virtual Cartan
  inclusion}  
when  the pseudo-expectation is faithful, or equivalently, when
$\L(\C,\D)=0$.    The purpose of Section~\ref{secCP} is to
give a large class of virtual Cartan inclusions.  
Theorem~\ref{L4DCP} shows that when
$\C$ is the reduced crossed product of the abelian \cstaralg\ $\D$ by a
discrete group $\Gamma$, then, provided the relative commutant $\D^c$
of $\D$ in $\C$ is abelian, $(\C,\D^c)$ is a virtual Cartan
inclusion.  We characterize when $\D^c$ is abelian in terms of the
dynamics of the action of $\Gamma$ on $\hat{\D}$ in
Theorem~\ref{abelHxabelcom};  this result shows that $\D^c$ is abelian
precisely when the germ isotropy subgroup $H^x$ of $\Gamma$ is abelian
for every $x\in\hat{\D}$.  
The results of Section~\ref{secCP} are summarized
in Theorem~\ref{embedcrossedprod}.

One of the  motivations for our study of inclusions was to
provide a context for the study of certain nonselfadjoint
subalgebras.  If $(\C,\D)$ is a \cstardiag, there are numerous papers
devoted to the study of various (usually nonselfadjoint) closed algebras $\A$
with $\D\subseteq \A\subseteq \C$, see 
\cite{DavidsonPowerIsAuHo,DonsigHudsonKatsoulisAlIsLiAl,%
DonsigPittsCoSyBoIs,HopenwasserPetersPowerSuGrC*Al,%
MuhlyQiuSolelCoNuSpSuOpAl,MuhlySolelOnTrSuGrC*Al,MuhlySaitoSolelCoTrOpAl}
to name just a  few.   An often 
successful strategy for the analysis of the subalgebras
of \cstardiag s is to use the coordinatization of $(\C,\D)$ (via the
twist) to impose coordinates on the subalgebras; properties of the
coordinate system then reflect properties of the subalgebra.  

Since many classes of regular MASA inclusions are neither \cstardiag s nor
Cartan inclusions, it
is natural to wonder whether coordinate methods may be used to analyze
nonselfadjoint subalgebras of regular MASA inclusions.  A strategy for
doing so is to try to regularly embed a given regular MASA inclusion $(\C,\D)$
into a \cstardiag\ $(\C_1,\D_1)$ and then to restrict the coordinates
obtained from $(\C_1,\D_1)$ to the $(\C,\D)$ or to the given
subalgebra.   This leads to the following
problem.
\begin{remark}{Problem}\label{embedprob}  Characterize when a  given regular
inclusion $(\C,\D)$ can be regularly embedded into a \cstardiag.
\end{remark}
We give a solution to Problem~\ref{embedprob} in
Section~\ref{sectRegEm}.    To do this, we introduce  a new family
$\fS(\C,\D)\subseteq \Mod(\C,\D)$, which we call \textit{compatible
  states}.    When $(\C,\D)$ is a regular MASA inclusion,
$\fS_s(\C,\D)\subseteq \fS(\C,\D)$.  The intersection of the left
kernels of the states in $\fS(\C,\D)$ is an ideal of $\C$,
$\rad(\C,\D)$.   The regular inclusion $(\C,\D)$ regularly embeds in a
\cstardiag\ if and only if $\rad(\C,\D)=(0)$, see
Theorem~\ref{embedCdiag}.  In particular, any virtual Cartan inclusion
regularly embeds into a \cstardiag.

The needed properties of compatible states are developed 
in Section~\ref{sectCompSt}.  Compatible states can be defined for any
inclusion, and we expect that they may be useful in other contexts as
well.  While compatible states exist in abundance for
any regular MASA inclusion,~Theorem~\ref{allunitaries}  implies
 that compatible states need not exist for a general regular inclusion.

 For a regular MASA inclusion $(\C,\D)$, it is always the case that
 $\rad(\C,\D)\subseteq \L(\C,\D).$ We have been unable to resolve the
 question of whether equality holds.  We provide some insight into
 this question in Section~\ref{desCSregMASA}.  Given $\sigma\in
 \hat{\D}$, there is an equivalence relation $R_1$ on
 $H_\sigma:=\{v\in \N(\C,\D): \rho(v^*dv)=\rho(d) \text{ for all $d\in
   \D$}\}$, and the set $H_\sigma/R_1$ of equivalence classes of this
 equivalence relation may be made into a $\bbT$-group.  The main
 result of this section, Theorem~\ref{compatstpdf}, shows that there
 is a bijective correspondence between $\{\rho\in\fS(\C,\D):
 \rho|_\D=\sigma\}$ and a certain family of \phom s on $H_\sigma/R_1$.
 Theorem~\ref{compatstpdf} thus gives a description of a certain class
 of state extensions of $\sigma$.  Our description leads us to suspect
 that it is possible for $\rad(\C,\D)$ to be a proper subset of
 $\L(\C,\D)$.

The purpose of Section~\ref{GpReIn} is to discuss certain twisted
groupoids arising from a regular MASA inclusion.  The methods of this
section are suitable modifications to our context of the methods used
by Kumjian and Renault when coordinatizing \cstardiag s and Cartan
inclusions.  We show that given any regular MASA inclusion $(\C,\D)$,
there is a twist associated to $(\C,\D)$.  We use $\fS_s(\C,\D)$ or
another suitable subset $F$ of $\fS(\C,\D)$ as unit space of the
twist, and define functionals $[v,\rho]$ on $\C$ much as
in equation~\eqref{vrhointro}, except that the functional $\rho\circ E$
appearing in that formula is replaced with an element of
$F$.  The main result of this
section, Theorem~\ref{inctoid}, shows that when $\rad(\C,\D)$ is
trivial, there is a regular $*$-monomorphism of $(\C,\D)$ into the
Cartan inclusion arising from the twist associated to $(\C,\D)$.  This
result gives perspective to the embedding results of
Section~\ref{sectRegEm}, and suggests that it is indeed possible to
coordinatize subalgebras using this twist, as indicated prior to
Problem~\ref{embedprob}.

In~\cite{PittsNoAlAuCoBoIsOpAl}, we gave a method for extending an
isometric isomorphism between subalgebras $\A_i$ of \cstardiag s
$(\C_i,\D_i)$ with $\D_i\subseteq \A_i\subseteq \C_i$ to a
$*$-isomorphism of the \cstar-subalgebra $C^*(\A_1)$ of $\C_1$
generated by $\A_1$ onto the corresponding subalgebra $C^*(\A_2)$.
The two main ingredients of this method were:  a) show that $\D_i$
norms $\C_i$ in the sense of Pop-Sinclair-Smith
(\cite{PopSinclairSmithNoC*Al}), and b) show that the $C^*$-envelope
of $\A_i$ is isometrically isomorphic to $C^*(\A_i)$.   In
Section~\ref{appl}, we show that if the hypothesis that $(\C_i,\D_i)$
is weakened from \cstardiag\ to virtual Cartan inclusion, then 
both these ingredients still hold.   Theorem~\ref{niceiso}
generalizes~\cite[Theorem~2.16]{PittsNoAlAuCoBoIsOpAl} 
to the context of virtual Cartan inclusions.  This is a considerable
generalization, and allows for the simplification of some arguments in
the literature.

The author is indebted to William Arveson, who greatly influenced the
author and the field of operator algebras.  His passing saddens us all.

We thank Ken Davidson, Allan Donsig, William Grilliette, Vern
Paulsen, and Vrej Zarikian for several very helpful conversations.

\subsection{Preliminaries}   Given
a Banach space $\X$, we will use $\dual{\X}$ instead of the traditional
$\X^*$ to denote the Banach space dual.  Likewise if $\alpha:
\X\rightarrow \Y$ is a bounded linear map between Banach spaces, we use
$\dual{\alpha}$ to denote the adjoint map, $f\in \dual{\Y}\mapsto
f\circ\alpha\in \dual{\X}.$

If $X$ is a topological space and $E\subseteq X$, 
 $E^\circ$ denotes the interior of  $E$. 
Also, for $f:X\rightarrow \bbC$, we write
$\supp{f}$ for the set $\{x\in X: f(x)\neq 0\}$.

\begin{remark}{Standing Assumption}  For the remainder of this paper, all
  \cstaralg s will be unital, and if $\D$ is a sub-\cstaralg\ of the
  \cstaralg\ $\C$, we assume that the unit for $\D$ is the same as the
  unit for $\C$.
\end{remark}

Let $\C$ be a \cstaralg, and let
$\S(\C)$ be the state space of $\C$.  For $\rho\in\S(\C)$ let
$L_\rho=\{x\in\C:\rho(x^*x)=0\}$ be the
  left kernel of $\rho$, and let $(\pi_\rho,\H_\rho, \xi)$ be the GNS
  representation corresponding to $\rho$.  We regard $\C/L_\rho$ as a
  dense subset of $\H_\rho$, and for $x\in\C$ will often write
  $x+L_\rho$ to denote the vector $\pi_\rho(x)\xi$.  
  Denote the inner product on $\H_\rho$ by $\innerprod{\cdot,\cdot}_\rho$.

We now
recall some facts about projective topological spaces, projective
covers, and injective envelopes of abelian \cstaralg s.  
Following~\cite{HadwinPaulsenInPrAnTo}, given a compact
Hausdorff space $X$, a pair $(P,f)$ consisting of a compact Hausdorff
space $P$ and a continuous map $f:P\rightarrow X$ is called a
\textit{cover for $X$} (or simply a \textit{cover})
if $f$ is surjective.  A cover $(P,f)$ is \textit{rigid} if the only
continuous map $h:P\rightarrow P$ which satisfies $f\circ h=f$ is
$h=\text{id}_P$; the cover $(P,f)$ is \textit{essential} if whenever
$Y$ is a compact Hausdorff space and $h:Y\rightarrow P$ is continuous
and satisfies $f\circ h$ is onto, then $h$ is onto.

A compact
Hausdorff space $P$ is \textit{projective}  if whenever $X$ and $Y$ are
compact Hausdorff spaces and $h:Y\rightarrow X$ and $f: P\rightarrow
X$ are continuous maps with $h$ surjective,  there exists a
continuous map $g: P\rightarrow Y$ with $g\circ h=f$.  A Hausdorff
space which is extremally disconnected (i.e.\ the closure of every
open set is open) and compact is
\textit{Stonean}.  In~\cite[Theorem~2.5]{GleasonPrToSp}, Gleason
proved  that a
compact Hausdorff space $P$ is projective if and only $P$ is Stonean.

By~\cite[Proposition~2.13]{HadwinPaulsenInPrAnTo}, if $(P,f)$ is a
cover for $X$ with $P$ a projective space, then $(P,f)$ is rigid if
and only if $(P,f)$ is essential.  A \textit{projective cover} for $X$
is a rigid cover $(P,f)$ for $X$ such that  $P$ is
projective.  A projective cover for $X$ always
exists~\cite[Theorem~2.16]{HadwinPaulsenInPrAnTo} and is unique in the
sense that if $(P_1,f_1)$ and $(P_2, f_2)$ are projective covers for
$X$, then there is a unique homeomorphism $h:P_1\rightarrow P_2$ such
that $f_1=f_2\circ h$.

The concept of an injective envelope for an abelian unital \cstaralg s
is  dual to the concept of a projective cover of a compact
Hausdorff space: if
$(P,f)$ is a cover for $X$, let $\iota: C(X)\rightarrow C(P)$ be the
map $d\mapsto d\circ f$; then $(P,f)$ is a projective cover if and
only if $(C(P),\iota)$ is an injective, envelope of
$C(X)$~\cite[Corollary~2.18]{HadwinPaulsenInPrAnTo}.  
  (Injective envelopes can be also be defined for general unital \cstaralg s,
  not just abelian, unital \cstaralg s.  
  Like projective covers, injective envelopes of unital \cstaralg s
  have a uniqueness property.  If $\A$ is a unital \cstaralg, and
  $(\B_1,\sigma_1)$ and $(\B_2,\sigma_2)$ are injective envelopes for
  $\A$, then there exists a unique $*$-isomorphism $\theta:
  \B_1\rightarrow \B_2$ such that
  $\theta\circ\sigma_1=\sigma_2$~\cite[Theorem~4.1]{HamanaInEnC*Al}.)

  Recall that a unital \cstaralg\ $\C$ is \textit{monotone complete} if every
  bounded increasing net in the self-adjoint part, $\C_{s.a.}$, of
  $\C$ has a least upper bound in $\C_{s.a.}$.
  Hamana~\cite{HamanaReEmCStAlMoCoCStAl} shows that every injective
  \cstaralg\ is monotone complete.  When $(x_\lambda)$ is a bounded
  increasing net in $\C_{s.a.}$,
  $\sup_{\C}x_\lambda$ means the least upper bound of $(x_\lambda)$ in $\C$. 

Theorem~6.6 of \cite{HamanaReEmCStAlMoCoCStAl} implies that if
$\A$ is a unital, abelian \cstaralg\ then any injective envelope
$(\B,\sigma)$ for $\A$ is \textit{Hamana-regular} in the sense that
whenever $x\in \B$ is self-adjoint, $x=\sup_{\B}\{\sigma(a): a\in\A,
a=a^*\text{ and } a\leq x\}$.  (Since $\A$ is abelian, we 
 regard $\{\sigma(a): a\in\A,
a=a^*\text{ and } a\leq x\}$ as a net indexed by itself.)  As Hamana observes
in~\cite{HamanaReEmCStAlMoCoCStAl}, when $x\in \A$ is positive, then
also, $x=\sup_{\B}\{\sigma(a): a\in\A\text{ and } 0\leq \sigma(a)\leq
b\}$.

Here is a description of an injective envelope of an abelian
\cstaralg.  For details, see~\cite[Theorem~1]{GonshorInHuC*AlII}.
Let $X$ be a compact Hausdorff space.  Define an equivalence
relation on the algebra $\B(X)$ of all bounded Borel complex-valued functions on
$X$ by $f\sim g$ if and only if $\{x\in X: f(x)-g(x)\neq 0\}$ is a set
of first category.  The equivalence class $J$ of the zero function is an
ideal in $\B(X)$ and the quotient $\D(X):=\B(X)/J$ is called the
\textit{Dixmier algebra}.  Define $j: C(X)\rightarrow \D(X)$ by
$j(f)=f+J$.  Then $(\D(X),j)$ is an injective envelope for $\D$.

  We conclude this section with a few comments regarding categories.  
  Let $\fC$ be the category of unital abelian \cstaralg s with
  $*$-homomorphisms, and let $\fO$ be the category of operator systems
  with completely positive (unital) maps.  Let $\D$ be a unital
  abelian \cstaralg.  Then $\D$ is an injective object in $\fC$ if and
  only if $\D$ is an injective object in $\fO$,
  see~\cite[Theorem~2.4]{HadwinPaulsenInPrAnTo}.  (The statement
  of~\cite[Theorem~2.4]{HadwinPaulsenInPrAnTo}, mentions the category
  of operator systems without explicitly giving the morphisms, but the
  proof makes it clear that the authors mean the category of operator
  systems and unital, completely positive maps.)  

  Hamana shows that in the category $\fO$, there is an injective
  object $I(\D)$ and a one-to-one completely positive
  $\iota:\D\rightarrow I(\D)$ such that the extension $(I(\D),\iota)$
  is rigid and essential.  Hamana and Hadwin-Paulsen
  (see~\cite{HamanaInEnOpSy}
  and~\cite[Corollary~2.18]{HadwinPaulsenInPrAnTo}) observe that
  $I(\D)$ is endowed with a product which makes it into an abelian
  \cstaralg\ (and $\iota$ a $*$-monomorphism).  Set $X=\hat{\D}$, and
  let $(P,f)$ be a projective cover for $X$, so that the map
  $\tau:\D\rightarrow C(P)$ given by $\tau(x)=\hat{x}\circ f$ is a
  one-to-one $*$-homomorphism of $\D$ into $C(P)$.  Corollary~2.18 of
  \cite{HadwinPaulsenInPrAnTo} also shows the existence of a
  $*$-isomorphism $\theta:C(P)\rightarrow I(\D)$ such that
  $\theta\circ\tau=\iota$.

Thus, for us an injective envelope
 $(I(\D),\iota)$ for $\D$ will be a rigid and essential extension of
 $\D$ in
 the category $\fC$.  The comments above show that this is
 equivalent to saying that $(I(\D),\iota)$  is
 a rigid and essential extension for $\D$ in $\fO$.

\section{Dynamics of Regular Inclusions}\label{secdym}

Given a regular inclusion $(\C,\D)$, the $*$-semigroup $\N(\C,\D)$ of
normalizers acts via partial homeomorphisms on the maximal ideal space
$\hat{\D}$ of $\D$.  The purpose of this section is to discuss some of
the features of this action.  The first subsection is devoted to
notation and some background facts regarding normalizers and
intertwiners.
The second subsection gives a characterization of the
extension property  in terms of the
   dynamics associated with the action of $\N(\C,\D)$ on $\hat{\D}$ 
(Theorem~\ref{freeaction}).
An interesting consequence is that for any regular
MASA inclusion $(\C,\D)$ with $\D$ injective is an EP inclusion, see
Theorem~\ref{inMASAincl}.

\subsection{Normalizers and Intertwiners}

Let $(\C,\D)$ be an inclusion.  Closely related to normalizers are
intertwiners. 
\begin{definition}
An \textit{intertwiner} for $\D$ is an element $v\in\C$ such that
$v\D=\D v$.  We denote the set of all intertwiners by $\inter(\C,\D)$.
\end{definition}

Proposition~3.3 of~\cite{DonsigPittsCoSyBoIs} shows
that for any intertwiner $v$, the elements $v^*v$ and $vv^*$ belong to
the relative commutant $\D^c$ of $\D$ in $\C$.  When $v$ is a
normalizer, the fact that $\D$ is unital shows that both $v^*v$ and $vv^*$
belong to $\D$, and it is easy to see that 
$\{v\in \inter(\C,\D): v^*v, vv^*\in \D\}\subseteq \N(\C,\D)$.
   In general however, there exists $v\in\inter(\C,\D)$ with
   $\{v^*v,vv^*\}\not\subseteq \D$.  (For a simple example, observe that every
operator in $M_2(\bbC)$ is an intertwiner for the inclusion,
$(M_2(\bbC),\bbC I_2)$.)  Also, by 
\cite[Proposition~3.4]{DonsigPittsCoSyBoIs}, $\N(\C,\D)\subseteq
\overline{\inter(\C,\D)}$.  However, the final paragraph of the proof of
\cite[Proposition~3.4]{DonsigPittsCoSyBoIs}, shows somewhat more than this,
namely that $\overline{\{v\in \inter(\C,\D): v^*v, vv^*\in\D\}}=
\N(\C,\D).$    Summarizing, we have the following result.
\begin{proposition}[{\cite[Propositions~3.3 and
  3.4]{DonsigPittsCoSyBoIs}}] \label{intnorrel}
Let $(\C,\D)$ be an inclusion. Then
\begin{equation*}
\overline{\{v\in \inter(\C,\D): v^*v, vv^*\in\D\}}= \N(\C,\D)\subseteq
\overline{\inter(\C,\D)}.
\end{equation*}
  Furthermore, when $\D$ is a MASA in $\C$,
  $\N(\C,\D)=\overline{\inter(\C,\D)}$.
\end{proposition}

The following example of a regular homomorphism will be useful in the sequel.
\begin{lemma}\label{abelcom} Suppose $(\C,\D)$ is an
  inclusion such that the relative commutant, $\D^c$, of $\D$ in $\C$
  is abelian.  Then $(\C,\D^c)$ is a MASA inclusion and
  $\N(\C,\D)\subseteq \N(\C,\D^c)$; in particular the
  identity map $\text{id}: (\C,\D)\rightarrow (\C,\D^c)$ is a regular
  $*$-homomorphism.  
\end{lemma}
\begin{proof}
Since $\D$ and $\D^c$ are abelian,  $(\C,\D^c)$ is a MASA
inclusion. 

To show that $\N(\C,\D)\subseteq \N(\C,\D^c)$, suppose
$v\in\inter(\C,\D)$.  Fix $h\in\D^c$, and let $d\in \D$ be
self-adjoint.  Since $v$ is a $\D$-intertwiner, we may find a
self-adjoint $d'\in\D$ so that $dv=vd'$.  Then
$d(vhv^*)=vd'hv^*=vhd'v^*=(vhv^*)d,$ from which it follows that
$vhv^*\in\D^c$.  Similarly, $v^*h v\in\D^c$.  We conclude that
$\inter(\C,\D)\subseteq \N(\C,\D^c)$.  Since $\N(\C,\D^c)$ is
norm-closed, Proposition~\ref{intnorrel} yields,
$$\N(\C,\D)\subseteq \overline{\inter(\C,\D)}\subseteq \N(\C,\D^c).$$
Thus, $(\C,\D^c)$ is a  MASA inclusion and the identity mapping
$\text{id}:(\C,\D)\rightarrow (\C,\D^c)$ is a regular
$*$-homomorphism.
\end{proof}

For any topological space $X$, a \textit{partial homeomorphism} is a
homeomorphism $h: S\rightarrow R$, where $S$ and $R$ are open subsets
of $X$.  As usual, $\text{dom}(h)$ and $\text{ran}(h)$ will denote the
domain and range of the partial homeomorphism $h$.  We use $\oinv(X)$
to denote the inverse semigroup of all partial homeomorphisms of $X$.
When $\S$ is a $*$-semigroup, a semigroup homomorphism
$\alpha:\S\rightarrow \oinv(X)$ is a \textit{$*$-homomorphism} if for
every $s\in\S$, $\alpha(s^*)=\alpha(s)^{-1}.$ A subset $\G$ of
$\oinv(X)$ which is closed under composition and inverses (i.e. a sub
inverse semigroup) is called a \textit{pseudo-group on $X$}.
Associated to any pseudo-group $\G$ on $X$ is the \textit{groupoid of
  germs} which is the set of equivalence classes, $\{[x,\phi,y]:
\phi\in \G, y\in \dom(\phi), x=\phi(y)\}$, where $[x,\phi,
y]=[x_1,\phi_1,y_1]$ if and only $y=y_1$ and there exists a
neighborhood $N$ of $y$ such that $\phi|_N=\phi_1|_N$.  Elements
$[w,\phi,x]$ and $[y,\psi,z]$ are composable if $x=y$ and then
$[w,\phi,x]\, [y,\psi, z]=[w,\phi\psi,z]$ and
$[x,\phi,y]^{-1}=[y,\phi^{-1},x]$.  The range and source maps are
$r([x,\phi,y])=x$ and $s([x,\phi,y])=y$.  Thus $X$ may be identified
with the unit space of the groupoid of germs.

Recall (see~\cite[Proposition~6]{KumjianOnC*Di}) that a
normalizer $v$ determines a partial homeomorphism 
$$\beta_v: \{\rho\in
\hat{\D}: \rho(v^*v)> 0\}\rightarrow\{\rho\in\hat{\D}:
\rho(vv^*)> 0\}\dstext{given by}
\beta_v(\rho)(d)=\frac{\rho(v^*dv)}{\rho(v^*v)}\qquad (d\in\D).$$

Clearly $\N(\C,\D)$ and $\inter(\C,\D)$ are $*$-semigroups under
multiplication.  Routine, but tedious, calculations show that the map
$\N(\C,\D)\ni v\mapsto \beta_v$ is a $*$-semigroup homomorphism
$\beta: \N(\C,\D)\rightarrow \oinv(\hat{\D})$.  We record this fact as
a proposition.

\begin{proposition}[{\cite[Lemma~4.10]{RenaultCaSuC*Al}}]\label{betacalc}  Suppose $(\C,\D)$ is an
  inclusion.  Then the following statements hold.

\begin{enumerate}
\item Suppose that $v,w\in\N(\C,\D)$ and $\rho\in\hat{\D}$ satisfies
  $\rho(w^*v^*vw)\neq 0$.  Then $\rho(w^*w)\neq 0$, and 
   $\beta_{vw}(\rho)=\beta_v(\beta_w(\rho))$.
\item For every $v\in\N(\C,\D)$, 
  $\beta_{v^*}=(\beta_v)^{-1}$. 
\end{enumerate}
\end{proposition}

Following 
Renault~\cite{RenaultCaSuC*Al} we will  call the
collection, 
$$\P\G(\C,\D):=\{\beta_v:
v\in\N(\C,\D)\}$$ the \textit{Weyl pseudo-group} of the inclusion.
Also, 
 the groupoid of germs 
of $\P\G(\C,\D)$  is the \textit{Weyl
  groupoid} of the inclusion, which we denote by $\W\G(\C,\D)$.

\begin{remark}{Remarks} 
\begin{enumerate}
\item\label{Dpres} 
Observe that if $\theta$ is a regular homomorphism, then
$\theta(\D_1)\subseteq \D_2$: indeed, for $d\in\D_1$ with $d\geq 0$,
$d^{1/2}\in \N(\C_1,\D_1)$ and so $\theta(d) =
\theta(d^{1/2})1\theta(d^{1/2})\in \D_2.$  
It follows that the dynamics of inclusions under regular
  homomorphisms are well-behaved in the sense that if
  $\theta:(\C_1,\D_1)\rightarrow (\C_2,\D_2)$ is a regular
  homomorphism of the inclusions $(\C_i,\D_i)$, then whenever
  $v\in\N(\C_1,\D_1)\setminus\ker\theta$, the following diagram commutes:
$$
\xymatrix{
\hat{\D}_2\supseteq\dom(\beta_{\theta(v)})
\ar[r]^<<<<{\beta_{\theta(v)}} \ar[d]_{\dual{\theta}}     &     
\ran(\beta_{\theta(v)})\subseteq \hat{\D}_2 \ar[d]^{\dual{\theta}}\\
\hat{\D}_1\supseteq\dom(\beta_v)   \ar[r]^<<<<<{\beta_v}    &  \ran(\beta_v)\subseteq \hat{\D}_1.
}$$

\item For some purposes, it is easier to work with intertwiners than
  normalizers.  Thus, one might be tempted to define a regular homomorphism using
  intertwiners instead of normalizers, that is, by mandating
  $\theta(\I(\C_1,\D_1))\subseteq \I(\C_2,\D_2)$.  However, a 
  disadvantage of doing so is that such a $\theta$ need not carry $\D_1$
  into $\D_2$, which is why we use $\N(\C,\D)$ rather than $\I(\C,\D)$
  in Definition~\ref{mainlots}.  For an example of this, let
  $\C=C([0,1])$ and let $\D=\{f\in C([0,1]): f(0)=f(1/2)\}$.
  Then $\C=\I(\C,\D)$, and since every unitary in $\C$ normalizes
  $\D$, we see that $(\C,\D)$ is regular (as in
  Definition~\ref{mainlots}).  Taking $(\C_i,\D_i)=(\C,\D)$, then any
  automorphism $\theta$ of $\C$ satisfies $\theta(\I(\C_1,\D_1))\subseteq
  \I(\C_2,\D_2)$, yet clearly one may choose $\theta$ so that
  $\theta(\D_1)\not\subseteq \theta(\D_2).$
\end{enumerate}
\end{remark}

\subsection{Quasi-Freeness and the Extension Property}

By Proposition~\ref{betacalc}, $\S:=\{\beta_v: v\in\N(\C,\D)\}$ is an
inverse semigroup of partial homeomorphisms of
$\hat{\D}$.  Recall that a group $G$ of homeomorphisms of a
space $X$ acts freely if whenever $g\in G$ has a fixed point, then $g$ is the
identity.  Paralleling the notion for groups, we make the following definition.

\begin{definition}  Suppose that $\S$ is a $*$-semigroup and $X$ is a
  compact Hausdorff space, and that $\alpha:\S\rightarrow \oinv(X)$ is
  a $*$-semigroup homomorphism.   We say that $\S$ \textit{acts quasi-freely} on $X$
if whenever $s\in \S$, $\{ x\in \text{dom}(\alpha(s)): 
 \alpha(s)(x)
=x\}$ is an open set in $X$.  
\end{definition}
When $\Gamma$ is a group acting quasi-freely on $X$, this says that
for each $s\in \Gamma$, the set of fixed points of $\alpha(s)$ is a
clopen set; in particular, when $X$ is a connected set, the notions of
free and quasi-free actions for a group (acting as homeomorphisms) on $X$ coincide.

In some circumstance, quasi-freeness is automatic.  Recall that a
Hausdorff topological space $X$ is \textit{extremally disconnected} if
the closure of every open subset of $X$ is open and that $X$ is a
\textit{Stonean space} if $X$ is compact, Hausdorff, and extremally
disconnected.  We shall show that if a $*$-semigroup acts on a Stonean
space, then it acts quasi-freely.  To do this we require the following
topological proposition. The proof is a straightforward adaptation of
the elegant proof by Arhangel$'$skii of Frol\'{i}k's
Theorem~\cite[Theorem~3.1]{FrolikMaExDiSp} on fixed points of
homeomorphisms of extremally disconnected spaces.  We provide a sketch
of the proof for the convenience of the reader.
\begin{proposition}[Frol\'{i}k's Theorem]\label{fixed} Let $X$ be an
  extremally disconnected space,
let $V, W$ be clopen subsets of $X$, and suppose $\homeo: V\rightarrow
W$ is a homeomorphism of $V$ onto $W$.  Then the set of fixed points
$F:=\{x\in V: \homeo(x)=x\}$ is a clopen subset of $X$.  Moreover,
there are three disjoint clopen subsets $C_1, C_2, C_3$ of $X$ such that for
$i=1,2,3,$ $\homeo(C_i)\cap C_i=\emptyset = C_i\cap F$ and $V=F\cup
C_1\cup C_2 \cup C_3.$
\end{proposition}
\begin{proof}[Proof (see
  {\cite[Theorem~1]{ArchangelskiiToAlStExDiSeGr}})] Call an open
subset $A\subseteq V$ \textit{$\homeo$-simple} if $\homeo(A)\cap
A=\emptyset$.  By the Hausdorff maximality theorem, there exists a
maximal chain $\G$ of $\homeo$-simple sets.  Put $U=\bigcup \G$.  Then
$U$ is also a $\homeo$-simple subset of $V$, and since $\overline{U}$
is open, maximality shows that $U$ is in fact clopen.

Next observe that $\homeo(U)\cap V$ and $\homeo^{-1}(U\cap W)$ are clopen
$\homeo$-simple sets, and put 
$$M= U\cup (\homeo(U)\cap V)\cup \homeo^{-1}(U\cap W).$$  Since the
intersection of $F$ with any $\homeo$-simple subset of $V$ is empty, we
have $M\cap F=\emptyset$.  We shall show that $F=V\setminus M$.   

Suppose to the contrary, that $x\in V\setminus M$ satisfies
$\homeo(x)\neq x$.  Let $H$ be an open subset of $V$ such that $x\in H$
and $H\cap M$ and $\homeo(H)\cap H$ are both empty.  Then $H$ is
$\homeo$-simple and 
\begin{equation}\label{intersections}
H\cap U= H\cap (\homeo(U)\cap V)= H\cap \homeo^{-1}(U\cap
W)=\emptyset.\end{equation}  But \eqref{intersections} implies that
$H\cup U$ is a $\homeo$-simple set which properly contains $U$,
contradicting the maximality of $U$.  So $F=V\setminus M$.  

Since both $V$ and $M$ are clopen, so is $F$.  Finally, to complete
the proof, take $C_1:=U$, $C_2:=\homeo(U)\cap V,$ and
$C_3:=\homeo^{-1}(U\cap W)\setminus (\homeo(U)\cap V)$.
\end{proof}

With this preparation, we now show that any action of a $*$-semigroup on a
Stonean space is quasi-free.

\begin{theorem}\label{quasiStonean} Suppose that $X$ is a Stonean
  space, $\S$ is a $*$-semigroup, and $\alpha:\S\rightarrow \oinv(X)$
  is a $*$-semigroup homomorphism.  Then $\S$ acts quasi-freely on
  $X$.
\end{theorem}
\begin{proof}
  Fix $s\in \S$, and consider the open sets $G:= \text{dom}(\alpha(s))$ and
  $H:=\text{ran}(\alpha(s))$.  Since $X$ is compact and extremally
  disconnected,  the
  Stone-\v{C}ech compactifications $\beta G$ and $\beta H$ are
  homeomorphic to $\overline{G}$ and $\overline{H}$ respectively
  (\cite[Exercises~15G(1) and~19G(2)]{WilliardGeTo}).  Since
  $\alpha(s)$ is a homeomorphism of $G$ onto $H$, general properties
  of the Stone-\v{C}ech compactification show that $\alpha(s)$ extends
  to a homeomorphism $h$ of $\overline{G}$ onto $\overline{H}.$

Let $F\subseteq \overline{G}$ be the set of fixed points for $h$;
Proposition~\ref{fixed} shows that $F$ is clopen in $X$.  Therefore,
$$\{x\in\text{dom}(\alpha(s)): \alpha(s)(x)=x\}= 
F\cap \text{dom}(\alpha(s))$$ is open in $X$.  Thus $\S$ acts
quasi-freely on $X$.
\end{proof}

Quasi-freeness is intimately related to the extension property.
 The following result is known in the case when $(\C,\D)$ is a regular
 MASA inclusion with faithful expectation,
 see~\cite[Proposition~5.11]{RenaultCaSuC*Al}.  Parts of the proof
 below follow the
proof of~\cite[Proposition~3.12]{DonsigPittsCoSyBoIs}, but we
reproduce it here for convenience of the reader. 
\begin{theorem}\label{freeaction}
Let $(\C,\D)$ be a regular inclusion. Then  the following statements are equivalent:
\begin{enumerate}
\item[a)] $\D$ has the extension
property in $\C$;
 \item[b)] $\D$ is a MASA in $\C$ and the action $v\mapsto \beta_v$ of
   the semigroup $\N(\C,\D)$ is a  quasi-free action on $\hat{\D}$;
\item[c)] $\D$ is a MASA in $\C$ and for each $\sigma\in\hat{\D}$, the
  isotropy group of $\W\G(\C,\D)$ at $\sigma$ is the trivial group.
\end{enumerate}
\end{theorem}
\begin{proof}
(a)$\Rightarrow$(b).
Suppose that $\D$ has the extension property.
Then~\cite[Corollary~2.7]{ArchboldBunceGregsonExStC*AlII} shows that
$\D$ is a MASA in $\C$, and that there exists a conditional
expectation $E:\C\rightarrow \D$.   Suppose that
$v\in\N(\C,\D)$, and $\sigma\in\hat{\D}$ satisfies $\sigma(v^*v)>0$ and
$\beta_v(\sigma)=\sigma$.  
By \cite[Proposition~3.12]{DonsigPittsCoSyBoIs}, we have
$v^*E(v)\in\D$.
Also, if $G$ is the unitary group of $\D$, we have for
$g\in G$,
$$\sigma(v^*gvg^{-1})=\sigma(v^*gv)\sigma(g^{-1})=\beta_v(\sigma)(g)
\sigma(v^*v)\sigma(g^{-1})=\sigma(g)\sigma(v^*v)\sigma(g^{-1})=\sigma(v^*v).$$
The extension property
and~\cite[Theorem~3.7]{ArchboldBunceGregsonExStC*AlII} show that
$E(v)\in\overline{co}\{gvg^{-1}: g\in G\}$, so that
$$\sigma(v^*E(v))=\sigma(v^*v),$$ whence $\sigma(v^*E(v))=\sigma(v^*v)\neq
0$.

Hence there exists an open set $U\subseteq \hat{\D}$ so that
$\sigma\in U$ and 
$\tau(v^*E(v))\neq 0$ for every $\tau\in U$.  Since $v^*E(v)\in \D$,
we have
$\tau=\beta_{v^*E(v)}(\tau)=\beta_{v^*}(\beta_{E(v)}(\tau))=\beta_{v^*}(\tau)$
for every $\tau\in U$.  But $\beta_v^{-1}=\beta_{v^*}$, so
$\beta_v(\tau)=\tau$ for $\tau\in U$.  Thus $\{\sigma\in\hat{\D}:
\sigma(v^*v)>0\text{ and } \beta_v(\sigma)=\sigma\}$ is open in
$\hat{\D}$, so the semigroup $\{\beta_v:
v\in\N(\C,\D)\}$ acts quasi-freely on $\hat{\D}$.
   
Before proving (b)$\Rightarrow$(a),  we establish some notation.  We use  $[\C,\D]$ for the
set $\{cd-dc: d\in\D,
c\in\C\}$.   Also, a normalizer $v\in\N(\C,\D)$ is a \textit{free
  normalizer} if $v^2=0$.   Kumjian notes
in~\cite{KumjianOnC*Di}, that any free normalizer belongs to
$\overline{\spn}[\C,\D]$, because $v(v^*v)^{1/n}-(v^*v)^{1/n}v=
v(v^*v)^{1/n}\rightarrow v$.

We turn now to the proof of (b)$\Rightarrow$(a).  So  assume that (b)
holds.  We shall prove that
\begin{equation}\label{norminDcom}
\N(\C,\D)\subseteq   \D+\overline{\spn}[\C,\D].
\end{equation}
Once this inclusion is established, an
application of 
\cite[Theorem~2.4]{ArchboldBunceGregsonExStC*AlII} will show that
whenever $\rho_1$ and $\rho_2$ are states of $\C$ with
$\rho_1|_\D=\rho_2|_\D\in\hat{\D}$, then for every $v\in\N(\C,\D)$, 
 $\rho_1(v)=\rho_2(v)$.
Regularity of $(\C,\D)$ then implies that $\rho_1=\rho_2$, so that
 $\D$ has the extension property.

 To show~\eqref{norminDcom}, fix $v\in \N(\C,\D)$.
Let $F=\{\sigma\in\hat{\D}: \sigma(v^*v)>0\text{ and }
\beta_v(\sigma)=\sigma\}$.  By hypothesis, $F$ is an open subset of
$\hat{\D}$.  Let $\eps >0$, and let  $$X_{\eps} :=F\cap
\{\sigma\in\hat{\D}: \sigma(v^*v)\geq\eps^2\}.$$  Then $X_\eps$ is a
closed subset of $\hat{\D}$.  Let $f_\eps\in\D$ be such that $0\leq
f_\eps \leq I$, $\hat{f}_\eps|_{X_\eps}\equiv 1$ and 
$\overline{\supp}(\hat{f}_\eps)\subseteq F$.  Next let    
$$Y_\eps:=F^c\cap\{\sigma\in\hat{\D}: \sigma(v^*v)\geq \eps^2\}.$$
Clearly $Y_\eps\cap \overline{\supp}(\hat{f}_\eps)=\emptyset.$

For
$\sigma\in Y_\eps$, we have $\beta_v(\sigma)\neq \sigma$, so 
we may find $d\in \D$ with $\sigma(d)=1$, $0\leq
d\leq 1$, $\overline{\supp}(\hat{d} )\cap\overline{\supp}(\hat{f}_\eps)=\emptyset$  and $(vd)^2=0$.  
Compactness of $Y_\eps$ and a partition of unity argument show that
there exists $n\in\bbN$ and a collection of functions
$\{g_j\}_{j=1}^n\subseteq 
\D$ such that, with $g_\eps=\sum_{j=1}^n g_j$, we have: 
$$(vg_j)^2=0,\quad 0\leq g_j\leq I,\quad\overline{\supp}(\hat{g}_j)\cap
 \overline{\supp}(\hat{f}_\eps)=\emptyset,\quad 0\leq
g_\eps\leq I, \dstext{and}\hat{g}_\eps|_{Y_\eps}\equiv 1.$$

Then $g_\eps f_\eps=0$, and $\norm{v(I-(f_\eps +g_\eps))}<\eps$.  So
$v=\lim_{\eps\rightarrow 0}(vf_\eps +vg_\eps)$.  From Kumjian's observation, $vg_\eps\in
\overline{\spn}[\C,\D]$.  Moreover, since the closed support of
$\hat{f}_\eps$ is contained in $F$, we find that $vf_\eps$ commutes with
$\D$.  Hence $vf_\eps\in\D$, as $\D$ is a MASA in $\C$.

Let $\eps_n$ be a sequence of positive numbers decreasing to $0$.  We
may choose elements $f_{\eps_n}\in\D$ as above but with the additional
condition that $f_{\eps_m}\leq
f_{\eps_n}$ whenever $m <n$.
For $n>m$ and $\sigma\in\hat{\D}$, we
have \begin{equation}\label{excas}
\sigma(v^*v(f_{\eps_n}-f_{\eps_m})^2)=\begin{cases}
  0&\text{if $\sigma(f_{\eps_n}-f_{\eps_m})=0$}\\
\sigma(v^*v)\sigma(f_{\eps_n}-f_{\eps_m})^2 & \text{if $
  \sigma(f_{\eps_n}-f_{\eps_m})\neq 0$}.
\end{cases}\end{equation}
Notice that if $\sigma\in \overline{F}\cap F^c$, then $\sigma(v^*v)=0$.  By
continuity of $\widehat{v^*v}$, 
given $\delta >0$, there exists a open set $U\subseteq \hat{\D}$ with
$U\supseteq \overline{F}\cap F^c$ and $\sigma(v^*v)<\delta$ if
$\sigma\in U$.  Since $\overline{\supp}(f_{\eps_n}-f_{\eps_m})\subseteq U$ when
$n,m$ are sufficiently large, it follows from~\eqref{excas} that $vf_{\eps_n}$ is a Cauchy
sequence in $\D$, and hence converges to $k\in \D$.  Then
$$v-k=\lim_{n\rightarrow\infty}(vg_{\eps_n})\in \overline{\spn}[\C,\D],$$ whence
$v=k+(v-k)\in \D+\overline{\spn}[\C,\D]$.  As noted above, this is
sufficient to complete the proof that $\D$ has the extension
property.  Thus (a) holds.

(b)$\Rightarrow$(c).  Fix $\sigma\in\hat{\D}$, and suppose that
$[\sigma,\phi,\sigma]$ belongs to the isotropy group of $\W\G(\C,\D)$
at $\sigma$.   Then there is an open set $N$ and $v\in\N(\C,\D)$ such
that with $\sigma\in
N\subseteq\dom\phi\cap\dom\beta_v $ such that $\beta_v|_N=\phi|_N$.
By quasi-freeness of the action, there exists a neighborhood $N_1$ of
$\sigma$ contained in $N$ such that $\beta_v|_{N_1}=\text{id}|_{N_1}$.
Thus, $[\sigma,\phi,\sigma]=[\sigma,\text{id},\sigma]$ so that the
isotropy group of $\W\G(\C,\D)$ at $\sigma$ is trivial.

(c)$\Rightarrow$(b).  Let $v\in \N(\C,\D)$ and suppose that
$\sigma\in\dom(\beta_v)$ is a fixed point for $\beta_v$.  Then
$[\sigma,\beta_v,\sigma]$ is in the isotropy group for $\W\G(\C,\D)$
at $\sigma$, so that
$[\sigma,\beta_v,\sigma]=[\sigma,\text{id},\sigma]$.  By the
definition of the Weyl groupoid,  $\sigma$ belongs to the interior of
$\{x\in\dom(\beta_v): \beta_v(x)=x\}$.  Hence $\{x\in\dom(\beta_v):
\beta_v(x)=x\}$ is an open set in $X$.  As this holds for every
$v\in\N(\C,\D)$, the action $v\mapsto \beta_v$ is quasi-free.
\end{proof}

As an immediate corollary of our work, we have the following.

\begin{theorem}\label{inMASAincl}  
Suppose $(\C,\D)$ is a regular MASA inclusion, with
  $\D$  an injective \cstar-algebra.  Then $(\C,\D)$ has the
  extension property.
\end{theorem}

\begin{proof} Results of Dixmier and
Gonshor~\cite{DixmierCeEsCoSt,GonshorInHuC*AlII} show that $\D$ is
injective if and only if $\hat{\D}$ is a compact extremally
disconnected space.
Now combine Theorems~\ref{quasiStonean} with the equivalence of (a) and
(b) in Theorem~\ref{freeaction}.
\end{proof}
 
\begin{remark}{Example}  Suppose that $\M$ is a von Neumann algebra
  and $\D\subseteq \M$ is a MASA.  Let $\C$ be the norm-closure of
  $\N(\M,\D)$.  Then $(\C,\D)$ has the extension property.   Note that
  in particular, when $\D$ is a Cartan MASA in $\M$ in the sense of
  Feldman and Moore~\cite{FeldmanMooreErEqReII}, then $(\C,\D)$ is a \cstar-diagonal.
\end{remark}
\begin{remark}{Remark}
The regularity hypothesis in Theorem~\ref{inMASAincl} cannot be
removed.  Indeed, \cite[Corollary~4.7]{AkemannShermanCoExOnMASAs} 
shows that when
$\C$ is the hyperfinite $II_1$ factor, and $\D\subseteq \C$ is any
MASA, then $(\C,\D)$ fails to have the extension property.
\end{remark}

\subsection{$\D$-modular states and ideals of $(\C,\D)$}

In this final subsection, we make a simple observation:  the action
$\N(\C,\D)\ni v\mapsto \beta_v\in \oinv{\hat{\D}}$ may be regarded as
an action on certain states of $\C$.

\begin{definition}\label{Dmodstate}  Let $(\C,\D)$ be an inclusion.  
A state $\rho$ on $\C$ is \textit{$\D$-modular} if
  for every $x\in \C$ and $d\in \D$, 
$$\rho(dx)=\rho(d)\rho(x)=\rho(xd).$$  We let $\Mod(\C,\D)$ be the
collection of all $\D$-modular states on $\C$;  equip $\Mod(\C,\D)$ with
the relative weak-$*$ topology.
Then $\Mod(\C,\D)$ is closed and hence is 
compact.  Using the Cauchy-Schwartz inequality, it is easy to see that
$$\Mod(\C,\D)=\{\rho\in\S(\C): \rho|_\D\in \hat{\D}\}.$$  
For $\sigma\in\hat{\D}$, let $\Mod(\C,\D,\sigma)$ be the set of all
state extensions of $\sigma$, so
$$\Mod(\C,\D,\sigma):=\{\rho\in \S(\C): \rho|_\D=\sigma\}.$$
\end{definition}

The following simple observation will be useful during the sequel.
\begin{lemma}\label{vactmod}  Let $(\C,\D)$ be an inclusion,
 and $v\in \N(\C,\D)$.  If $\sigma\in \dom(\beta_v)$ and
 $\beta_v(\sigma)\neq \sigma$, then $\rho(v)=0$ for every
 $\rho\in\Mod(\C,\D,\sigma)$. 
\end{lemma}
\begin{proof}
Let $d\in\D$ satisfy $\beta_v(\sigma)(d)=0$ and $\sigma(d)=1$.  Then
for $\rho\in\Mod(\C,\D,\sigma)$, we have
$$\rho(v)=\frac{\rho(vv^*v)}{\rho(v^*v)}=
\frac{\rho(dvv^*v)}{\sigma(v^*v)} =
\frac{\rho(vv^*dv)}{\sigma(v^*v)}=\rho(v)\beta_v(\sigma)(d) =0.$$
\end{proof}

When $\rho\in\Mod(\C,\D)$ and $v\in \N(\C,\D)$ satisfies
$\rho(v^*v)\neq 0$, the state $\tilde{\beta}_v(\rho)$ on $\C$ given by
$$\tilde{\beta}_v(\rho)(x):=\frac{\rho(v^*xv)}{\rho(v^*v)}$$ again belongs to
$\Mod(\C,\D)$.   When there is no danger of confusion, we sometimes
simplify notation and  
write $\beta_v(\rho)$ instead of $\tilde{\beta_v}(\rho)$.  Thus $\N(\C,\D)$ also
acts on $\Mod(\C,\D)$, and  for every $\rho\in \Mod(\C,\D)$, we have
$\tilde{\beta}_v(\rho)|_{\D}=\beta_v(\rho|_\D)$.

\begin{definition}  A subset $F\subseteq \Mod(\C,\D)$ is 
  \textit{$\N(\C,\D)$-invariant} if for every $v\in \N(\C,\D)$ and
  $\rho\in F$ with $\rho(v^*v)\neq 0$, we have $\tilde{\beta}_v(\rho)\in F$. 
\end{definition}

We record the following fact for use in the sequel.
\begin{proposition} \label{invideal}  Let $(\C,\D)$ be a regular
  inclusion and suppose that $F\subseteq \Mod(\C,\D)$ is 
  $\N(\C,\D)$-invariant.  Then the set 
$$\K_F:=\{x\in\C: \rho(x^*x)=0\text{ for all } \rho\in F\}$$ is a
  closed, two-sided ideal in $\C$.  Moreover, if 
  $\{\rho|_\D: \rho\in F\}$ is weak-$*$ dense in $\hat{\D}$, 
then $\K_F\cap \D=(0).$   
\end{proposition}
\begin{proof}
As $\K_F$ is the intersection of closed left-ideals, it remains only
to prove that $\K_F$ is a right ideal.  By regularity, it suffices to
prove that if $x\in\K_F$ and $v\in \N(\C,\D)$, then $xv\in \K_F$.
Let $\rho\in F$.  If $\rho(v^*v)\neq 0$, then by hypothesis, we obtain
$\rho(v^*x^*xv)=\beta_v(\rho)(x^*x)\rho(v^*v)=0$.  On the other hand, if
$\rho(v^*v)=0$, then $\rho(v^*x^*xv)\leq \norm{x^*x}\rho(v^*v)=0$.  In
either case, we find $\rho(v^*x^*xv)=0$.  As this holds for every
$\rho\in\F$, we find $xv\in\K_F$, as desired.   The final statement is
obvious. 
\end{proof}

\section{Pseudo-Conditional Expectations for
  Regular MASA
  Inclusions} \label{MainExpect}

As noted in Example~\ref{noce}, there exist regular MASA inclusions
$(\C,\D)$ for which no conditional
expectation of $\C$ onto $\D$ exists.  The purpose of this section is
to show that nevertheless, there is always a unique pseudo-expectation
for a regular MASA inclusion.  

Given a normalizer $v$, our first task is to connect the dynamics of
$\beta_v$ with the ideal structure of $\D$.

\begin{lemma}\label{normintbetav}  Let $(\C,\D)$ be an
  inclusion and suppose $v\in\N(\C,\D)$.  If $d\in\D$ and
  $\supp(\hat{d})\subseteq (\fix{\beta_v})^\circ$, then $vd=dv$.
Moreover, if $(\C,\D)$ is a MASA inclusion, then $vd=dv\in\D$.
\end{lemma}
\begin{proof}
Note that $v^*dv$ and $v^*vd$ both belong to $\D$.
We first show that
for every $\rho\in\hat{\D}$, 
\begin{equation}\label{acom}
\rho(v^*dv)=\rho(v^*vd).
\end{equation}
Let
$\rho\in\hat{\D}$.  There are three cases.  First suppose  $\rho(v^*v)=0$. Then as
$\rho(v^*v)=\norm{v}_\rho^2$, the Cauchy-Schwartz
inequality gives,
$$|\rho(v^*dv)|=|\innerprod{dv,v}_\rho|\leq
\norm{dv}_\rho\norm{v}_\rho=0,$$ so \eqref{acom} holds when
$\rho(v^*v)=0$.  

Suppose next that $\rho(v^*v)>0$ and $\beta_v(\rho)(d)\neq 0$.
Then $\beta_v(\rho)\in \supp(\hat{d})$, so $\beta_v(\rho)\in
\fix{\beta_v}=\fix{\beta_{v^*}}$.  Thus, we get $\beta_v(\rho)=
\beta_{v^*}(\beta_v(\rho))=\rho$,
and hence $\rho(v^*dv)=\rho(v^*v)\rho(d)=\rho(v^*vd)$.

Finally suppose that $\rho(v^*v)>0$ and $\beta_v(\rho)(d)=0$.  Then
$\rho(v^*dv)=0$.  We shall show that $\rho(d)=0$.  If not,  the hypothesis
on $d$ shows that $\rho\in\fix{\beta_v}$.  Hence, 
$0\neq\rho(d)=\beta_v(\rho)(d)=\frac{\rho(v^*dv)}{\rho(v^*v)}=0$,
which is absurd.  So $\rho(d)=0$, and~\eqref{acom} holds in
this case also.    Thus we have established \eqref{acom} in all cases.

Thus $v^*dv=v^*vd$.   So for every $n\in \bbN$, 
$$0=v^*dv-v^*vd=v^*(dv-vd)=vv^*(dv-vd)=(vv^*)^n(dv-vd).$$  It follows
that for every polynomial $p$ with $p(0)=0$, we have,
$p(vv^*)(dv-vd)=0$.  Therefore, for every $n\in\bbN$,
$$0=(vv^*)^{1/n}(dv-vd)=d(vv^*)^{1/n}v-(vv^*)^{1/n}vd.$$  Since
$\lim_{n\rightarrow \infty} (vv^*)^{1/n}v=v$, we have $vd=dv$.

Now suppose that $(\C,\D)$ is a MASA inclusion.  For $a\in\D$, we have
$\supp(\widehat{da})\subseteq \supp(\hat{d})\subseteq
(\fix{\beta_v})^\circ,$ so we have $v(da)=(da)v$.  Since $vd=dv$, we
get $(vd)a=a(vd)$.  Since $\D$ is a
MASA, $vd\in\D$ and the proof is complete.
\end{proof}

Let $v\in\N(\C,\D)$.  Observe that if $d\in\D$ and
$\supp(\hat{d})\subseteq \fix{\beta_v}$, then we actually have
$\supp(\hat{d})\subseteq (\fix{\beta_v})^\circ.$ Thus
$$\{d\in\D: \supp(\hat{d})\subseteq \fix{\beta_v}\}=\{d\in\D:
\supp(\hat{d})\subseteq (\fix{\beta_v})^\circ\}$$ is a closed ideal of
$\D$ isomorphic to $C_0((\fix{\beta_v})^\circ)$.  
By the Fuglede-Putnam-Rosenblum commutation theorem,
\begin{equation}\label{Fuglede}
\{d\in\D: dv=vd\}=\{d\in\D: dv^*=v^*d\},
\end{equation} and it follows that  
 $\{d\in\D: dv=vd\in\D\}$ is a closed 
ideal of $\D$.  
The next proposition shows how the set
$(\fix{\beta_v})^\circ$ may be described algebraically. 

\begin{proposition}\label{Jvee}  Let $(\C,\D)$ be a  MASA
  inclusion.  If $v\in\N(\C,\D)$, then
$$\{d\in\D: \supp(\hat{d})\subseteq
  (\fix{\beta_v})^\circ\}=\overline{\D v^*v}\cap \{d\in\D:
dv=vd\in\D\}.$$
\end{proposition}
\begin{proof}
Notice that $\overline{\D v^*v}=\{d\in\D: \supp(\hat{d})\subseteq
\supp(\widehat{v^*v})\}.$ 
Since $(\fix{\beta_v})^\circ\subseteq \supp(\widehat{v^*v})$,
Lemma~\ref{normintbetav} shows
that 
$$\{d\in\D: \supp(\hat{d})\subseteq (\fix{\beta_v})^\circ\}\subseteq
\overline{\D v^*v}\cap \{d\in\D: dv=vd\in\D\}.$$

Now suppose that $d\in \overline{\D v^*v}\cap \{h\in\D: hv=vh\in\D\}$
and $d\neq 0$.  Let $\rho_0\in\supp(\hat{d})$ and set
$r:=|\rho_0(d)|$.  Then $r>0$.  Put $G=\{\rho\in\hat{\D}:
|\rho(d)|>r/2\}$.  We show that $G\subseteq \fix{\beta_v}$.  Fix
$\rho\in G$.  Since $d\in\overline{\D v^*v}$, we have
$\supp(\hat{d})\subseteq \supp((\widehat{v^*v})),$ so $\rho(v^*v)\neq
0$.  Since $d$ belongs to the ideal $\{f\in\D: fv=vf\in\D\}$, we find
(using~\eqref{Fuglede}) that for every $a\in\D$,
$$
\beta_v(\rho)(a) = \frac{\rho(v^*av)}{\rho(v^*v)}=
\frac{\rho(v^*avd)}{\rho(v^*v)\rho(d)}
=\frac{\rho(v^*(ad)v)}{\rho(v^*v)\rho(d)}
=\frac{\rho(adv^*v)}{\rho(v^*v)\rho(d)}=\rho(a).$$  It follows that
$G\subseteq \fix{\beta_v}$.   Since $G$ is an open subset of
$\hat{\D}$ with $\rho_0\in G$,  we have $\rho_0\in
(\fix{\beta_v})^\circ.$  So $\supp(\hat{d})\subseteq
(\fix{\beta_v})^\circ$, as desired.  

\end{proof}

We need some notation.  
\begin{remark*}{Notation}  Let $(\C,\D)$ be an inclusion.
\begin{enumerate}
\item[a)]  For a closed ideal $\fJ$ in $\D$,  we let  
   $$\fJ^\perp:=\{d\in\D: dg=0 \text{ for all } g\in \fJ\}$$ denote
   the complement of $\fJ$ in the lattice of closed ideals of $\D$.
 \item[b)] Given any two closed ideals $\fJ_1$ and $\fJ_2$ in $\D$,
   $\fJ_1\vee \fJ_2$ denotes the closed ideal generated by $\fJ_1$ and
   $\fJ_2$.
\item[c)] 
  For  $S\subseteq \D$,  $\innerprod{S}_\D$  denotes the closed
  two-sided ideal of $\D$ generated by the set $S$.  When the context
  is clear, we drop the subscript and simply use $\innerprod{S}$.   
\item[d)] For 
  $v\in\N(\C,\D)$, let 
\begin{equation}\label{jveedef}
J_v:=\{d\in\D: vd=dv\in\D\}\cap \innerprod{v^*v}\dstext{and}
K_v:=\innerprod{v^*v}^\perp \vee \innerprod{\{v^*hv-hv^*v: h\in
  \D\}}.\end{equation}
\end{enumerate}
\end{remark*}

Recall that an ideal $\fJ$ in a \cstaralg\ $\C$ is an
\textit{essential ideal} if $\fJ\cap L\neq (0)$ for every closed
two-sided ideal $(0)\neq L\subseteq \C$.  

We will show $J_v\vee K_v$ is an essential
ideal in $\D$.  It is easy to see that $J_v\subseteq K_v^\perp$.  If
equality holds, then the fact that $J_v\vee K_v$ is an essential ideal
follows readily.  However, we have not found a simple proof of this
fact, so we proceed along different lines.

\begin{proposition}\label{essentialv}  Let $(\C,\D)$ be a regular MASA
  inclusion and let $v\in\N(\C,\D)$.  Then $J_v\vee K_v$ is an
  essential ideal in $\D$.
\end{proposition}
\begin{proof}
We shall show that $J_v\vee K_v$ is essential by showing that $(J_v
\vee K_v)^\perp=(0).$  Suppose that $d\in\D$ and $d(J_v\vee K_v)=0$.  

First we show  
\begin{equation}\label{invv}
\supp{\hat{d}}\subseteq \overline{\supp{\widehat{v^*v}}}.
\end{equation}
Indeed, if $\rho\in\hat{\D}$ and $\rho\notin\overline{\supp{\widehat{v^*v}}}$,
then we may find $h\in\D$ such that $\hat{h}(\rho)=1$ and $\hat{h}(\sigma)=0$
for every $\sigma\in \overline{\supp{\widehat{v^*v}}}$.  Then
$h\in\innerprod{v^*v}^\perp\subseteq K_v$, so $dh=0$.  As $\rho(h)=1$, this shows
that $\rho\notin\supp{\hat{d}}$.  Thus~\eqref{invv} holds.

Next, we claim that 
\begin{equation}\label{infv}
\supp{\hat{d}}\cap \supp{\widehat{v^*v}}\subseteq
(\fix{\beta_v})^\circ.
\end{equation}
Let $\rho\in \supp{\hat{d}}\cap \supp{\widehat{v^*v}}.$  For
every $h\in\D$, we have $v^*hv-v^*vh\in K_v$.  Since $\rho(d)\neq 0$
and $d\in K_v^\perp$,
this gives $\rho(v^*hv)=\rho(v^*v)\rho(h)$ for every $h\in \D$.  Since
$\rho(v^*v)\neq 0$ by hypothesis, we have $\rho\in\fix{\beta_v}$.  As
$\supp{\hat{d}}\cap \supp{\widehat{v^*v}}$ is an open subset of
$\hat{\D}$,  we obtain~\eqref{infv}.

Suppose that $\rho\in (\fix{\beta_v})^\circ$.  By
Proposition~\ref{Jvee} there exists $h\in J_v$ so that $\rho(h)=1$.
Since $d\in J_v^\perp$, we obtain $\rho(d)=0$.  Hence
\begin{equation}\label{fvsn}
(\fix{\beta_v})^\circ\cap\supp{\hat{d}}=\emptyset.
\end{equation}

Combining~\eqref{invv},~\eqref{infv}, and~\eqref{fvsn} we obtain,
$$\supp{\hat{d}}\subseteq \overline{\supp{\widehat{v^*v}}}\setminus
\supp{\widehat{v^*v}}. $$  But $\overline{\supp{\widehat{v^*v}}}\setminus
\supp{\widehat{v^*v}}$ has empty interior, so
$\supp{\hat{d}}=\emptyset$.  Therefore,  $d=0$
as desired.
\end{proof}

The following is probably well known, but we include it for
completeness.  Recall that if $\A$ is a unital injective \cstaralg, then 
any bounded increasing net $x_\lambda$ of self-adjoint
elements of $\A$ has a least upper bound in $\A$, which we denote by
$\sup_\A x_\lambda.$ 

\begin{lemma}\label{eauid}  Let $\D$ be a unital, abelian \cstaralg,
  let $(I(\D),\iota)$ be an injective envelope for $\D$
  and suppose that $J\subseteq \D$ is a closed ideal.  Let $J_1^+$ be
  the positive part of the unit ball of $J$ and regard $J_1^+$ as a
  net indexed by itself.  Put $P:=\sup_{I(\D)}\iota(J_1^+)$.  Then
  $P$ is an  projection in $I(\D)$.  

If in addition, $J$ is an essential
  ideal of $\D$, the following hold:
\begin{enumerate}
\item[i)] if $a,b\in I(\D)$ and $a\iota(h)=b\iota(h)$ for every $h\in
  J$, then $a=b$; 
\item[ii)] $P=I$.
\end{enumerate}
\end{lemma}

\begin{proof}
The mapping $x\in J_1^+\mapsto x^{1/2}\in J_1^+$ is an order isomorphism of
$J_1^+$, so $P=\sup_{I(\D)} \{\iota(x^{1/2}): x\in J_1^+\}$.  So $P^2$ is
also an upper bound for $\iota(J_1^+).$  Hence $P\leq P^2$.  But as
$\norm{P}\leq 1$, we have $P^2\leq P$.  So $P$ is a projection.

Now assume $J$ is an essential ideal, and suppose $a,b\in I(\D)$
satisfy $(a-b)\iota(h)=0$ for every $h\in J$.  
By Hamana-regularity of $\iota(\D)$ in $I(\D)$, we have
$|a-b|=\sup_{I(\D)} \{d\in \D: 0\leq \iota(d)\leq |a-b|\}$.  But
$K:=\{d\in \D: \iota(d)\in \innerprod{|a-b|}_{I(\D)}\}$ is a closed ideal of $\D$ with
$K\subseteq J^\perp$.  Hence $K=0$, so  $|a-b|=0$.  

Notice that if $d\in J_1^+$, then $\iota(d)P=\iota(d)$.  It follows that for every
$d\in J$, $\iota(d)P=\iota(d)$.   By part~(i), $P=I$ when $J$ is an essential ideal.
 \end{proof}

The following extension of Definition~\ref{Dmodstate} will be useful.
\begin{definition}\label{Dmodlinmap}  
Let $(\C,\D)$ be an inclusion and $\B$ an algebra.
\begin{enumerate}
\item
A linear map $\Delta:\C\rightarrow \B$ is
  \textit{$\D$-modular} (or more simply \textit{modular})  if for every
  $x\in \C$ and $d\in \D$,
$$\Delta(xd)=\Delta(x)\Delta(d)\dstext{and}\Delta(dx)=\Delta(d)\Delta(x).$$
\item A homomorphism $\theta:\D\rightarrow \B$ is 
\textit{$\D$-thick in $\B$} if for every non-zero
element $b\in\B$, the ideal $\{d\in \D: b\theta(d)=0\}^\perp$ is a
non-zero ideal of $\D$.
\end{enumerate}
\end{definition}

When $\B$ is abelian, notice that the restriction of a $\D$-modular
map to $\D$ is a homomorphism.
The next lemma gives an example which
 will be  used in the proof of Theorem~\ref{uniquecpmap}.
\begin{lemma}\label{thickinjenv} Let $\D$ be an abelian \cstaralg\ and let
  $(I(\D),\iota)$ be an injective envelope for $\D$.  Then $\iota$ is
  $\D$-thick in $I(\D)$.  
\end{lemma}
\begin{proof} Suppose $b\in I(\D)$ is non-zero.
  The Hamana regularity of $I(\D)$ ensures that there exists a
  non-zero $h\in\D$ such that $0\leq \iota(h)\leq b^*b.$   If $d\in
  \D$ satisfies $\iota(d) b=0$, then $\supp(\widehat{\iota(d)})\cap
  \supp(\hat{b})=\emptyset$.  Since
    $\supp(\widehat{\iota(h)})\subseteq \supp(\hat{b})$, we get
    $\iota(dh)=0$, whence $h\in \{d\in \D: \iota(d)b=0\}^\perp$.  
\end{proof}

Our interest in $\D$-thick
homomorphisms will be with the restrictions of $\D$-modular maps to $\D$.
The following lemma will be useful.
\begin{lemma} \label{agreev} Let  $(\C,\D)$ be a regular MASA
  inclusion, let $\B$ be a unital abelian Banach algebra 
  and let $v\in\N(\C,\D)$.   
For $i=1,2$, suppose
  $\Delta_i:\C\rightarrow \B$ are bounded $\D$-modular maps  such
  that $\Delta_1|_\D=\Delta_2|_\D$ and set $\iota:=\Delta_i|_\D$.  
Then for every $h\in J_v\vee K_v$,
\begin{equation}\label{hJvKv}
 \Delta_1(vh)=\Delta_2(vh).
\end{equation} 
In fact,  
\begin{enumerate}
\item[a)] for every $h\in K_v$, $\Delta_1(vh)=0=\Delta_2(vh)$;  
\item[b)] for every $h\in J_v$,
$\Delta_1(vh)=\iota(vh)=\Delta_2(vh).$
\end{enumerate}
   Moreover, if $\iota$ is also $\D$-thick
in $\B$,
then $\Delta_1=\Delta_2$.
\end{lemma}
\begin{proof}
  For (a),  we consider two cases.
First,
if $h\in \innerprod{v^*v}^\perp$, then for $i=1,2$, 
$$\Delta_i(vh)=\lim_{n\rightarrow\infty} \Delta_i(v(v^*v)^{1/n}h)=0.$$
Second, suppose that $h\in\{v^*dv-v^*vd:d\in\D\}$.  
Then for some $d\in \D$,  
\begin{align*}
\Delta_i(vh)&=\Delta_i(v(v^*dv-v^*vd))=\Delta_i(vv^*dv)-\Delta_i(vv^*vd)\\
&=\iota(dv^*v)\Delta_i(v)-\Delta_i(vv^*v)\iota(d) = \iota(d)\Delta_i(v)\iota(v^*v)-\iota(d)\Delta_i(vv^*v)=0.
\end{align*}
Thus $\Delta_i(vh)=0$ for all $h$ in a generating set for $K_v$.
Since $\Delta_i$ are $\D$-modular and bounded, we obtain  (a).

Next, suppose that $h\in J_v$.  Then since $vh\in\D$,
\begin{equation}\label{hinJvee} 
\Delta_1(vh)=\iota(vh)=\Delta_2(vh).
\end{equation}  
This gives (b).

Parts (a) and (b) imply that $\Delta_1(vh)=\Delta_2(vh)$ for all $h$ in
a generating set for $J_v\vee K_v$, so boundedness and $\D$-modularity
of $\Delta_i$ yields~\eqref{hJvKv}.

Now suppose that $\iota$ is $\D$-thick in $\B$ and $v\in \N(\C,\D)$.
Let $b=\Delta_1(v)-\Delta_2(v)$ and choose any $h\in \{d\in \D:
b\iota(d)=0\}^\perp\cap (J_v\vee K_v).$ Since $h\in J_v\vee K_v$, we
have
$b\iota(h)=\Delta_1(v)\iota(h)-\Delta_2(v)\iota(h)=\Delta_1(vh)-\Delta_2(vh)=0$.
Since $h\in \{d\in \D: b\iota(d)=0\}^\perp$, we get $h^2=0$.  As $\D$
is abelian, $h=0$.  This shows that $\{d\in \D:
b\iota(d)=0\}^\perp\cap (J_v\vee K_v)=(0).$ Since $J_v\vee K_v$ is an
essential ideal, $\{d\in \D: b\iota(d)=0\}^\perp=(0)$.  Since $\iota$
is $\D$-thick in $\B$, we see that $b=0$. Hence
$\Delta_1(v)=\Delta_2(v)$.

Since this holds for every
$v\in \N(\C,\D)$, 
regularity of $(\C,\D)$ yields $\Delta_1=\Delta_2$.

\end{proof}

Lemma~\ref{agreev} has an interesting consequence for uniqueness of
extensions
of pure states on $\D$ to $\C$,  which we now present.  This result 
 generalizes  a result found in
~\cite{RenaultCaSuC*Al},   however, the proof is rather different.
Notice that Theorem~\ref{denseuep}
holds when $\C$ is separable or when there is a countable subset
$X\subseteq \N(\C,\D)$ such that $\C$ is the \cstaralg\ generated by
$\D$ and $X$.  We shall use Theorem~\ref{denseuep} in the proof of
Theorem~\ref{norming}.

\begin{theorem}\label{denseuep}  Suppose $(\C,\D)$ is a regular MASA
  inclusion and that $N\subseteq \N(\C,\D)$ is a countable set such
  that the norm-closed $\D$-bimodule generated by $N$ is $\C$.  Let 
$$\unistex:=\{\sigma\in\hat{\D}: \sigma \text{ has a unique state extension
  to }\C\}.$$  Then $\unistex$ is dense in $\hat{\D}$.
\end{theorem}
\begin{proof}
  For each $v\in N$, let $G_v:=\{\sigma\in\hat{\D}: \sigma|_{J_v\vee
    K_v}\neq 0\}.$ Clearly $G_v$ is open in $\hat{\D}$ and since
  $J_v\vee K_v$ is an essential ideal in $\D$, $G_v$ is dense in $\hat{\D}$.
  Baire's theorem shows that
$$P:=\bigcap_{v\in N} G_v$$ is dense in $\hat{\D}$.

Let $\sigma\in P$ and suppose for $i=1,2$, $\rho_i$ are states on $\C$
such that $\rho_i|_\D=\sigma$.  The Cauchy-Schwartz inequality shows
that $\rho_i:\C\rightarrow \bbC$ are $\D$-modular maps.  

Fix $v\in N$.  Since $\sigma\in G_v$, we may find $h\in J_v\vee K_v$
such that $\sigma(h)=1$.  By
Lemma~\ref{agreev} we have
$$\rho_1(v)=\rho_1(v)\sigma(h)=\rho_1(vh)=\rho_2(vh)=\rho_2(v)\sigma(h)=\rho_2(v).$$
Since $N$ generates $\C$ as a $\D$-bimodule and $\rho_i$ are
$\D$-modular, we see that $\rho_1=\rho_2$.  Hence $P\subseteq \unistex$, and
the proof is complete.
\end{proof}

We now show that any regular MASA inclusion has a unique completely
positive mapping $E:\C\rightarrow I(\D)$ which extends the inclusion
mapping of $\D$ into $I(\D)$.  

\begin{lemma}\label{bimodmap} Let $(\C,\D)$ be an inclusion and
  $(I(\D),\iota)$ an injective envelope of $\D$.  Let $E:\C\rightarrow
  I(\D)$ be a pseudo-expectation for $\iota$.  Then $E$ is $\D$-modular.
\end{lemma}
\begin{proof}
  Let $\rho\in\widehat{I(\D)}$ and put $\sigma=\rho\circ E$.  Then
  $\sigma|_\D\in\hat{\D}$, so the Cauchy-Schwartz inequality implies
  that for every $x\in\C$ and $d\in \D$,
  $\sigma(xd)=\sigma(x)\sigma(d).$
  Hence, $$\rho(E(xd))=\sigma(xd)=\sigma(x)\sigma(d)
  =\rho(E(x))\rho(E(d))=\rho(E(x)E(d)).$$ As this holds for
  every $\rho\in \widehat{I(\D)}$, we obtain $E(xd)=E(x)E(d)$.
  The proof that $E(dx)=E(d)E(x)$ is similar.
\end{proof}

\begin{theorem}\label{uniquecpmap}  Let $(\C,\D)$ be a regular MASA
  inclusion, and let $(I(\D),\iota)$ be an 
  injective envelope of $\D$.  Then there exists a
  unique pseudo-expectation $E:\C\rightarrow I(\D)$ for $\iota$.
  Furthermore, suppose
  $v\in \N(\C,\D)$.  Then
\begin{enumerate}
\item[a)] $E(vh)=\iota(vh)$ for every $h\in J_v$;
\item[b)] $E(vh)=0$ for every $h\in K_v$;
\item[c)] $|E(v)|^2=\iota(v^*v)P$, where $P:=\sup_{I(\D)}(\iota((J_v)_1^+))$. 
\end{enumerate}
\end{theorem}
\begin{proof}
The injectivity of $I(\D)$ guarantees the existence of
 a completely positive unital map $E:\C\rightarrow I(\D)$   
such that for every $d\in\D$,
\begin{equation}\label{useinj}
E(d)=\iota(d).
\end{equation}
Lemma~\ref{agreev} and
Lemma~\ref{thickinjenv} imply that $E$ is unique.

Now suppose that $v\in\N(\C,\D)$.   Since $E|_\D=\iota$, parts
 (a) and (b)  follow from
 Lemma~\ref{agreev}. 

We turn now to (c).
Observe that $K_v\subseteq J_v^\perp$.   Indeed
$J_v\subseteq \innerprod{v^*v}$, so if $h\in \innerprod{v^*v}^\perp$
then $hJ_v=0$.  On the other hand, if $d\in D$ and $h=v^*dv-v^*vd$,
then for $g\in J_v$, we have $hg=(v^*dv-v^*vd)g=v^*dgv-v^*vdg=0$
because $dg\in J_v$.  Thus $h J_v=0$ for every $h$ in a generating set
for $K_v$, so $K_v\subseteq J_v^\perp$.

Now let $h\in K_v$.
By~\cite[Corollary~4.10]{HamanaReEmCStAlMoCoCStAl}, we have 
$$\iota(h)^*P\iota(h)=\iota(h)^*\left[\sup_{I(\D)}\iota((J_v)_1^+)\right] \iota(h) 
=\sup_{I(\D)}\iota(h^*(J_v)_1^+ h)=0.$$  Since $P$ is a projection, 
it follows that $\iota(h)P=0$.  Hence $\iota(v^*v)P\iota(h^*h)=0$. 
Next, by part (b), we have for every $h\in K_v$,
$|E(v)|^2\iota(h^*h)=\iota(h^*)E(v)^*E(v)\iota(h)=\iota(h^*)E(v)^*E(vh)=0$.
Therefore, for $h\in K_v$,
\begin{equation}\label{KvEh}
|E(v)|^2\iota(h^*h)=\iota(v^*v)P\iota(h^*h)=0.
\end{equation}

Since 
$E(vh)=\iota(vh)$ for every $h\in J_v$, 
we see that for $h\in J_v$,
\begin{equation}\label{JvEh}
E(v^*)E(v)\iota(h^*h)=E((vh)^*)E(vh)=\iota(h^*v^*)\iota(vh) = 
\iota(v^*v)\iota(h^*h)= \iota(v^*v)P\iota(h^*h).
\end{equation}

Combining \eqref{KvEh} and \eqref{JvEh}, we see that for every $h\in
J_v\vee K_v$,
$$E(v)^*E(v)\iota(h^*h)=\iota(v^*v)P\iota(h^*h).$$  
Then $E(v)^*E(v)=\iota(v^*v)P$ by Lemma~\ref{eauid}, so we have (c).
\end{proof}

The following ``dual''  to
Theorem~\ref{uniquecpmap} is now easily established.  Notice that in
the context of Theorem~\ref{uniquecpmap}, when
$\rho\in\widehat{I(\D)}$, $\dual{E}(\rho)=\rho\circ E\in\Mod(\C,\D)$.

\begin{theorem}\label{uniqlift} Let $(\C,\D)$ be a regular MASA
  inclusion, let $(I(\D),\iota)$ be an injective envelope for
  $\D$, and let $E$ be the pseudo-expectation for $\iota$.   
The map $\dual{E}:\widehat{I(\D)}\rightarrow \Mod(\C,\D)$ is
  the unique continuous map of 
  $\widehat{I(\D)}$ into $\Mod(\C,\D)$ such that for every
  $\rho\in \widehat{I(\D)}$, $\dual{E}(\rho)|_\D=\rho\circ\iota$.
\end{theorem}
\begin{proof} Clearly $\dual{E}$ has the stated property, so we need
  only prove uniqueness.

Suppose that $\ell$ is a continuous map of
$\widehat{I(\D)}$ into $\Mod(\C,\D)$ such that for every $\rho\in
\widehat{I(\D)}$, $\ell(\rho)|_\D=\rho\circ\iota$.
For $x\in\C$, define a function $\phi_x:\widehat{I(\D)}\rightarrow \bbC$ by
$\phi_x(\rho)=\ell(\rho)(x)$.  Since $\ell$ is continuous,
$\phi_x$ is continuous.  Hence there exists a unique element
$E_1(x)\in I(\D)$ such that $\phi_x$ is the Gelfand transform of
$E_1(x)$.  Using the fact that 
$\Mod(\C,\D)\subseteq \S(\C)$, we find $E_1$ is linear, bounded,
unital and positive.  Since $I(\D)$
is abelian, $E_1$ is completely positive.  For $d\in\D$ we have
$\rho(E_1(d))=\ell(\rho)(d)=\rho(\iota(d))$.  Therefore,
$E_1|_\D=\iota$, so $E_1$ is a
pseudo-expectation for $\iota$.  By Theorem~\ref{uniquecpmap}, $E_1=E$, hence
$\ell=\dual{E}$. 

\end{proof}

\begin{definition}\label{defcst}  Let $(\C,\D)$ be a regular MASA inclusion,  let
  $(I(\D),\iota)$ be a \cstar-envelope for $\D$, and let $E$ be the
  (unique) pseudo-expectation for $\iota$.  Define
$$\fS_s(\C,\D):=\{\rho\circ E: \rho\in \widehat{I(\D)}\}.$$    We shall
call states belonging to $\fS_s(\C,\D)$ 
    \textit{strongly compatible} states.    Clearly,
    $\fS_s(\C,\D)$ is a closed subset of $\Mod(\C,\D)$.  Observe that
    $\hat{\D}=\{\tau|_\D: \tau\in \fS_s(\C,\D)\}$; this is because 
    $\hat{\D}=\{\rho\circ\iota:\rho\in\widehat{I(\D)}\}$. 
\end{definition}

Let $r:\Mod(\C,\D)\rightarrow \hat{\D}$ be the restriction map,
$r(\rho)=\rho|_\D$. 
We now show that $\fS_s(\C,\D)$ is the unique minimal closed subset of
$\Mod(\C,\D)$ for which $r$ is onto.  In a certain sense, this allows
us to determine $\fS_s(\C,\D)$ without the use of the pseudo-expectation.

\begin{theorem}\label{minonto}  Let $(\C,\D)$ be a regular MASA
  inclusion and suppose $F\subseteq \Mod(\C,\D)$ is closed and
  $r(F)=\hat{\D}$.  Then $\fS_s(\C,\D)\subseteq F$.

Suppose further that there exists a countable subset $N\subseteq \N(\C,\D)$
such that the norm-closed $\D$-bimodule generated by $N$ is $\C$ and
set $$\unistex(\C,\D):=\{\rho\in \Mod(\C,\D): \rho|_\D\text{ has a unique state extension
  to }\C\}.$$
Then 
$$\fS_s(\C,\D)=\overline{\unistex(\C,\D)}^{\text{w-}*}.$$

\end{theorem}
\begin{proof}
Since $F$ is closed and $\Mod(\C,\D)$ is compact, $F$ is compact and
Hausdorff. As $\widehat{I(\D)}$ is  projective (in the category of
compact Hausdorff spaces and continuous maps) and 
 $r$ maps $F$
onto $\hat{\D}$, there exists a continuous map
$\ell:\widehat{I(\D)}\rightarrow F$ such that
$\dual{\iota}=r\circ\ell$. Let $\epsilon:F\rightarrow \Mod(\C,\D)$ be
the inclusion map.  Then  $\ell':=\epsilon\circ \ell$ is a continuous
map of $\widehat{I(\D)}$ into $\Mod(\C,\D)$ such that
$r\circ \ell'=\dual{\iota}.$  Theorem~\ref{uniqlift} shows that
$\ell'=\dual{E}$.   Therefore, the range of $\dual{E}$ is contained in
$F$,
that is, $\fS_s(\C,\D)\subseteq F$.

Suppose now that there is a countable subset $N\subseteq \N(\C,\D)$
which generates $\C$ as a $\D$-bimodule.
  Theorem~\ref{denseuep} implies that $\hat{\D}=r(\overline{\unistex(\C,\D)})$, so we
have $\fS_s(\C,\D)\subseteq \overline{\unistex(\C,\D)}$.  To complete the proof,
observe that 
$\fS_s(\C,\D)$ is closed and $\unistex(\C,\D)\subseteq \fS_s(\C,\D)$.
\end{proof}
 
The following result shows that, in the terminology of Section~\ref{sectCompSt},
each element of $\fS_s(\C,\D)$ is a compatible state.  
\begin{proposition}\label{rhocircE}  Let $(\C,\D)$ be a regular MASA
inclusion and let $\sigma\in\fS_s(\C,\D)$.    
Then for every $v\in \N(\C,\D)$, 
$$|\sigma(v)|^2\in\{0,\sigma(v^*v)\}.$$
\end{proposition}
\begin{proof} Let $\rho\in\widehat{I(\D)}$ be such that
  $\sigma=\rho\circ E$ and 
suppose $v\in\N(\C,\D)$ is such that $\sigma(v)\neq 0$.  Then  $0\neq
|\rho(E(v))|^2$, so by part (c)
of Theorem~\ref{uniquecpmap}, $\rho(P)\neq 0$.
  By Lemma~\ref{eauid},  $P$ is a projection, so $\rho(P)=1$.  Thus,
$|\rho(E(v))|^2=\rho(\iota(v^*v))=\rho(E(v^*v))=\sigma(v^*v)$. 
\end{proof}

Our next goal  is
Theorem~\ref{natext} below, which shows that $\fS_s(\C,\D)$ is an
$\N(\C,\D)$-invariant subset of $\Mod(\C,\D)$.  In the case that $\C$
is countably generated as a $\D$-bimodule, this follows from the
second part of Theorem~\ref{minonto}, but we have not found a proof in
the general case using Theorem~\ref{minonto}.  
The fact that 
 the left
kernel of the pseudo-expectation is an ideal will follow easily from
Theorem~\ref{natext}, see
Theorem~\ref{Eideal}. 

Our route to Theorem~\ref{natext} involves a study  of the Gelfand
support of $E(v)$ for $v\in \N(\C,\D)$;  we also obtain a formula for the
Gelfand transform of $E(v)$.  We begin with three lemmas on properties
of projective covers. 

\begin{lemma}\label{pcover}  Let $X$ be a compact Hausdorff space and
  let $(P,f)$ be a projective cover for $X$. 
\begin{enumerate}
\item[a)] If $G\subseteq X$ is an
  open set, then $(\overline{f^{-1}(G)},f|_{\overline{f^{-1}(G)}})$ 
is a projective cover for
  $\overline{G}$.
\item[b)]  If $Q\subseteq P$ is clopen, then 
  $f(Q)^\circ$ is dense in $f(Q)$ and $Q=\overline{f^{-1}(f(Q)^\circ)}$.
\end{enumerate}
\end{lemma}

\begin{proof} 
Before beginning the proof, observe that if $\iota:C(X)\rightarrow
C(P)$ is given by $\iota(d)=d\circ f$, then $(C(P),\iota)$ is an
injective envelope for $C(X)$.

\textit{a)} 
Let $Y:=\overline{f^{-1}(G)}$.  Then $Y$ is
  compact, so  $f(Y)$ is a closed subset of $X$ which contains
  $G$.  Hence $\overline{G}\subseteq f(Y)$.  On the other hand,
  $f^{-1}(\overline{G})$ is a closed subset of $P$ containing
  $f^{-1}(G)$, so $Y\subseteq f^{-1}(\overline{G})$.  Hence
  $f(Y)\subseteq \overline{G}$.  Therefore $(Y,f)$ is a cover for
  $\overline{G}$.

Since $P$ is projective, it is Stonean, and hence $Y$ is clopen in
$P$.  Since clopen subsets of projective spaces are projective, 
$Y$ is  projective.  The proof 
of part (a) 
will be
complete once we show that $(Y,f|_Y)$ is a rigid cover for
$\overline{G}$~(see \cite[Theorem~2.16]{HadwinPaulsenInPrAnTo}).

So suppose that $h:Y\rightarrow Y$ is continuous and $f\circ h=
f|_Y$.  Define $\tilde{h}:P\rightarrow P$ by 
$$\tilde{h}(t)=\begin{cases} h(t)&\text{if $t\in Y$}\\ t&\text{if $t\notin
    Y$.}\end{cases}$$ Since $Y$ is clopen in $P$, $\tilde{h}$ is
continuous.  Moreover, $f\circ \tilde{h}=f$, so by the rigidity of
$(P,f)$, we see that $\tilde{h}$ is the identity map on $P$.
Therefore $h$ is the identity map on $Y$, which shows that $(Y,f|_Y)$
is a rigid cover.

\textit{b)}  The case when $Q=\emptyset$ is trivial, so we assume $Q$
is non-empty.  Let $$\M:=\{d\in C(X): 0\leq \iota(d)\leq
\chi_Q\}\dstext{and set} 
G:=\bigcup_{d\in \M} \supp(d).$$ 
Then $G$ is non-empty because $\chi_Q\neq 0$ and $(C(P),\iota)$ is
Hamana regular.  
Notice that $f^{-1}(G)\subseteq Q$:  indeed, if $p\in f^{-1}(G)$, then
there exists $d\in \M$ such that $f(p)\in \supp(d)$, so $0<d(f(p))\leq
\chi_Q(p)$, whence $p\in Q$.  

As $P$ is extremally disconnected,
$W:=\overline{f^{-1}(G)}$ is a clopen subset of $P$.  We will show
that $W=Q$. Clearly
  $W\subseteq Q$.  If $Q\setminus W\neq \emptyset$, then 
  $Q\setminus W$ is a clopen subset of $P$, so we may find
  a non-zero $d_1\in C(X)$ with $0\leq \iota(d_1)\leq \chi_{Q\setminus
    W}\leq \chi_Q$.  But then $d_1\in\M$, so $\supp(\iota(d_1))\subseteq
  f^{-1}(G)\subseteq W,$ contradicting $0\leq \iota(d_1)
  \leq\chi_{Q\setminus W}$.  Hence $W=Q$. 

  By part (a), $(W, f|_W)$ is a cover for $\overline{G}$.  Thus,
  $f(W)=f(Q)=\overline{G}$.  As $f^{-1}(G)\subseteq Q$, we have $G\subseteq
  f(Q)^\circ$, so that $f(Q)^\circ$ is dense in $f(Q)$.  

  Finally, put $W_1:= \overline{f^{-1}(f(Q)^\circ)}.$ Since
  $G\subseteq f(Q)^\circ$, we have $Q=W\subseteq W_1$.  Part (a) again
  shows that $(W_1, f|_{W_1})$ is a projective cover for
  $\overline{f(Q)^\circ}=f(Q)$, so in particular, this cover is
  essential.  The inclusion map $\alpha$ of $W$ into $W_1$ satisfies
  $f(\alpha(W))=f(Q)$.  Because $(W_1, f|_{W_1})$ is an essential
  cover of $f(Q)$, we conclude $\alpha$ is onto.  Thus $W=W_1$, and
  the proof of (b) is complete.
\end{proof}

We leave the proof of the following lemma to the reader.

\begin{lemma}\label{exthomeo}  Suppose that for $i\in\{1,2\}$,  
  $X_i$ is a compact
  Hausdorff space and that $\phi:X_1\rightarrow X_2$ is a
  homeomorphism. Let $(C_i, f_i)$ be a projective cover for $X_i$.
  Then there exists a unique homeomorphism $\Phi: C_1\rightarrow C_2$
  such that $f_2\circ \Phi=\phi\circ f_1.$
\end{lemma} 

Our final lemma on projective covers is a strengthening of
Lemma~\ref{exthomeo}:  rather than extending a homeomorphism to the
injective envelope, partial homeomorphisms are extended.

\begin{lemma}\label{extphomeo}  Let $(P,f)$ be a projective cover for
  the compact Hausdorff space $X$ and suppose $h\in\oinv(X)$ is a
  partial homeomorphism.  Then there exists a unique partial
  homeomorphism $I(h)\in \oinv(P)$ such that:  
\begin{enumerate}
\item[a)] $\dom(I(h))=f^{-1}(\dom(h))$, $\ran(I(h))=f^{-1}(\ran(h))$ and
\item[b)] $h\circ f=f\circ I(h)$.
\end{enumerate}
\end{lemma}

\begin{proof}
Let $H_1=\dom(h)$ and $H_2=\ran(h)$.  
Set $$\T:=\{G\subseteq H: G \text{ is open in $X$ and
  $\overline{G}\subseteq H_1$}\}.$$

Let $G\in \T$.  Lemma~\ref{pcover} shows that
$(\overline{f^{-1}(G)},f)$ and $(\overline{f^{-1}(h(G))},f)$ are
projective covers for $\overline{G}$ and $\overline{h(G)}$.  By 
Lemma~\ref{exthomeo},
there exists a unique homeomorphism
$h_G:\overline{f^{-1}(G)}\rightarrow \overline{f^{-1}(h(G))}$ such
that $$h\circ f|_{\overline{f^{-1}(G)}}=f\circ h_G.$$
We now let $I(h)$ be the inductive limit of $\{h_G\}_{G\in \T}$.  Here
is an outline.

Since $$\bigcup_{G\in\T}G=H_1=\bigcup_{G\in \T}\overline{G}\dstext{we
  have}
\bigcup_{G\in\T}f^{-1}(G)=f^{-1}(H_1)=\bigcup_{G\in\T}f^{-1}(\overline{G}).$$

Let  $G_1, G_2\in\T$ and suppose $G_1\cap G_2\neq \emptyset$.   Put
$Q:=\overline{f^{-1}(G_1)}\cap \overline{f^{-1}(G_2)}.$  Then $Q$ is a
clopen set in $P$, so by Lemma~\ref{pcover}, $Q=\overline{f^{-1}(f(Q)^\circ)}$.
Note that $f(Q)^\circ\in\T$.  Hence
$$h\circ f|_Q=h\circ f|_{\overline{f^{-1}(f(Q)^\circ)}}=f\circ
h_{f(Q)^\circ}.$$  Thus, for $i=1,2$, 
$$f\circ h_{G_i}|_{Q}=h\circ f|_Q.$$ This means that given $p\in P$,
we may define 
$I(h)(p)=h_G(p)$ where $G$ is any element of $\T$ containing $p$.

Clearly $I(h)$ satisfies (a) and (b).  If $H$ is another such map,
then for every $G\in\T$, the restrictions of $H$ and $I(h)$ to
$\overline{f^{-1}(G)}$ are both equal to $h_G$;  this gives uniqueness
of $I(h)$.    The continuity and bijectivity of  $I(h)$ are left to
the reader.

\end{proof} 

When $(I(\D),\iota)$ is an injective envelope for $\D$,
$(\widehat{I(\D)},\dual{\iota}|_{\widehat{I(\D)}})$ is a projective
cover for $\hat{\D}$.    In the following technical result, we will simply write
$\dual{\iota}$ instead of $\dual{\iota}|_{\widehat{I(\D)}}.$

\begin{proposition}\label{formula}
Suppose that $(\C,\D)$ is a regular MASA inclusion and $v\in
\N(\C,\D)$.  Let $(I(\D),\iota)$ be an injective envelope for $\D$,
and let $E$ be the pseudo-expectation for $\iota$.
Then
\begin{equation}\label{suppEv}
(\dual{\iota})^{-1}((\fix{\beta_v})^\circ)\subseteq
\supp(\widehat{E(v)})\subseteq
(\dual{\iota})^{-1}(\fix{\beta_v})\dstext{and}
\overline{(\dual{\iota})^{-1}((\fix{\beta_v})^\circ)}=
\overline{\supp(\widehat{E(v)})}.
\end{equation}  Moreover, if
$\rho\in (\dual{\iota})^{-1}((\fix\beta_v)^\circ)$, then for any $d\in J_v$
with 
$\rho(\iota(d))\neq0$,  
\begin{equation}\label{formulaEv}
\rho(E(v))=\frac{\rho(\iota(vd))}{\rho(\iota(d))}.
\end{equation}
\end{proposition}

\begin{proof}
Suppose that $\rho\in\widehat{I(\D)}$ and $\rho\circ\iota\in
(\fix{\beta_v})^\circ.$  Then $\rho(\iota(v^*v))\neq 0$ (as
$\fix{\beta_v}\subseteq \dom(\beta_v)=\supp(\widehat{v^*v})$).  
By Proposition~\ref{Jvee} there exists $d\in J_v$  such that
$\rho(\iota(d))\neq0$.  Since $vd\in\D$,    
$$0\neq\rho(\iota(d^*d))\rho(\iota(v^*v))= 
\rho(\iota(d^*v^*vd))=|\rho(\iota(vd))|^2.$$ 
Hence 
$$\rho(E(v))=\frac{\rho(E(v))\rho(\iota(d))}{\rho(\iota(d))}=
\frac{\rho(E(vd))}{\rho(\iota(d))}=\frac{\rho(\iota(vd))}{\rho(\iota(d))}\neq
0.$$  
Thus we obtain~\eqref{formulaEv} and 
also $(\dual{\iota})^{-1}((\fix{\beta_v})^\circ)\subseteq \supp(\widehat{E(v)}).$

We next show  
\begin{equation}\label{preoppo}
\supp(\widehat{E(v)})\subseteq (\dual{\iota})^{-1}(\fix{\beta_v}).
\end{equation} Suppose that $\rho\in\supp(\widehat{E(v)}).$    By
Proposition~\ref{rhocircE}, 
$0\neq |\rho(E(v))|^2=\rho(E(v^*v))=\rho(\iota(v^*v)).$  Hence
$\rho\circ\iota\in \dom(\beta_v)$.  Let $d\in\D$ be
such that $d\geq 0$ and $\rho(\iota(d))\neq 0$. Then
$\rho(E(d^{1/2}v))=\rho(\iota(d))^{1/2}\rho(E(v))\neq 0$.  Then   
\begin{equation}\label{FBV}
\beta_v(\rho\circ\iota)(d)=\frac{\rho(\iota(v^*dv))}{\rho(\iota(v^*v))}
=\frac{\rho(E(v^*dv))}{\rho(\iota(v^*v))}
= \frac{\rho(E((d^{1/2}v)^*(d^{1/2}v)))}{\rho(\iota(v^*v))}
=\frac{|\rho(E(d^{1/2}v))|^2}{\rho(\iota(v^*v))}\neq 0.
\end{equation}
(The last equality
in~\eqref{FBV} follows from Proposition~\ref{rhocircE}.)
As \eqref{FBV} holds for every $d\in\D^+$ with $\rho(\iota(d))\neq 0$,
we conclude that 
$\beta_v(\rho\circ\iota)=\rho\circ\iota$.  This gives \eqref{preoppo}.

The first paragraph of the proof gives
$\overline{(\dual{\iota})^{-1}((\fix{\beta_v})^\circ)}\subseteq 
\overline{\supp(\widehat{E(v)})}.$  
Let $$Q:=\overline{\supp(\widehat{E(v)})}\setminus
\overline{(\dual{\iota})^{-1}((\fix{\beta_v})^\circ)}\dstext{and}
\fI:=\{d\in\D: \supp\widehat{\iota(d)}\subseteq Q\}.$$  Then $Q$ is a
clopen
set, and $\fI$ is a closed ideal in $\D$.  

We claim that 
$\fI\subseteq (J_v\vee K_v)^\perp$.  To see this, fix $d\in\fI$. 
Proposition~\ref{Jvee} shows that for any $h\in J_v$,
  $\supp(\widehat{\iota(h)})\subseteq
(\dual{\iota})^{-1}((\fix(\beta_v)^\circ))$; thus  $dh=0$ for
any $h\in J_v$.  
Now we show $dK_v=0$.  Suppose $h\in K_v$.  By Theorem~\ref{uniquecpmap}(b),
$E(v)\iota(h)=0$, so $\widehat{\iota(h)}$ vanishes on
$\supp(\widehat{E(v)})$.  Continuity implies that $\widehat{\iota(h)}$
vanishes on $Q$ as well.  Therefore, $\iota(d)\iota(h)=0$, so $dh=0$. 
As $dh=0$ for all $h$ belonging to a generating set for $J_v\vee K_v$,
$d\in (J_v\vee K_v)^\perp$.  The claim follows.

As $J_v\vee K_v$ is an essential
ideal, $(J_v\vee K_v)^\perp=(0),$ whence $\fI=(0)$.  
The Hamana regularity of $(I(\D),\iota)$ now
implies that $Q=\emptyset$, which completes the proof.
\end{proof}

\begin{remark}{Notation}\label{vhat}  Let $(\C,\D)$ be a 
regular MASA inclusion, and let 
  $v\in \N(\C,\D)$.  Define $\hat{v}: (\fix{\beta_v})^\circ\rightarrow
  \bbC$ as follows.  Given $\sigma\in (\fix\beta_v)^\circ$, choose
  $d\in J_v$ so that $\sigma(d)\neq 0$ and set 
$$\hat{v}(\sigma)=\frac{\sigma(vd)}{\sigma(d)}.$$
Proposition~\ref{formula} shows this is well-defined, and it is easy to
show that $\hat{v}$ is a bounded continuous function on $(\fix\beta_v)^\circ$.   
Extend $\hat{v}$ to a bounded Borel function on $\hat{\D}$ by
defining it to be zero off $(\fix\beta_v)^\circ;$ we
denote this extension by $\hat{v}$ as well.
\end{remark}
\begin{remark}{Remark}   Take $(I(\D),\iota)$ to be the Dixmier algebra of $\hat{\D}$ and $\iota$ to
be the map which takes  $d\in \D$ to the equivalence class of
$\hat{d}$ in $I(\D)$.  Then 
 $E(v)$ is the equivalence class of the bounded Borel function $\hat{v}$ in
the Dixmier algebra.
\end{remark}

\begin{theorem}\label{natext} Suppose that $(\C,\D)$ is a regular MASA
  inclusion.   Then $\fS_s(\C,\D)$ is a compact $\N(\C,\D)$-invariant
  subset of $\Mod(\C,\D)$ 
and the restriction mapping $\fS_s(\C,\D)\ni\rho\mapsto \rho|_\D$ 
is a continuous surjection.

In fact, given $v\in\N(\C,\D)$ and $\rho\in\fS_s(\C,\D)$ such that
$\rho(v^*v)\neq 0$, let $\tau\in\widehat{I(\D)}$ satisfy
$\rho=\tau\circ E$.  Then $\tilde{\beta}_v(\rho)=I(\beta_v)(\tau)\circ E.$
\end{theorem}
\begin{proof}
As noted following Definition~\ref{defcst}, $\fS_s(\C,\D)$ is a closed
subset of $\Mod(\C,\D)$ and $r(\fS_s(\C,\D))=\hat{\D}$.  So
$\fS_s(\C,\D)$ is compact, and the weak-$*$-weak-$*$ continuity of
$r$ is clear.

Now let $\rho\in\fS_s(\C,\D)$ and fix $\tau\in \widehat{I(\D)}$ so
that $\rho=\tau\circ E$.  Let $v\in\N(\C,\D)$ be such that 
$\rho(v^*v)\neq 0$.  Then $\tau(E(v^*v))=\tau(\iota(v^*v))\neq 0$, so 
Lemma~\ref{extphomeo} shows that $\tau\in \dom(I(\beta_v))$.

For $\tet\in \dom(I(\beta_v))$, define states on $\C$ by
$$\mu_\tet(x)=\frac{\tet(E(v^*xv))}{\tet(\iota(v^*v))}\dstext{and}
\mu_\tet'(x) =I(\beta_v)(\tet)(E(x)).$$ (Observe that 
$\mu_\tau=\tilde{\beta}_v(\rho)$.)  Note that
$$\mu_\tet|_\D=\beta_v(\tet\circ\iota)= \mu_\tet'|_\D.$$  Hence for 
$d\in\D$ and $x\in \C$, we have
$$\mu_\tet(x)\beta_v(\tet\circ\iota)(d)=\mu_\tet(xd)\dstext{and}
\mu_\tet'(x)\beta_v(\tet\circ\iota)(d)=\mu_\tet'(xd).$$

In particular, if $w\in\N(\C,\D)$ and $d\in J_w$ we have
\begin{equation}\label{mutau}
\mu_\tet(w)\beta_v(\tet\circ\iota)(d)=\mu_\tet(wd)=\beta_v(\tet\circ\iota)(wd)
=\mu_\tet'(wd)=
\mu_\tet'(w)\beta_v(\tet\circ\iota)(d).
\end{equation}

To complete the proof, it suffices to show that 
for every $w\in \N(\C,\D)$, 
$\mu_\tau(w)=\mu_\tau'(w)$.  
So fix $w\in \N(\C,\D)$.  We show that $\mu_\tau(w)=\mu_\tau'(w)$ 
 by proving the following two
statements:
\begin{enumerate}
\item if
$\mu_\tau'(w)\neq 0$, then $\mu_\tau(w)=\mu_\tau'(w);$ and
\item if $\mu_\tau(w)\neq 0$, then $\mu_\tau'(w)=\mu_\tau(w)$.
\end{enumerate}

If $\mu_\tau'(w)\neq0$, then $I(\beta_v)(\tau)\in
\supp(\widehat{E(w)})$, so Proposition~\ref{formula} implies that there
exists a net $\rho_\alpha\in (\dual{\iota})^{-1}((\fix\beta_w)^\circ)$ such
that $\rho_\alpha\rightarrow I(\beta_v)(\tau)$.    As
$\ran(I(\beta_v))$ is an open subset of $\widehat{I(\D)}$, 
we may assume that $\rho_\alpha\in
\ran(I(\beta_v))$ for every $\alpha$.  Put
$\tau_\alpha=I(\beta_v)^{-1}(\rho_\alpha)$, so $(\tau_\alpha)_\alpha$ is a
net in $\dom(I(\beta_v))$ such that 
$\tau_\alpha\rightarrow\tau$. 
Since $\rho_\alpha\circ\iota\in (\fix\beta_w)^\circ$, 
given $\alpha$ we may find $d_\alpha\in
J_w$ such that $0\neq \rho_\alpha(\iota(d_\alpha)).$  But
$\rho_\alpha(\iota(d_\alpha))=I(\beta_v)(\tau_\alpha)(\iota(d_\alpha))=
\beta_v(\tau_\alpha\circ\iota)(d_\alpha)$.  Taking
$\lambda=\tau_\alpha$ in  \eqref{mutau} shows that
$\mu_{\tau_\alpha}(w)=\mu_{\tau_\alpha}'(w)$.  Continuity of the maps
$\tet\mapsto \mu_\tet$ and $\tet\mapsto \mu_\tet'$ gives
$\mu_\tau(w)=\mu_\tau'(w)$.

Turning to (2), suppose  $\mu_\tau(w)\neq 0.$  Then $\tau\in
\supp(\widehat{E(v^*wv)})$, so there exists a net
$\tau_\alpha\in
(\dual{\iota})^{-1}((\fix\beta_{v^*wv})^\circ)$  with
$\tau_\alpha\rightarrow \tau$.   Then $\tau_\alpha\circ\iota\in
\dom(\beta_{v^*wv})\subseteq \dom\beta_v$.
Thus, for a given $\alpha$, we may find a
neighborhood $N$ of $\tau_\alpha\circ \iota$ with $N\subseteq
(\fix\beta_{v^*wv})^\circ\cap\dom(\beta_v)$.  Now for each $y\in N$, we have
$\beta_{v^*wv}(y)=y$, so $\beta_{w}((\beta_v)(y))=\beta_v(y)$.  Hence
$\beta_v(N)\subseteq \fix(\beta_w)$.   Therefore
$\beta_v(\tau_\alpha\circ\iota)\in \fix(\beta_w)^\circ.$   So if $d\in
J_w$ satisfies $\beta_v(\tau_\alpha\circ\iota)(d)\neq 0$, then
\eqref{mutau} gives $\mu_{\tau_\alpha}'(w)=\mu_{\tau_\alpha}(w)$.
Continuity again gives $\mu_\tau'(w)=\mu_\tau(w)$.

Thus both (1) and (2) hold, and the proof is complete.  
\end{proof}

We now show the left kernel of the pseudo-expectation on a regular
MASA inclusion is an ideal which is the unique ideal which is maximal 
with respect to being
disjoint from $\D$.
\begin{theorem}\label{Eideal}  Let $(\C,\D)$ be a regular MASA
  inclusion.  Then the left kernel of the pseudo-expectation $E$, 
$$\L(\C,\D):=\{x\in\C:E(x^*x)=0\}$$ is an ideal of $\C$ such that 
$\L(\C,\D)\cap \D=(0).$  

Moreover,  if $\K\subseteq \C$ is an ideal such
that $\K\cap \D=(0)$, then $\K\subseteq \L(\C,\D)$.

\end{theorem}
\begin{proof}
Theorem~\ref{natext} shows that $\fS_s(\C,\D)$ satisfies the
hypotheses of Proposition~\ref{invideal}.  In the notation of
Proposition~\ref{invideal}, we have $\L(\C,\D)=\K_{\fS_s(\C,\D)}$, so
$\L(\C,\D)$ is a norm-closed two-sided ideal of $\C$.  
If $x\in \L(\C,\D)\cap \D$, then $0=E(x^*x)=\iota(x^*x)$.  As
$\iota$ is one-to-one,  $x=0$.

Suppose now that $\K\subseteq \C$ is an ideal with $\K\cap \D=(0).$
Let $\C_1=\C/\K$, and let $q:\C\rightarrow \C_1$ be the quotient map.  
Since $\K\cap \D=(0)$, $q|_\D$ is faithful,  so we may regard $\D$ as
a subalgebra of $\C_1$.  Let $F:=\{\rho\circ q: \rho\in
\Mod(\C_1,\D)\}$.  Then $F$ is a closed subset of $\Mod(\C,\D)$, 
and the restriction
map, $f\in F\mapsto f|_\D$ maps $F$ onto $\hat{\D}$.
By Theorem~\ref{minonto}, $\fS_s(\C,\D)\subseteq F$.  Hence every
element of $\fS_s(\C,\D)$ annihilates $\K$, so $\K\subseteq \L(\C,\D)$.
\end{proof}

The ideal
$\L(\C,\D)$ behaves reasonably well with respect to certain regular
$*$-homomorphisms, as the next result shows.

\begin{corollary}\label{LRegSub}
Suppose for $i=1,2$ that $(\C_i,\D_i)$ are regular MASA inclusions,
and
that $\alpha:(\C_1,\D_1)\rightarrow (\C_2,\D_2)$ is a regular
$*$-homomorphism such that $\alpha|_\D$ is one-to-one.
Then \begin{equation}\label{Lsubeq}
\{x\in \C_1: \alpha(a)\in \L(\C_2,\D_2)\}\subseteq \L(\C_1,\D_1).
\end{equation}
\end{corollary}
\begin{proof}
The set $\{x\in \C_1: \alpha(x)\in \L(\C_2,\D_2)\}$ is an ideal of
$\C_1$ whose intersection with $\D_1$ is trivial.
\end{proof}

\begin{remark}{Remark} We expect that equality holds  in~\eqref{Lsubeq}
 if for every $0\leq h\in \D_2$,
$h=\sup_{\D_2}\{\alpha(d): d\in \D_1\text{ and } 0\leq \alpha(d)\leq
h\}$.  Also, it would not be surprising if this condition
characterized equality in~\eqref{Lsubeq}.
\end{remark}

Since every Cartan inclusion $(\C,\D)$ satisfies $\L(\C,\D)=(0)$, we
make the following definition.
\begin{definition}\label{virtualCartan}  
A \textit{virtual Cartan inclusion} is a regular MASA inclusion such that
$\L(\C,\D)=(0).$ 
\end{definition}

\section{Compatible States}\label{sectCompSt}

Since the extension property does not always hold for an inclusion
$(\C,\D)$, we identify a useful class of states in $\Mod(\C,\D)$,
which we call
$\D$-compatible states.  

To motivate the definition, observe that when $(\C,\D)$ is a regular
EP inclusion, the only way to extend a pure state $\sigma\in\hat{\D}$
to $\C$ is via composition with the expectation: $\rho:=\sigma\circ
E$.  Then the GNS representation $(\pi_\rho,\H_\rho)$ arising from
$\rho$ has the property that for any $v\in\N(\C,\D)$ either $I+L_\rho$
and $v+L_\rho$ are orthogonal in the Hilbert space $\H_\rho$, or one
is a scalar multiple of the other, according to whether or not the
Gelfand transform of $E(v)$ is zero in a neighborhood of $\sigma$.
Furthermore, the techniques used in the proof
of~\cite[Proposition~5.4]{DonsigPittsCoSyBoIs} show that
$\pi_\rho(\D)''$ is an atomic MASA in $\B(\H_\rho)$ and also that
for every $v\in \N(\C,\D)$, $v+L_\rho$ is an eigenvector for $\pi_\rho(\D)$. 
The intersection $\J$ of the kernels of such representations is 
the left kernel of the expectation $E$, $\D\cap \J=(0)$, and the
quotient of $(\C,\D)$ by $\J$ yields a \cstardiag\ with the same
coordinate system as $(\C,\D)$, see \cite[Theorem~4.8]{DonsigPittsCoSyBoIs}.

We shall define the set of compatible states to be those states $\rho$
on $\C$ for which the vectors $\{v+L_\rho:v\in \N(\C,\D)\}$ form an
orthogonal set of vectors.  These states have many of the properties
listed in the previous paragraph, but have the advantage of not
needing the extension property or a conditional expectation for their
definition.  Here is the formal definition.

\begin{definition}  Let $(\C,\D)$ be an inclusion. 
\begin{enumerate}
\item A state
  $\rho$ on $\C$ is called \textit{$\D$-compatible} if
  for every $v\in \N(\C,\D)$,
$$|\rho(v)|^2\in\{0, \rho(v^*v)\}.$$
 When the context is clear, we will simply use the term
  \textit{compatible state} instead of $\D$-compatible state.
\item We will use $\fS(\C,\D)$ to denote the set of all 
 $\D$-compatible states on $\C$.  Topologize $\fS(\C,\D)$ with the
 relative weak-$*$ topology.
\item For $\rho\in\fS(\C,\D)$, let $\Delta_\rho:=\{v\in\N(\C,\D):
  \rho(v)\neq 0\}$, and $\Lambda_\rho:=\{v\in\N(\C,\D): \rho(v^*v)
  >0\}$.  Define a relation $\sim_\rho$ on
  $\Lambda_\rho$ by $(v,w)\in \sim_\rho$ if and only if $v^*w\in
  \Delta_\rho$.   (We shall prove that $\sim_\rho$ is an equivalence
  relation momentarily, and then will simply write $v\sim_\rho w$.)

\end{enumerate} 
\end{definition}

\begin{remark}{Remarks} \label{basic}
\begin{enumerate}
\item When $(\C,\D)$ is a regular MASA inclusion,
  Proposition~\ref{rhocircE} shows that $\fS_s(\C,\D)\subseteq
  \fS(\C,\D)$.
\item As $|\rho(x)|^2\leq \rho(x^*x)$ for any state $\rho\in\dual{\C}$
  and any $x\in\C$, we see that $\D$-compatible states satisfy an
  extremal property relative to the normalizers for $\D$, and one
  might expect an inclusion relationship between compatible states and
  pure states.  However, there is not.  Example~\ref{ToeplitzAlg}
  gives an example of a Cartan inclusion $(\C,\D)$ and element of
  $\fS(\C,\D)$ which is not a pure state on $\C$, while
  Example~\ref{disGrp} gives an example of a Cartan inclusion
  $(\C,\D)$ and a pure state $\rho$ on $\C$ such that $\rho\in
  \Mod(\C,\D)$, yet $\rho\notin\fS(\C,\D)$.  As we shall see
  momentarily, $\fS(\C,\D)\subseteq \Mod(\C,\D)$.  Thus no such
  inclusion relationship exists.
\item 
For general inclusions, it is possible that
  $\fS(\C,\D)=\emptyset$ (see Theorem~\ref{allunitaries}).
\item \label{basic2} The following simple observation will be useful during the
  sequel:  for $i=1,2$, let $(\C_i,\D_i)$ be inclusions and suppose
  that $\alpha: \C_1\rightarrow \C_2$ is a regular and unital $*$-homomorphism.
  Then $$\dual{\alpha}(\fS(\C_2,\D_2))\subseteq \fS(\C_1,\D_1).$$
\end{enumerate} 
\end{remark}

Here are some properties of elements of $\fS(\C,\D)$ which hold for
any inclusion.

\begin{proposition}\label{Dextreme}
Let $(\C,\D)$ be an inclusion and let $\rho\in\fS(\C,\D)$.
  The following statements hold.
\begin{enumerate}
\item\label{Dext0}  Suppose $v\in\Delta_\rho$.  Then
  for every $x\in\C$, $$\rho(vx)=\rho(v)\rho(x)=\rho(xv).$$
\item \label{Dext1} 
The restriction of $\rho$ to $\D$ is a multiplicative linear
functional on $\D$.
\item \label{Dext2} Suppose
  $v\in\Delta_\rho$. Then for every $x\in\C$,
  $$\rho(v^*xv)=\rho(v^*v)\rho(x).$$
\item\label{Dext2a0}  If $v_1, v_2\in\Lambda_\rho$
  and $(v_1,v_2)\in \sim_\rho$, then 
\[
|\rho(v_1^*v_2)|^2=\rho(v_1^*v_1)\rho(v_2^*v_2).\] Moreover, $\sim_\rho$ is an
  equivalence relation on $\Lambda_\rho$.
\item\label{Dext2a1} $\fS(\C,\D)$ is an $\N(\C,\D)$-invariant subset of
  $\Mod(\C,\D)$.  
\item\label{Dext2a} If $v\in \Lambda_\rho$, then $v+L_\rho$ is an
  eigenvector for $\pi_\rho(\D)$; in particular, for every $d\in\D$,
  $$\pi_\rho(d)(v+L_\rho) = \frac{\rho(v^*dv)}{\rho(v^*v)} v + L_\rho.$$ 
\item \label{Dext3}
The set $\fS(\C,\D)$ is weak-$*$ closed in
  $\dual{\C}$ and the restriction mapping, $\rho\in\fS(\C,\D)\mapsto
  \rho|_\D$, is a weak-$*$--weak-$*$ continuous mapping of $\fS(\C,\D)$
  into $\hat{\D}$.   
\end{enumerate}
\end{proposition}

\begin{proof} 
\prt{Dext0} Since $\rho\in\fS(\C,\D)$, an easy calculation yields
$v-\rho(v)I\in L_\rho$.  But $L_\rho$ is a left ideal and
$L_\rho\subseteq \ker\rho$.  So for $x\in\C$, we have
$\rho(x(v-\rho(v)I)) =0$.  So $\rho(xv)=\rho(x)\rho(v)$.  As
$\rho(v^*)=\overline{\rho(v)}\neq 0$, a similar argument shows that
$0=\rho((v-\rho(v)I)x).$ So part~\eqref{Dext0} holds.  

\prt{Dext1}
Since $\D\subseteq \N(\C,\D)$, this follows from part~\eqref{Dext0} and
continuity of $\rho$.

\prt{Dext2} This follows from part~\eqref{Dext0} and the fact that
$\Delta_\rho$ is closed under the adjoint operation.

 \prt{Dext2a0} Let $\sigma=\rho|_\D$ and for $i=1,2$ put
$\sigma_i=\beta_{v_i}(\sigma)$. Then $\sigma_1=\sigma_2$ by
statement~\eqref{Dext2} and Proposition~\ref{betacalc}. Therefore,
since $\rho(v_1^*v_2)\neq 0$, we have
$$|\rho(v_1^*v_2)|^2=\rho(v_2^*v_1v_1^*v_2)
=\sigma_2(v_1v_1^*)\sigma(v_2^*v_2)
=\sigma_1(v_1v_1^*)\sigma(v_2^*v_2)= \rho(v_1^*v_1)\rho(v_2^*v_2).$$

Clearly the relation $\sim_\rho$ is reflexive and symmetric on
$\Lambda_\rho$.  For $i=1,2,3$, suppose $v_i\in \Lambda_\rho$,
$(v_1,v_2)\in \sim_\rho $ and $(v_2,v_3)\in \sim_\rho$.  The equality verified in
the previous paragraph shows that in $\H_\rho$,
$|\innerprod{v_1+L_\rho,
  v_2+L_\rho}_\rho|^2=\norm{v_1+L_\rho}_\rho^2\,\norm{v_2+L_\rho}_\rho^2$.
Hence there exists a non-zero scalar $t$ such that
$tv_1+L_\rho=v_2+L_\rho$.  Similarly, there exists a non-zero scalar
$s$ such that $v_2+L_\rho=sv_3+L_\rho$.  So $\{v_1+L_\rho,
v_3+L_\rho\}$ is a linearly dependent set of non-zero linearly vectors in $\H_\rho$. Thus
$\rho(v_1^*v_3)\neq 0$, whence $(v_1,v_3)\sim_\rho.$

\prt{Dext2a1}  Let $v\in\Lambda_\rho$.  For $w\in\N(\C,\D)$, we claim
that $|\rho(w^*v)|^2\in \{0,\rho(w^*w)\rho(v^*v)\}$.  If $\rho(w^*v)\neq
0$, then as $|\rho(w^*v)|^2\leq\rho(w^*w)\rho(v^*v)$, we find that
$w\in\Lambda_\rho$ and $w\sim_\rho v$, so the claim holds by
statement~\eqref{Dext2a0}.   Hence
$$|\beta_v(\rho)(w)|^2=\frac{|\rho(v^*(wv))|^2}{\rho(v^*v)^2}
\in\left\{0, \frac{ \rho(v^*v) \rho(v^*w^*wv)}{\rho(v^*v)^2}\right\}
=\{0,\beta_v(\rho)(w^*w)\},$$ so $\beta_v(\rho)\in\fS(\C,\D)$.

\prt{Dext2a} Suppose $v\in\N(\C,\D)$ and that $\rho(v^*v)\neq 0$. For
$d\in\D$, let $\ds\sigma_1(d)= \frac{\rho(v^*dv)}{\rho(v^*v)}$.  Then
$\sigma_1\in\hat{\D}$, and 
 for $d\in\D$,
we have
\begin{align*}
\norm{(\pi_\rho(d)-\sigma_1(d)I)v+L_\rho}^2_\rho&=
\rho(v^*(d-\sigma_1(d)I)^*(d-\sigma_1(d)I)v)\\ &=
\sigma_1((d-\sigma_1(d)I)^*(d-\sigma_1(d)I))\rho(v^*v)=0.
\end{align*}
We conclude that $\pi_\rho(d)v+L_\rho=\sigma_1(d)v+L_\rho$, so $v+ L_\rho$ is an
eigenvector for $\pi_\rho(\D)$ and statement~\eqref{Dext2a} holds.

\prt{Dext3} Suppose $(\rho_\lambda)_{\lambda\in\Lambda}$
is a net in $\fS(\C,\D)$ and $\rho_\lambda$ converges weak-$*$ to
$\rho\in\dual{\C}$.  Let $v\in\N(\C,\D)$.  If $\rho(v)\neq 0$, then
for large enough $\lambda$, $\rho_\lambda(v)\neq 0$.  Hence
$|\rho(v)|^2=\lim_\lambda |\rho_\lambda(v)|^2 =\lim_\lambda
\rho_\lambda(v^*v)=\rho(v^*v).$ It follows that $\rho\in\fS(\C,\D)$,
so $\fS(\C,\D)$ is weak-$*$ closed.  The continuity of the restriction
mapping is obvious.

\end{proof}

\begin{remark}{Remark} 
Statement~\eqref{Dext0} says that if $v\in \Delta_\rho$, then $v\in \fM_\rho$, where
  $\fM_\rho=\{x\in\C: \rho(xy)=\rho(yx)=\rho(x)\rho(y)\;
  \forall\, y\in\C\}$, see~\cite{AndersonExReReStC*Al}.   Also, if $\B$
  is the closed linear span of $\Delta_\rho$, then $\B$ is a
  \cstaralg\ because $\Delta_\rho$ is closed under multiplication.
  Clearly $\D\subseteq \B$, so that $(\B,\D)$ is an inclusion enjoying
  the properties of regularity or MASA inclusion when $(\C,\D)$ has the
  same properties.
\end{remark}

We turn now to the issue of existence of compatible states.
When $(\C,\D)$ is a regular MASA inclusion, Theorem~\ref{natext},
shows that every $\sigma\in \hat{\D}$ extends to an
element of $\fS_s(\C,\D)$.  Applying Proposition~\ref{rhocircE}, we see  that  
compatible states 
exist in abundance for regular MASA inclusions.  We record this fact
as a theorem.

\begin{theorem} \label{scregMASA}
Let $(\C,\D)$ be a regular MASA inclusion.  If 
$\sigma\in \hat{\D}$, there exists $\rho\in\fS(\C,\D)$ such that
$\rho|_\D=\sigma$.  Moreover, $\rho$ may be chosen so that
$\rho\in\fS_s(\C,\D)$.  
\end{theorem}

The following result summarizes what we know regarding the existence
of compatible states when the hypothesis of regularity in
Theorem~\ref{scregMASA} is weakened.  Notice that in both parts of the
following result, a conditional expectation is present.

\begin{theorem}\label{rcmany}
  Suppose $(\C,\D)$ is an inclusion.   
\begin{enumerate}
\item[a)] If $(\C,\D)$ is a MASA inclusion  and
  there exists a conditional expectation $E:\C\rightarrow \D$,
  then  $\dual{E}|_{\hat{\D}}$ is a continuous one-to-one map of
  $\hat{\D}$ into $\fS(\C,\D)$.   
\item[b)] When $(\C,\D)$ has the extension property (but is not
  necessarily regular), then
  $\dual{E}|_{\hat{\D}}$ is a homeomorphism of $\hat{\D}$ onto
    $\fS(\C,\D).$  
\end{enumerate}
\end{theorem}

\begin{proof}  a) Since $E$ is onto, $\dual{E}$ is injective and
  continuous.  We must show that $\dual{E}$ carries $\hat{\D}$ into
  $\fS(\C,\D)$. 

  Let $\sigma\in\hat{\D}$ and set $\rho=\sigma\circ E$.  Suppose
  $v\in\N(\C,\D)$, and $\rho(v)\neq 0$.   By the Cauchy-Schwartz
  inequality, $\rho(vv^*)\neq
  0$.  The definition of $\rho$ shows $E(v)\neq 0$.  Let $\ds
  x:=\frac{v^*E(v)}{\rho(vv^*)}$.  

  We claim that  $x$ commutes with $\D$.  This is easy to see when
  $v\in\inter(\C,\D)$.   Since $\D$ is a MASA,
  Proposition~\ref{intnorrel} gives 
  $\N(\C,\D)=\overline{\inter(\C,\D)}$.  A continuity argument
  now establishes the claim.  

Therefore, $x\in\D$.  Hence
  $\rho(v)\rho(x)=\rho(vx)=\rho(vv^*E(v)\rho(vv^*)^{-1})= \rho(v)$,
  so that $\rho(x)=1$.  Since $vx\in \D$ we obtain,
$$|\rho(v)|^2=|\rho(vx)|^2=\rho(x^*v^*vx)=\rho(x^*)\rho(v^*v)\rho(x)=\rho(v^*v).$$
We conclude that  $\rho\in\fS(\C,\D)$, as desired.

b) Now suppose that $(\C,\D)$ has the extension property. If
$\rho\in\fS(\C,\D)$, put $\sigma=\rho|_\D.$ Then $\sigma\in\hat{\D}$.
By the extension property, we have $\rho=\sigma\circ E,$ so
$\rho=\dual{E}(\sigma)$, whence $\dual{E}|_{\hat{\D}}$ is onto.  If
$\dual{E}(\sigma_1)=\dual{E}(\sigma_2)$, then the extension property
yields $\sigma_1=\sigma_2$.  So $\dual{E}$ is a continuous bijection
of $\hat{\D}$ onto $\fS(\C,\D)$.  Since
$\hat{D}$ and $\fS(\C,\D)$ are both compact and Hausdorff, 
$\dual{E}|_{\hat{\D}}$ is a homeomorphism.

\end{proof}

\begin{theorem}\label{existcompatA}  Let $(\C,\D)$ be a
  regular inclusion (we do not assume $\D$ is a MASA in $\C$).
  The following statements hold.
\begin{enumerate}
\item[i)]\label{existcompatA1}  Suppose $(\C_1,\D_1)$ is a regular
  MASA inclusion and $\alpha:(\C,\D)\rightarrow (\C_1,\D_1)$ is a regular
  and unital $*$-homomorphism.
  Then $\dual{\alpha}$ maps $\fS_s(\C_1,\D_1)$ into $\fS(\C,\D)$.
\item[ii)]    If the relative commutant $\D^c$ of $\D$ in $\C$ is 
  abelian, then $\fS_s(\C,\D^c)\subseteq \fS(\C,\D)$ and the restriction map
$\rho\in\fS_s(\C,\D^c)\mapsto \rho|_\D$, is a weak-$*$--weak-$*$
continuous mapping of $\fS_s(\C,\D^c)$ onto $\hat{\D}$.
\end{enumerate}
\end{theorem}
\begin{proof}
  We have already observed in Remark~\ref{basic}\eqref{basic2} that $\dual{\alpha}$ carries
  $\fS(\C_1,\D_1)$ into $\fS(\C,\D)$.  As $\fS_s(\C_1,\D_1)\subseteq
  \fS(\C_1,\D_1)$, the first statement holds.

Now suppose that $\D^c$ is abelian.  
Lemma~\ref{abelcom} shows that $(\C,\D^c)$ is a regular MASA inclusion
and that the identity mapping of $\C$ onto itself 
is regular.  
By part (i), $\dual{\text{Id}}$ carries $\fS_s(\C,\D^c)$ into
$\fS(\C,\D)$; thus $\fS_s(\C,\D^c)\subseteq \fS(\C,\D)$. 
As any element of $\hat{\D}$ can be extended to an element of
$\widehat{\D^c}$, we see that 
the restriction map is onto.
Part~\eqref{Dext3} of Proposition~\ref{Dextreme} gives the weak-$*$
continuity.
  
\end{proof}

We turn now to a result which shows that there are inclusions with few
compatible states.  In fact, some inclusions have no compatible
states.  This result applies when the relative commutant of $\D$ in
$\C$ is all of $\C$, e.g.\ $(\C,\bbC I)$.  The result shows that when
$\N(\C,\D)$ is too large, it may happen that $\fS(\C,\D)$ is empty.
For example, when $\C$ is a unital simple \cstaralg\ with
$\dim(\C)>1$, then $\fS(\C,\bbC I)=\emptyset.$

\begin{theorem}\label{allunitaries}  Let $(\C,\D)$ be an
  inclusion and let $\U(\C)$ be the unitary group of $\C$.  Assume
  that $\U(\C)\subseteq \N(\C,\D)$.  Then $(\C,\D)$ is regular and
  $\fS(\C,\D)$ is the set of all multiplicative linear functionals on
  $\C$.
\end{theorem}
\begin{proof}
Since $\spn(\U(\C))=\C$, $(\C,\D)$ is a regular inclusion.  As
every multiplicative linear functional on $\C$ is a compatible state,
we need only prove that every element of $\fS(\C,\D)$ is a
multiplicative linear functional.

Fix   
  $\rho\in\fS(\C,\D)$. Then for every unitary $U\in\C$ we have
  $\rho(U)\in\{0\}\cup \bbT$.    Let $\pi$ be a universal
  representation of $\C$, and identify $\ddual{\C}$ with the von
  Neumann algebra $\pi(\C)''$.  Also, regard $\C$ as a subalgebra of
  $\ddual{\C}$.   
  Let $\ddual{\rho}$ denote the normal state
  on $\ddual{\C}$ obtained from $\rho$.  By
  \cite[II.4.11]{TakesakiThOpAlI}, every unitary in $\ddual{\C}$ is
  the strong-$*$ limit of a net of unitaries in $\C$.  Since $\ddual{\rho}$ is
  normal, 
 $\ddual{\rho}(W)\in\{0\}\cup
  \bbT$ for every unitary $W\in\ddual{\C}$.  

Let $P$ be a projection in $\ddual{\C}$.  We shall show that
$\ddual{\rho}(P)\in \{0,1\}$.   We argue by  contradiction.  Suppose that
$0<\ddual{\rho}(P) <1$.  Then  
   $0< |
  \ddual{\rho}(P)+ i\ddual{\rho}(I-P)| < 1.$  Put $W=P
  +i(I-P).$  Then $W$ is a unitary belonging to
  $\ddual{\C}$, 
  and therefore we may find a net
  $U_\alpha$ of unitaries in $\C$ so that $U_\alpha$
  converges strong-$*$ to $W$.  But then $|\rho(U_\alpha)|\rightarrow
  |\ddual{\rho}(W)|\in (0,1).$   This implies that there exists a unitary
  $U\in\C$ such that $|\rho(U)|\in (0,1)$, which is a contradiction.
  Therefore $\ddual{\rho}(P)\in\{0,1\}$ for every projection
  $P\in\ddual{\C}$.  

  Now let $P, Q\in \ddual{\C}$ be projections.  We claim that
  $\ddual{\rho}(PQ)=\ddual{\rho}(P)\ddual{\rho}(Q)$.  By the
  Cauchy-Schwartz inequality, $|\ddual{\rho}(PQ)|\leq
  \ddual{\rho}(P)\ddual{\rho}(Q)$, so that $\ddual{\rho}(PQ)=0$ if
  $0\in\{\ddual{\rho}(P), \ddual{\rho}(Q)\}$.  Suppose then that
  $\ddual{\rho}(P)=\ddual{\rho}(Q)=1$. Since $2P-I$ and $2Q-I$ are
  unitaries in $\ddual{\C}$, we may find nets of unitaries $u_\alpha$
  and $v_\alpha$ in $\C$ so that $u_\alpha$ and $v_\alpha$ converge
  $*$-strongly to $2P-I$ and $2Q-I$ respectively.  Both $\rho(u_\alpha)$
  and $\rho(v_\alpha)$ are eventually non-zero because
$$\lim\rho(u_\alpha)=\ddual{\rho}(2P-I)=1=\ddual{\rho}(2Q-I)= \lim\rho(v_\alpha).$$
As multiplication on bounded subsets of $\ddual{\C}$ is jointly
continuous in the strong-$*$ topology, $u_\alpha v_\alpha$
converges strongly to $(2P-I)(2Q-I).$ By
Proposition~\ref{Dextreme}\eqref{Dext0},
$$\ddual{\rho}((2P-I)(2Q-I))=\lim\rho(u_\alpha v_\alpha)=\lim
\rho(u_\alpha)\rho(v_\alpha)=\ddual{\rho}(2P-I)\ddual{\rho}(2Q-I)=1.$$  On the
other hand, a calculation shows that 
$$\ddual{\rho}((2P-I)(2Q-I))=4\ddual{\rho}(PQ)-3.$$ Combining these equalities
gives $\ddual{\rho}(PQ)=1$, as desired.  The
claim follows.

Let  $X=\sum_{j=1}^n\lambda_jP_j$ and $Y=\sum_{j=1}^n
\mu_jQ_j$ be linear combinations of projections $\{P_j\}_{j=1}^n$ and
$\{Q_j\}_{j=1}^n$ in $\ddual{\C}$.
It follows from the previous paragraph that
$\ddual{\rho}(XY)=\ddual{\rho}(X)\ddual{\rho}(Y)$.    
Since any von Neumann algebra is the norm closure of the
span of its projections, $\ddual{\rho}$ is multiplicative on
$\ddual{\C}$.  It then follows that $\rho$ is multiplicative on $\C$.

\end{proof}

We now turn to the representations arising from states in
$\fS(\C,\D)$.  We begin with a simple lemma concerning
states on regular inclusions, whose proof we leave to the reader.

\begin{lemma}\label{precompatProp}  Let $(\C,\D)$ be a regular
  inclusion and suppose that $\rho$ is a state on $\C$.  
Then $$\spn\{v+L_\rho:
  v\in\N(\C,\D), \, \rho(v^*v)>0\}$$ is norm-dense in $\H_\rho$.
\end{lemma}

\begin{proposition} \label{compatpure}
Let $(\C,\D)$ be a regular inclusion, let $\rho\in\fS(\C,\D)$, 
and let  $T\subseteq \Lambda_\rho$ be
chosen so that for every $v\in T$, $\rho(v^*v)=1$ and $T$ contains
exactly one element from each $\sim_\rho$ equivalence class. 
Then the
following statements hold.
\begin{enumerate}
\item\label{compatpure1} $\{v+L_\rho:v\in T\}$ is an orthonormal basis for
  $\H_\rho.$
\item\label{compatpure2}  For $v\in T$, let $\K_v:=\{\xi\in\H_\rho : \pi_\rho(d)\xi=\rho(v^* d v) \xi \text{ for
  all } d\in \D\}$  and let $\sigma=\rho|_\D$.  Then
$\K_v=\overline{\spn}\{w+L_\rho: w\in T\text{ and }
\beta_w(\sigma)=\beta_v(\sigma)\}.$
\item\label{compatpure3} For $v\in T$, let $P_v$ be the orthogonal
  projection of $\H_\rho$ onto $\K_v$.  Then $P_v$ is a minimal
  projection in $\pi_\rho(\D)''$ and
  $\join_{v\in T} P_v=I$. 
\item \label{compatpure4} $\pi_\rho(\D)''$ is  an abelian and atomic
  von Neumann algebra.  
\end{enumerate} 
\end{proposition}
\begin{proof}
\prt{compatpure1}
 If $v,
w\in T$ are distinct, then $\rho(v^*w)=0$, so that $\{v+L_\rho: v\in
T\}$ is an orthonormal set.  Part~\eqref{Dext2a0} of
Proposition~\ref{Dextreme} and the Cauchy-Schwartz inequality show that if $v\in T$, and $w\in
\Lambda_\rho$ is such that $v\sim_\rho w$, then $w+L_\rho\in
\spn\{v+L_\rho\}$.  This, together with Lemma~\ref{precompatProp}, shows
  that  
$$\overline{\spn}\{v+L_\rho: v\in T\}=
\overline{\spn}\{v+L_\rho: v\in
\Lambda_\rho\}=\overline{\spn}\{v+L_\rho: v\in \N(\C,\D)\}=\H_\rho.$$  Thus
$\{v+L_\rho:v\in T\}$ is an orthonormal basis for $\H_\rho$.

\prt{compatpure2}   If $\xi\in \overline{\spn}\{w+L_\rho:
\beta_w(\sigma)=\beta_v(\sigma)\}$, Part~\eqref{Dext2a} of
Proposition~\ref{Dextreme} implies that $\xi\in \K_v$.    For the
opposite inclusion, suppose 
$\xi\in\K_v$.  Then for $w\in T$ and $d\in \D$ we have
$$\beta_v(\sigma)(d)\innerprod{\xi,w+L_\rho} =
\innerprod{\pi_\rho(d)\xi,w+L_\rho}=
\innerprod{\xi,\pi_\rho(d^*)(w+L_\rho)} =
\beta_w(\rho)(d)\innerprod{\xi,w+L_\rho}.$$ Hence if
$\innerprod{\xi,w+L_\rho}\neq 0$, then
$\beta_v(\sigma)=\beta_w(\sigma)$.   This yields $\xi\in
\overline{\spn}\{w+L_\rho: w\in T \text{ and }
  \beta_w(\sigma)=\beta_v(\sigma)\}.$

\prt{compatpure3}  First note that for $v\in T$, $v+L_\rho\in \K_v$; thus, since
$\{v+L_\rho: v\in T\}$ is an orthonormal basis for $\H_\rho$, we
obtain $\join_{v\in T} P_v=I$. 

 Let $X\in\pi_\rho(\D)'$ and $\xi\in
\K_v$.  Then for  $d\in \D$,
$$\pi_\rho(d)X\xi=X\pi_\rho(d)\xi=\rho(v^*dv) X\xi.$$
Therefore $X\xi\in \K_v$, showing that 
$\K_v$ is an invariant subspace for $X$.  As this holds for every
$X\in\pi_\rho(\D)'$, we conclude that $P_v\in \pi_\rho(\D)''$.

Let $v\in T$ and suppose that $Q\in \pi_\rho(\D)''$ is a 
projection with $0\leq Q\leq P_v$.  For all $d\in \D$ we have
$\pi_\rho(d)P_v= \beta_v(\sigma)(d)
P_v=\innerprod{\pi_\rho(d)(v+L_\rho), v+L_\rho}P_v $.  The Kaplansky Density
Theorem shows that for every $X\in \pi_\rho(\D)''$ we have
$XP_v=\innerprod{X(v+L_\rho), v+L_\rho}P_v\in\bbC P_v$.  Since $Q$ commutes with $P_v$, $QP_v$ is a
projection; hence 
$QP_v\in \{0,P_v\}$, so $P_v$ is a minimal projection in $\pi_\rho(\D)''$.

\prt{compatpure4}  This follows from statement~\eqref{compatpure3}.

\end{proof}

The following result shows that elements of $\fS(\C,\D)$ arise from
regular representations $\pi$ of $(\C,\D)$, which can be taken so that
$\pi(\D)''$ is atomic.    For  vectors $h_1, h_2$ in a Hilbert space $\H$
we use the notation $h_1h_2^*$ for the rank-one operator $h\mapsto
\innerprod{h,h_2}h_1$.

\begin{theorem}\label{atomic}  Let $(\C,\D)$ be a regular
  inclusion.  The following statements hold.
\begin{enumerate}
\item[i)] 
 Let $\rho\in\fS(\C,\D)$.  
and let $$\vngp_\rho:=\{(v+L_\rho)(v+L_\rho)^*: v\in
  \N(\C,\D)\}''\subseteq \B(\H_\rho).$$  Then $\vngp_\rho$ is an atomic MASA in
  $\B(\H_\rho)$ and $\pi_\rho:(\C,\D)\rightarrow
  (\B(\H_\rho),\vngp_\rho)$ is a regular $*$-homomorphism. 
\item[ii)]  Conversely, suppose $\pi:\C\rightarrow \bh$ is a regular $*$-homomorphism with
  $\pi(\D)''$ a (not necessarily atomic) MASA in $\bh$,  and let $E:\bh\rightarrow \pi(\D)''$
  be any conditional expectation.  Then for any pure state $\sigma$ of
  $\pi(\D)''$, $\sigma\circ E\circ\pi\in\fS(\C,\D)$.  
\end{enumerate}
\end{theorem}

\begin{proof}  
For the first statement, choose $T$ as in the statement of Proposition~\ref{compatpure}.
For $v\in\N(\C,\D)$, we have $v+L_\rho=0$ if $v\notin \Lambda_\rho$;
and if $v+L_\rho\neq 0$, then there exists $w\in T$ such that
$v\sim_\rho w$, so $(v+L_\rho)(v+L_\rho)^*\in \bbC
(w+L_\rho)(w+L_\rho)^*$.  Since $\B:=\{w+L_\rho:w\in T\}$ is an
orthonormal basis for $\H_\rho$, we see that $\vngp_\rho$ is an atomic MASA in
$\B(\H_\rho)$.  

We now show that $\pi_\rho$ is a regular homomorphism.  Let
$v\in\N(\C,\D)$ and let $w\in T$.  Then
$\pi_\rho(v)(w+L_\rho)(w+L_\rho)^*\pi_\rho(v)^*=
(vw+L_\rho)(vw+L_\rho)^*\in\vngp_\rho$.  As
$\spn\{(w+L_\rho)(w+L_\rho)^*: w\in T\}$ is weakly dense in $\vngp_\rho$,
we conclude that $\pi_\rho(v)\vngp_\rho \pi_\rho(v)^*\subseteq \vngp_\rho$.
Similarly $\pi_\rho(v)^*\vngp_\rho \pi_\rho(v)\subseteq \vngp_\rho$.  Thus
$\pi_\rho$ is a regular $*$-homomorphism.  

For the second statement, Theorem~\ref{rcmany} shows that if 
$\sigma\in\widehat{\pi_\rho(\D)''},$ then $\sigma\circ E\in\fS(\bh,\pi(\D)'')$.
Remark~\ref{basic}\eqref{basic2} completes the proof.
\end{proof}

\begin{remark}{Remark} We have $\pi_\rho(\D)''\subseteq \vngp_\rho$
  always, but in general they can be very different.  Consider the
  state $\rho=\rho_\infty$ from Example~\ref{ToeplitzAlg}.  Then,
  using the notation from that example, $\{S^n+L_\rho: n\in\bbZ\}$ is
  an orthonormal basis for $\H_\rho$ (where $S^n=S^*{}^{|n|}$ when
  $n<0$).  Note that
  $\pi_\rho(\D)''=\bbC I$, while $\vngp_\rho$ is a MASA.  
\end{remark}

The following
  proposition characterizes when $\pi_\rho(\D)''$ and $\vngp_\rho$
  coincide.  We first make a definition.

\begin{definition}\label{dstab}  Let $(\C,\D)$ be an inclusion and let $f\in
  \Mod(\C,\D)$.
  The
  \textit{$\D$-stabilizer} of $f$ is the set,
$$\dstab(f):=\{v\in \N(\C,\D):  \text{for all
  $d\in\D$, }
f(v^*dv)=f(d) \}.$$ If for every $v\in \dstab(f)$ and
$x\in \C$, we have $f(x)=f(v^*xv)$, then we call $f$ a
\textit{$\D$-rigid state}.
\end{definition}

\begin{proposition}\label{equivMASA}  Let $(\C,\D)$ be a regular
  inclusion, and suppose that $\rho\in\fS(\C,\D)$.
 The following
  statements are equivalent.
\begin{enumerate}
\item
\label{equivMa} $\pi_\rho(\D)''$ is a MASA in $\B(\H_\rho)$.
\item
\label{equivMb}
If $v\in\dstab(\rho)$, then $\rho(v)\neq 0$.
\item
\label{equivMc}
$\rho$ is a pure and $\D$-rigid state. 
\end{enumerate}
\end{proposition}

\begin{proof}
Throughout the proof, we let $\sigma=\rho|_\D$, which by
Proposition~\ref{Dextreme}\eqref{Dext1}, belongs to $\hat{\D}$.

Suppose $\pi_\rho(\D)''$ is a MASA in $\B(\H_\rho)$ and let
$v\in\dstab(\rho)$, so $v\in\Lambda_\rho$ and 
$\beta_v(\sigma)=\sigma$.  Then, using the
notation of Proposition~\ref{compatpure}, we find that
$P_I(v+L_\rho)=v+L_\rho$.  Since $\pi_\rho(\D)''$ is a MASA, $P_I$ is
the orthogonal projection onto $\bbC(I+L_\rho)$.
We conclude that
$v+L_\rho$ is a non-zero scalar multiple of $I+L_\rho$.  Hence
$0\neq \innerprod{v+L_\rho, I+L_\rho}=\rho(v)$.  Thus $v\in
\Delta_\rho$, so statement~\eqref{equivMa} implies
statement~\eqref{equivMb}.

Before proving the next implication, we pause for some generalities.
Suppose that $f$ is any state on $\C$ with the property that
$f|_\D=\sigma$.  If $v\in \N(\C,\D)$ and 
$f(v^*v)=0$, then the Cauchy-Schwartz inequality yields $f(v)=0$.
Also note  that if $v\in\Lambda_\rho$  satisfies
$\beta_v(\sigma)\neq \sigma$, then $f(v)=0$.  
 Indeed, for such  $v\in\Lambda_\rho$,
choose $d\in \D$ so that $\beta_v(\sigma)(d)=0$ and $\sigma(d)=1.$
Then as $v^*dv\in \D$,
$$0=f(v^*dv)=f(v)f(v^*dv)=f(vv^*dv)=f(dvv^*v)=
f(d)f(v)f(v^*v).$$  As $f(d)$ and $f(v^*v)$ are both
non-zero, we conclude that $f(v)=0$.

Now suppose statement~\eqref{equivMb} holds.  We first prove that $\rho$ is
pure.  
So suppose that $t\in [0,1]$ and that for $i=1,2,$ $\rho_i$ are
states on $\C$ and $\rho=t\rho_1+(1-t)\rho_2$.  As $\rho|_\D$ is a
pure state on $\D$, we have $\rho_i|_\D=\sigma.$  We claim that
for every $v\in \N(\C,\D)$,  $\rho_1(v)=\rho_2(v)=\rho(v)$.  By the
previous paragraph, it remains  only to prove this for $v\in \Lambda_\rho$
such that $\beta_v(\sigma)=\sigma$.  So suppose $v$ has this
property.  By the hypothesis in statement~\eqref{equivMb}, $\rho(v)\neq
0$.  Clearly $|\rho_i(v)|\leq
\rho_i(v^*v)^{1/2}=\rho(v^*v)^{1/2}=|\rho(v)|.$  Thus we have
$$t\frac{\rho_1(v)}{\rho(v)}+(1-t)\frac{\rho_2(v)}{\rho(v)}=1,$$ which
expresses $1$ as a convex combination of elements of the closed unit
disk.  Hence $\rho_i(v)=\rho(v)$, establishing the claim. 
By regularity, we conclude that $\rho_1=\rho_2=\rho$, so $\rho$ is a
pure state.

Next, if $v\in\Lambda_\rho$ and $\beta_v(\sigma)=\sigma$, then by
hypothesis, $\rho(v)\neq 0$.  So the final part of
statement~\eqref{equivMc} follows from part~\eqref{Dext0} of
Proposition~\ref{Dextreme}.   Thus statement~\eqref{equivMb}
implies statement~\eqref{equivMc}.

Finally, suppose that statement~\eqref{equivMc} holds.  Let $v,w\in
\Lambda_\rho$ be such that $\beta_v(\sigma)=\beta_w(\sigma)$.  We
shall show that $\{v+L_\rho,w+L_\rho\}$ is a linearly dependent set,
showing that $\K_v$ is one-dimensional.  We have
$\beta_{w^*v}(\sigma)=\sigma=\beta_{v^*w}(\sigma)$, so
$\rho(v^*ww^*v)^{-1/2}w^*v\in \dstab(\rho)$.  By hypothesis,
$\ds\rho(x)=
\frac{\rho(v^*wxw^*v)}{\rho(v^*ww^*v)}$
for every $x\in \C$. Thus if $\eta=\rho(v^*ww^*v)^{-1/2} w^*v+L_\rho,$
we have $\innerprod{\pi_\rho(x)\eta,\eta} =\innerprod{\pi_\rho(x)
  (I+L_\rho), I+L_\rho}$ for every $x\in \C$.  Since $\rho$ is pure,
$\pi_\rho(\C)''=\B(\H_\rho)$, so that for every $T\in\B(\H_\rho)$ we
obtain $$\innerprod{T\eta,\eta} =\innerprod{T(I+L_\rho), I+L_\rho}.$$
Hence $\{\eta,I+L_\rho\}$ is a linearly dependent set.  Thus,
$\{w^*v+L_\rho, I+L_\rho\}$ is linearly dependent. Since both vectors
in this set are non-zero, we find $0\neq \innerprod{w^*v+L_\rho,
  I+L_\rho}=\rho(w^*v).$ Applying part~\eqref{Dext2a0} of
Proposition~\ref{Dextreme} and the Cauchy-Schwartz inequality, we obtain 
$\{v+L_\rho,w+L_\rho\}$ is linearly dependent, as desired.

As $\K_v$ is one-dimensional, Proposition~\ref{compatpure} implies
 that $\pi_\rho(\D)''$ is a
MASA.

\end{proof}

\section{The $\D$-Radical and Embedding Theorems}\label{sectRegEm}

Our purpose in this section is to prove two embedding theorems.  The
first characterizes when a regular inclusion can be 
 regularly embedded  into a regular MASA
inclusion, while the second characterizes when a regular inclusion may
be regularly embedded  into a \cstar-diagonal.

The first of these theorems shows that the obvious necessary condition suffices.

\begin{theorem}\label{regemMASA}  Let $(\C,\D)$ be a regular
  inclusion.  The following statements are equivalent.
\begin{enumerate}
\item[a)]  There exists a regular MASA inclusion $(\C_1,\D_1)$ and a
  regular $*$-monomorphism $\alpha:(\C,\D)\rightarrow (\C_1,\D_1)$.
\item[b)]  The relative commutant $\D^c$ of $\D$ in $\C$ is abelian.
\end{enumerate}
\end{theorem}
\begin{proof}
Suppose that $(\C_1,\D_1)$ is a regular MASA inclusion and
$\alpha:(\C,\D)\rightarrow (\C_1,\D_1)$ is a regular $*$-monomorphism.
Let $(I(\D_1),\iota_1)$ be an injective envelope for $\D_1$ and let
$E_1:\C_1\rightarrow I(\C_1)$ be the pseudo-expectation for 
$\iota_1$. 
  
Observe that $(\D^c,\D)$ is a
regular inclusion, and  $\N(\D^c,\D)\subseteq \N(\C,\D)$.  
Let $\rho\in \fS_s(\C_1,\D_1)$. 
Part~(i) of Theorem~\ref{existcompatA} shows that 
 $\rho\circ\alpha\in
\fS(\C,\D)$, and hence
$\rho\circ\alpha|_{\D^c}\in \fS(\D^c,\D)$.
Theorem~\ref{allunitaries} implies $\rho\circ\alpha|_{\D^c}$ 
is a multiplicative linear
functional on $\D^c$.  By the definition of $\fS_s(\C_1,\D_1)$, we see
that for every $\tau\in
\widehat{I(\D_1)}$, $\tau\circ E_1\circ \alpha|_{\D^c}$ is a multiplicative
linear functional on $\D^c$.  
We conclude $E_1\circ\alpha|_{\D^c}$ is a $*$-homomorphism
of $\D^c$ into $I(\D_1)$.

Let $u\in \D^c$ be a unitary element.  Clearly, $u\in \N(\C,\D)$, so
regularity of $\alpha$ implies that $\alpha(u)\in \N(\C_1,\D_1)$.
Let $J_{\alpha(u)}$ be the ideal of $\D_1$ as defined
in~\eqref{jveedef}.  
Since $E_1(\alpha(u))$ is unitary, Theorem~\ref{uniquecpmap}(c) shows
$\sup_{I(\D_1)}(\iota(J_{\alpha(u)})_1^+)$ is the identity of
$I(\D_1)$.  Hence $J_{\alpha(u)}$ is an essential ideal in $\D_1$.
Then $(\fix{\beta_{\alpha(u)}})^\circ$ is dense in $\hat{\D}_1$ by
Proposition~\ref{Jvee}.  As $\dom(\beta_{\alpha(u)})=\hat{\D}_1$, we
get
$\hat{\D}_1=\overline{(\fix{\beta_{\alpha(u)}})^\circ}=\fix{\beta_{\alpha(u)}}$.
Therefore, $\hat{\D}_1=(\fix\beta_{\alpha(u)})^\circ.$ Another
application of Proposition~\ref{Jvee} gives $J_{\alpha(u)}=\D_1$, and
hence
 $\alpha(u)\in \D_1$.  This shows that the image of the
unitary group of $\D^c$ under $\alpha$ is abelian.  Since $\alpha$ is
faithful, we see that the unitary group of $\D^c$ is abelian, and
hence $\D^c$ is abelian.

For the converse, take $\alpha$ to be the
identity map on $\C$.  Lemma~\ref{abelcom} 
shows that $\alpha: (\C,\D)\rightarrow (\C,\D^c)$ is regular, so 
statement (a) follows from statement
(b).

\end{proof}

\begin{question}\label{regEPembed}  When does a regular inclusion
  regularly embed into a regular EP-inclusion?  
\end{question}
We conjecture that the
  conditions of Theorem~\ref{regemMASA} also characterize when a
  regular inclusion may be regularly embedded into a regular
  EP-inclusion.  Here is an approach to
  this problem.  Suppose $(\C,\D)$ is a regular inclusion with $\D^c$ 
  abelian. Let $\pi:\C\rightarrow \bh$ be a faithful
  representation of $\C$, let $\D_1=\pi(\D^c)''$ and let $\C_1$ be the
  (concrete) \cstaralg\ generated by $\pi(\C)$ and $\D_1$.  Then
  $(\C_1,\D_1)$ is regular, and $\pi : (\C,\D)\rightarrow (\C_1,\D_1)$ is a regular $*$-monomorphism.  
If $\D_1$ is a MASA in $\C_1$, then Theorem~\ref{inMASAincl} shows
that $(\C_1,\D_1)$ is an EP-inclusion.  Unfortunately, we have not
been able to decide whether the faithful represetation $\pi$ can be
chosen so that $(\C_1,\D_1)$ is a MASA inclusion.

We now define a certain ideal, the $\D$-radical of an
inclusion, and show its relevance to embedding regular inclusions
into \cstar-diagonals.

\begin{definition}  For an inclusion $(\C,\D)$, the 
\textit{$\D$-radical of $(\C,\D)$} is the set  
$$\rad(\C,\D):=\{x\in \C: \norm{\pi_\rho(x)}=0 \text{ for all }\rho\in
  \fS(\C,\D)\},$$ provided $\fS(\C,\D)\neq \emptyset$; otherwise
  define $\rad(\C,\D)=\C$.  (Note that $\rad(\C,\D)\neq \C$ whenever
  $(\C,\D)$ is a regular MASA inclusion.)
\end{definition}

When $(\C,\D)$ is a regular  inclusion, we have the following
description of $\rad(\C,\D)$.
\begin{proposition}\label{describerad}
Suppose that $(\C,\D)$ is a regular inclusion.  Then 
$$\rad(\C,\D)=\{x\in\C: \rho(x^*x)=0 \text{ for all } \rho\in\fS(\C,\D)\}.$$
\end{proposition}
\begin{proof}

  Let $J:=\{x\in\C: \rho(x^*x)=0 \text{ for all } \rho\in
  \fS(\C,\D)\}$.  If $x\in \rad(\C,\D)$ and $\rho\in\fS(\C,\D)$, then
  $\rho(x^*x)=\norm{\pi_\rho(x)(I+L_\rho)}=0$, and we find that
  $\rad(\C,\D)\subseteq J$.  For the opposite inclusion, let $x\in J$.
  Part~\eqref{Dext2a1} of Proposition~\ref{Dextreme} and
  Corollary~\ref{invideal} show that $J$ is a closed, two-sided
  ideal of $\C$.  Hence for every $c\in\C$ and $\rho\in\fS(\C,\D)$ we
  have $\rho(c^*x^*xc)=0$, which means that $\pi_\rho(x) =0 $ for
  every $\rho$.  So $x\in \rad(\C,\D)$, showing $\rad(\C,\D)=J$.

\end{proof}

\begin{remark}{Examples} \label{examRad} Here are some examples of 
  the $\D$-radical.
\begin{enumerate}
\item 
  By Proposition~\ref{homobehav}, $\rad(\C,\D)=(0)$ for any  Cartan
  inclusion. 
\item Suppose that $(\C,\D)$ is a regular EP inclusion.  Then
  $\rad(\C,\D)$ is the left kernel of the associated conditional
  expectation $E$, that is,
$\rad(\C,\D)=\{x\in\C: E(x^*x)=0\}$.  This follows from
Propositions~\ref{describerad} and \ref{rcmany}. 
\item \label{examRad3} Suppose that $(\C,\D)$ is an
  inclusion such that $\U(\C)\subseteq \N(\C,\D)$.  Then
  Theorem~\ref{allunitaries} shows that $\fS(\C,\D)$ is the set of
  characters on $\C$.  Since the intersection of the kernels of
  all characters is the commutator ideal,
  it follows from Proposition~\ref{describerad}  that
  $\rad(\C,\D)$ is the commutator ideal of $\C$.

\end{enumerate}
\end{remark}

\begin{question}  \label{radeqL} Observe that when $(\C,\D)$ is a regular MASA
  inclusion, $\rad(\C,\D)\subseteq \L(\C,\D)$, because
  $\fS_s(\C,\D)\subseteq \fS(\C,\D)$.  Is it possible for the inclusion to be
  proper?
\end{question}

\begin{proposition}\label{homobehav}  Let $(\C,\D)$ be a regular
  inclusion and suppose $(\C_1,\D_1)$ is a regular MASA inclusion.
 If
  $\alpha:(\C,\D)\rightarrow (\C_1,\D_1)$ is a regular
and unital  $*$-homomorphism, then $\rad(\C,\D)\subseteq
  \alpha^{-1}(\L(\C_1,\D_1))$. 
\end{proposition}
\begin{proof}
By Theorem~\ref{existcompatA}, $\alpha(\rad(\C,\D))\subseteq \L(\C_1,\D_1)$.
\end{proof}

Notice that when $\L(\C_1,\D_1)=(0)$, Proposition~\ref{homobehav}
implies that $\rad(\C,\D)\subseteq \ker\alpha$.  We have been unable
to decide whether equality holds in general.   However, the following
lemma shows that one can construct a \cstardiag\ and a regular
$*$-homomorphism such that equality holds.

\begin{lemma}\label{kerrad}  Suppose that $(\C,\D)$ is a regular
  inclusion.  Then there exists a \cstardiag\ $(\C_1,\D_1)$ and a
  regular $*$-homomorphism $\alpha:(\C,\D)\rightarrow (\C,\D_1)$ with
  $\ker\alpha= \rad(\C,\D)$.
\end{lemma}
\begin{proof}
  For each $\rho\in\fS(\C,\D)$, let $(\pi_\rho,\H_\rho)$ be the GNS
  representation of $\C$ arising from $\rho$.  Let
  $\H:=\bigoplus_{\rho\in\fS(\C,\D)} \H_\rho$ and let
  $\D_1=\bigoplus_{\rho\in\fS(\C,\D)}\vngp_\rho$, where $\vngp_\rho$
  is as in the statement of Theorem~\ref{atomic}.  As $\vngp_\rho$ is
  an atomic MASA in $\B(\H_\rho)$, we see that $\D_1$ is an atomic
  MASA in $\bh$.  Let $\C_1=\overline{\spn}\N(\bh,\D_1)$.  By
  Theorem~\ref{inMASAincl} and the fact that the expectation onto an
  atomic MASA in $\bh$ is faithful, $(\C_1,\D_1)$
is a \cstar-diagonal.

For each $v\in \N(\C,\D)$, the regularity of $\pi_\rho$
(Theorem~\ref{atomic}) shows that
$\bigoplus_{\rho\in\fS(\C,\D)}\pi_\rho(v)\in \N(\bh,\D_1)$.  Hence
for each $x\in \C$, $\bigoplus_{\rho\in\fS(\C,\D)}\pi_\rho(x)\in\C_1$.
Thus if $\alpha:\C\rightarrow \C_1$ is given by
$\alpha(x)=\bigoplus_{\rho\in\fS(\C,\D)}\pi_\rho(x),$ then $\alpha$ is a
regular $*$-homomorphism.
  By construction,
$\ker\alpha=\rad(\C,\D)$.
\end{proof}

The following is our main embedding result.

\begin{theorem} \label{embedCdiag}
  Let $(\C,\D)$ be a regular  inclusion.  Then there
  exists a \cstar-diagonal $(\C_1,\D_1)$ and a regular
  $*$-monomorphism $\alpha:(\C,\D)\rightarrow (\C_1,\D_1)$ if and only
  if $\rad(\C,\D)=0$.
\end{theorem}
\begin{proof}
Suppose that $(\C_1,\D_1)$ is a \cstar-diagonal and
$\alpha:(\C,\D)\rightarrow (\C_1,\D_1)$ is a regular
$*$-monomorphism.  Since $\L(\C_1,\D_1)=(0)$,
Proposition~\ref{homobehav} gives 
$\rad(\C,\D)\subseteq \ker\alpha =(0).$

The converse follows from Lemma~\ref{kerrad}.

\end{proof}

\begin{corollary}\label{allunitariesrad} Suppose that $(\C,\D)$ is an
  inclusion such that $\U(\C)\subseteq \N(\C,\D)$.  Then 
  there is a
  regular $*$-monomorphism of $(\C,\D)$ into a \cstardiag\ if and only
  if $\C$ is abelian.  
\end{corollary}
\begin{proof}
Since the commutator ideal is the intersection of the kernels of all
multiplicative linear functionals, the result follows directly from
Theorem~\ref{allunitaries} and Example~\ref{examRad}\eqref{examRad3}.
\end{proof}

\section{An Example: Reduced Crossed Products by Discrete
  Groups}\label{secCP}

In this section we consider the regular inclusion
 $(\C,\D)$, where $\C=\D\rtimes_r\Gamma$ is the reduced 
 crossed product of the unital abelian \cstaralg\ $\D=C(X)$ by a discrete
group $\Gamma$
of homeomorphisms of $X$. 
 
The main results of this section are: Theorem~\ref{abelHxabelcom},
which characterizes when the relative commutant $\D^c$ of $\D$ in
$\D\rtimes_r\Gamma$ is abelian in terms of the associated dynamical
system; Theorem~\ref{L4DCP}, which shows that when $\D^c$ is abelian,
$\L(\D\rtimes_r\Gamma,\D^c)=(0)$; and a summary result,
Theorem~\ref{embedcrossedprod} which gives a number of
characterizations for when $(\D\rtimes_r\Gamma,\D)$ regularly embeds
into a \cstardiag.  By choosing the space $X$ and group $\Gamma$
appropriately, the methods in this section can be used to produce an
example of a virtual Cartan inclusion $(\C,\D)$ where $\C$ is not
nuclear, see \cite[Theorem~4.4.3 or
Theorem~5.1.6]{BrownOzawaC*AlFiDiAp}.

Some of the results in this section complement results
from~\cite{SvenssonTomiyamaCoCXCStCrPr}.

We begin by establishing our notation.  This is
standard material, but we include it because there are a  number of variations
in the literature.

Throughout, let $X$ be a compact Hausdorff space, let $\Gamma$ be a
discrete group with unit element $e$ acting on $X$ as homeomorphisms
of $X$.  Thus there is a homomorphism $\Xi$ of $\Gamma$ into the
group of homeomorphisms of $X$, and for $(s,x)\in \Gamma\times X$, we
will write $sx$ instead of $\Xi(s)(x)$.  We will sometimes refer to
the pair $(X,\Gamma)$ as a \textit{discrete dynamical system.}  For
$s\in\Gamma$, let $\alpha_s\in\Aut(C(X))$ be given by
$$(\alpha_s(f))(x)= f(s^{-1}x), \qquad f\in C(X),\, x\in X.$$  

If $Y$ is any set, and $z\in Y$, we use $\delta_z$ to denote the
characteristic function of the singleton set $\{z\}$. 

Let $\D=C(X)$, and let $C_c(\Gamma,\D)$ be the set of all functions
$a:\Gamma\rightarrow \D$ such that $\{s\in\Gamma: a(s)\neq 0\}$ is a
finite set.  We will sometimes write $a(s,x)$ for the value of $a(s)$
at $x\in X$ instead of  $a(s)(x)$.  
Then $C_c(\Gamma,\D)$ is a $*$-algebra under the usual
twisted convolution product and adjoint operation: for $a,b\in C_c(\Gamma,\D)$, 
$$(ab)(t)=\sum_{r\in\Gamma} a(r)\alpha_r(b(r^{-1}t))\dstext{and}
(a^*)(t)=\alpha_t(a(t^{-1}))^*.$$
Let $\C=C(X)\rtimes_{r}\Gamma$ be the reduced crossed product of
$C(X)$ by $\Gamma$.

The group $\Gamma$ is naturally embedded into $\C$ via $s\mapsto
\emb_s$, where $\emb_s$ is the element of $C_c(\Gamma,\D)$ given by
$\emb_s(t)=\begin{cases} 0& \text{if $t\neq s$}\\ I&\text{if
    $t=s$.}\end{cases}$ 
Also, $\D$ is embedded into $C_c(\Gamma,\D)$ via the map $d\mapsto
d\emb_e$ and we  identify $\D$ with its image under this map.  Now
$\emb_s d \emb_{s^{-1}}=\alpha_s(d)$ and $\spn\{d\emb_s: d\in
\D, s\in\Gamma\}$ is norm dense in $\C$, so 
$\{\emb_s:s\in\Gamma\}\subseteq \N(\C,\D)$.  Thus $(\C,\D)$ is a
regular inclusion.

It is well known (see for example, the discussion of crossed products
in~\cite{BrownOzawaC*AlFiDiAp}) that the map $\coexp:C_c(\Gamma,
\D)\rightarrow \D$ given by $\coexp(a)=a(e)$ extends to a faithful
conditional expectation $\coexp$ of $\C$ onto $\D$.  Likewise, the maps
$\coexp_s:C_c(\Gamma,\D)\rightarrow \D$ given by $\coexp_s(a)=a(s)$
extend to norm-one linear mappings $\coexp_s$ 
of $\C$ onto $\D$.  Notice that for
$a\in\C$ and $s\in\Gamma$,
$$\coexp_s(a)=\coexp(a\emb_{s^{-1}}).$$   The maps $\coexp_s$ allow a useful
``Fourier series'' viewpoint 
for elements of $\C$:  $a \sim \sum_{s\in\Gamma} \coexp_s(a)
\emb_s$.

The following is well-known.  We sketch a proof for convenience of the
reader.
\begin{proposition}\label{faithfourier}
If $a\in\C$ and $\coexp_s(a)=0$ for every $s\in \Gamma$, then $a=0$.
\end{proposition}
\begin{proof}
For $a\in C_c(\Gamma,\D)$ and $s,t\in\Gamma$, a calculation shows that
\begin{equation}\label{mvEs}
\coexp_s(\emb_t a \emb_{t^{-1}})=\alpha_t(\coexp_{t^{-1}st}(a));
\end{equation}  a continuity argument then shows that~\eqref{mvEs}
actually holds for every $a\in\C$. 

Let $J=\{a\in \C: \coexp_t(a)=0\; \forall t\in\Gamma\}$.  Clearly $J$ is closed.
Then~\eqref{mvEs} shows that if $a\in J$ and $s\in \Gamma$, then
$\emb_s a \emb_{s^{-1}}\in J$.  Easy calculations now show that if
$d\in \D$, $s\in \Gamma$ and $a\in J$, then $\{da, ad, \emb_s a,
a\emb_s\}\subseteq J$, and by taking  linear combinations and
closures, we find that $J$ is a closed two-sided ideal of $\C$.  Thus,
if $a\in J$, $a^*a\in J$, so that $\coexp_e(a^*a)=\coexp(a^*a)=0$.  Hence $a=0$
by faithfulness of $\coexp$.  This shows that $J=(0)$, completing the proof.
\end{proof}

\begin{definition}  We make the following definitions.  
\begin{enumerate} 
\item For $s\in \Gamma$, let $F_s=\{x\in X: sx=x\}$ be
  the set of fixed points of $s$.
\item For $s\in\Gamma$, let
  $\fF_s=\{f\in\D: \supp(f))\subseteq  F_s^\circ\}$.
  Thus $\{\fF_s:s\in\Gamma\}$ is a family of closed  ideals in $\D$. 

\item  For $x\in X$, let
  $\Gamma^x:=\{s\in \Gamma:  sx=x\}$ be the isotropy group at $x$.
\item For $x\in X$, let  
$H^x:=\{s\in \Gamma: x \in (F_s)^\circ\}.$  We
will call $H^x$ the \textit{germ isotropy group at $x$.}  
\end{enumerate}
\end{definition}

\begin{remark*}{Remarks} We chose the terminology `germ isotropy'
  because $s\in H^x$ if and only if the homeomorphisms $s$ and
  $\text{id}|_X$ agree in a neighborhood of $x$, that is, they have
  the same germ.  It is easy to see that $H^x$ is a group; in
  fact, $H^x$ is a normal subgroup of $\Gamma^x$.  To see that $H^x$
  is a normal subgroup of $\Gamma^x$, fix
  $x\in X$ and let $s\in H^x$.  Then there exists an open neighborhood
  $V$ of $x$ such that $V\subseteq F_s$.  Let $t\in \Gamma^x$ and put
  $W=t^{-1}V$.  Since $tx=x$, $x\in W$.  For $y\in W$, $ty\in
  V$, so $sty=ty$.  Hence $t^{-1}st y=y$.  Therefore,
  $W\subseteq F_{t^{-1}st}$.  As $W$ is open and $x\in W$, we see that
  $x$ belongs to the interior of $F_{t^{-1}st}$, so $t^{-1}st\in H^x$
  as desired.

  Simple examples show the inclusion of $H^x$ in $\Gamma^x$ can be
  proper.  
\end{remark*}

We record a description of the relative commutant of $\D$ in $\C$.

\begin{proposition}\label{comdes}   We have
\begin{align*}
\D^c&=\{a\in\C: \alpha_s(d)\coexp_s(a)=d\coexp_s(a) \text{ for all $d\in\D$ and all
 $s\in\Gamma$}\}\\
&=\{a\in\C: \coexp_s(a)\in\fF_s \text{ for all $s\in\Gamma$}\}.
\end{align*}
\end{proposition}
\begin{proof}
A computation shows that for $a\in\C$, $d\in \D$ and $s\in \Gamma$,
\begin{equation}\label{cy} 
\coexp_s(da-ad)=(d-\alpha_s(d))\coexp_s(a).\end{equation}

Thus if $a\in\D^c$, we obtain $\alpha_s(d)\coexp_s(a)=d\coexp_s(a)$ for every
$d\in \D$ and $s\in \Gamma$.    Conversely, if $\coexp_s(a)\alpha_s(d)=d\coexp_s(a)$
for every $d\in\D$ and $s\in \Gamma$, 
Proposition~\ref{faithfourier} gives 
$a\in\D^c$.

For the second equality, suppose that $a\in\C$ and $\coexp_s(a)\in \fF_s$
for every $s\in \Gamma.$  Since  $\coexp_s(a)$ is supported in
$F_s^\circ$, an examination of \eqref{cy} shows that  $\coexp_s(da-ad)=0$
for every $d\in\D$.  By  Proposition~\ref{faithfourier} again,
 $a\in\D^c$.  For the reverse inclusion, suppose that $a\in
\D^c$.  Then for $d\in\D$ and $s\in\Gamma$,
$0=(d-\alpha_s(d))\coexp_s(a).$  Thus if $x\in X$ and $\coexp_s(a)(x)\neq 0$, we
have $d(x)-d(s^{-1}x)=0$ for every $d\in \D$.  It follows that the
support of $\coexp_s(a)$ is contained in $F_s$.  But $\supp(\coexp_s(a))$ is
open, so the reverse inclusion holds.
 \end{proof}

We now describe a representation useful for establishing certain formulae.
\begin{remark*}{The very discrete representation}
Let $\H=\ell^2(\Gamma\times X)$.  Then  $\{\vdb_{(t,y)}:
(t,y)\in \Gamma\times X\}$ is an orthonormal basis for $\H$.  
For $f\in C(X)$, $s\in \Gamma$, and
$\xi\in\H$, define representations $\pi$ of $C(X)$ and $U$ of $\Gamma$
on $\H$ by
$$(\pi(f)\xi)(t,y)=f(ty)\xi(t,y)\dstext{and}
(U_s\xi)(t,y)=\xi(s^{-1}t,y).$$  In particular,
$$\pi(f)\vdb_{(t,y)}=f(ty)\vdb_{(t,y)}\dstext{and}
U_s\vdb_{(t,y)}=\vdb_{(st,y)}.$$   
The 
 \cstaralg\ generated by the
images of $\pi$ and $U$ is isometrically isomorphic to
the reduced crossed product of $C(X)$ by $\Gamma$ (see~\cite[pages
117-118]{BrownOzawaC*AlFiDiAp}), and hence determines a faithful
representation $\theta:\C\rightarrow \B(\H)$.

A computation shows that  for
 $a\in \C$, $t,r\in\Gamma$ and $x,y\in X$,
$$\innerprod{\theta(a)\vdb_{(t,y)},\vdb_{(r,x)}}=\begin{cases}0&\text{if
    $x\neq y$;}\\ \coexp_{rt^{-1}}(a)(ry)&\text{if $x=y$.}\end{cases}$$
Also 
for $a\in \C$, $t\in \Gamma$ and $y\in X$, 
 we have
\begin{equation}\label{imabasis} \theta(a)\vdb_{(t,y)}=\sum_{s\in\Gamma}
\coexp_s(a)(sty)\vdb_{(st,y)}.
\end{equation}

\end{remark*}

We now define some notation.   
Let $\lambda:\Gamma\rightarrow \B(\ell^2(\Gamma))$ be the left
regular representation, and for $x\in X$, regard $\ell^2(H^x)$ as a
subspace of $\ell^2(\Gamma)$.  Then $C^*_r(H^x)$ is the \cstaralg\
generated by $\{\lambda_s|_{\ell^2(H^x)}: s\in H^x\}$.
Define $V_x:\ell^2(H^x)\rightarrow \H$ by 
$$(V_x\eta)(s,y)=\begin{cases} 0&\text{if $(s,y)\notin
    H^x\times\{x\}$}\\
  \eta(s) &\text{if $(s,y)\in H^x\times\{x\}.$}\end{cases}$$  Then for
$r\in H^x$, we have $V_x\vdb_r=\vdb_{(r,x)}$, so 
$V_x$ is an isometry.

\begin{proposition}\label{evaluate} For $x\in X$ and $a\in \C$, define
$\Phi_x(a):=V_x^*\theta(a)V_x.$ Then $\Phi_x$ is a completely
positive unital mapping of $\C$ onto $C^*_r(H^x)$ and   
  $\Phi_x|_{\D^c}$ is a $*$-epimorphism of $\D^c$ onto $C^*_r(H^x)$.
\end{proposition}
\begin{proof}  Clearly $\Phi_x$ is completely positive and unital. 
For $d\in\D$, $r\in \Gamma$ and $s,t\in H^x$ we have
\begin{align}\label{genA}
\innerprod{\Phi_x(d\emb_r)\vdb_s,\vdb_t}&=
\innerprod{V_x^*\theta(d\emb_r)V_x\vdb_s,\vdb_t}
=\innerprod{\pi(d)U_r\vdb_{(s,x)}, \vdb_{(t,x)}}\\
&=
\innerprod{\pi(d)\vdb_{(rs,x)}, \vdb_{(t,x)}}=
d(rsx)\innerprod{\vdb_{(rs,x)}, \vdb_{(t,x)}}= 
d(rx)\innerprod{\vdb_{rs}, \vdb_{t}}\notag\\
& = d(rx)\innerprod{\lambda_r\vdb_s,\vdb_t}.\notag  
\end{align} 

Hence for every $d\in\D$ and $r\in \Gamma$, \begin{equation}\label{genD}
\Phi_x(d\emb_r)=\begin{cases} 0&\text{if $r\notin H^x$,}\\
d(x)\lambda_r|_{\ell^2(H^x)}&\text{if $r\in H^x$.}
\end{cases}
\end{equation}
Therefore $\Phi_x$ maps a set of generators for
$\C$ into $C^*_r(H^x)$, giving $\Phi_x(\C)\subseteq C^*_r(H^x)$.

To show that  $\Phi_x|_{\D^c}$ is a $*$-homomorphism, it suffices to
prove that the range of $V_x$ is an invariant subspace for $\theta(\D^c)$.  
Note that $\ran(V_x)=\overline{\spn}\{\vdb_{(t,x)}:t\in H^x\}$. 
Let $a\in \D^c$ and fix $t\in H^x$.  
We claim that if  $s\in\Gamma$, $d\in \fF_s$ and  
$stx\in\supp(d)$, then   $s\in H^x$.  Indeed, suppose that 
$stx\in\supp(d)$. As $t\in H^x$, $stx=sx$.  So   
$sx\in F_s^\circ=F_{s^{-1}}^\circ$, which yields $x\in
F_s^\circ$. Thus $s\in H^x$, so the claim holds.  

Next by~\eqref{imabasis} and Proposition~\ref{comdes},
for $t\in H^x$, we have
$$\theta(a)\vdb_{(t,x)}=\sum_{s\in\Gamma}
\coexp_s(a)(stx)\vdb_{(st,x)}
= \sum_{s\in H^x}\coexp_s(a)(stx)\vdb_{(st,x)}\in\ran(V_x),
$$ as desired.  
 It follows that
$\Phi_x|_{\D^c}$ is a $*$-homomorphism.

It remains to show $\Phi_x(\D^c)=C^*_r(H^x)$.
If $s\in H^x$, let $d\in\fF_s$ be such that $d(x)=1$, and put
$a=d\emb_s$.  Then
 \eqref{genD}  shows that 
$\Phi_x(a)=\lambda_s|_{\ell^2(H^x)}$.  By Proposition~\ref{comdes},
$a\in\D^c$, and hence $\Phi_x(\D^c)$ is dense in $C^*_r(H^x)$.   
 Since $\Phi_x|_{\D^c}$ is a homomorphism, it 
has closed range.  Therefore  $\Phi_x(\D^c)=C^*_r(H^x)$.

\end{proof}

Let  
$$\bigoplus_{x\in X}C^*_r(H^x):=\left\{f\in \prod_{x\in X}C^*_r(H^x):
\sup_{x\in X}\norm{f(x)}<\infty\right\}$$ and for $f\in\bigoplus_{x\in X}
C^*_r(H^x)$, define $\norm{f}=\sup_{x\in X}\norm{f(x)}$.  Then with
product, addition, scalar multiplication and involution defined
point-wise, $\bigoplus_{x\in X}C^*_r(H^x)$ is a \cstaralg.

\begin{corollary}\label{comhom} 
  The map $\Phi:
  \C\rightarrow  \bigoplus_{x\in X} C^*_r(H^x)$ given by
  $\Phi(a)(x)=\Phi_x(a)$ is a faithful completely positive unital mapping
  such that $\Phi|_{\D^c}$ is a $*$-monomorphism.
\end{corollary}
\begin{proof}
It follows from the definition of $\Phi_x$ that $\Phi$ is unital and completely
positive. Proposition~\ref{evaluate} shows that  $\Phi|_{\D^c}$ is a
$*$-homomorphism; it remains to check that $\Phi$ is faithful.

For $x\in X$, let $\tr_x$ be the 
the trace on $C^*_r(H^x)$.
For $d\in \D$ and $s\in \Gamma$ equation~\eqref{genD} gives, 
$$\tr_x(\Phi_x(d\emb_s))=\left.\begin{cases}
0&\text{if $s\neq e$}\\
d(x)&\text{if $s=e$}\end{cases}\right\} =\coexp(d\emb_s)(x).$$  This formula extends by
linearity and continuity, so that for
$a\in \C$, $\tr_x(\Phi_x(a))=\coexp(a)(x).$  So if $a\geq
0$ belongs to $\C$ and $\Phi(a)=0$, then
$\coexp(a)=0$, so $a=0$.  Thus, $\Phi$ is faithful.
\end{proof}

\begin{theorem}\label{abelHxabelcom}
The relative commutant, $\D^c$, of $\D$ in $\C$ is abelian if and only
if $H^x$ is an abelian group for every $x\in X$.
\end{theorem}
\begin{proof}
Corollary~\ref{comhom} shows that if $H^x$ is abelian for every $x\in
X$, then $\D^c$ is abelian.

For the converse, we  prove the contrapositive.  Suppose that
$H^x$ is non-abelian for some $x\in X$.  Fix $s,t\in H^x$ so that
$st\neq ts$.  Then $x\in (F_s)^\circ\cap (F_t)^\circ$, so we may find $d\in
\D$ so that $d(x)=1$ and $\overline{\supp}(d)\subseteq (F_s)^\circ\cap
(F_t)^\circ.$ Then for $h\in \D$ and $z\in X$ we have (by examining the
cases $z\in F_s$ and $z\notin F_s$),
$$(\alpha_s(h)(z)-h(z))d(z)=(h(s^{-1}z)-h(z))d(z)=0.$$
Proposition~\ref{comdes} shows 
that $d\emb_s\in\D^c$.  Likewise, $d\emb_t\in\D^c$.  

Then $d\emb_sd\emb_t=d\alpha_s(d)\emb_{st}.$  Note that by
choice of $d$, $s\;
\overline{\supp}(d)=\overline{\supp}(d)$.  For $z\in X$, 
$$\alpha_s(d)(z)=d(s^{-1}z)=\begin{cases}
0& \text{if $z\notin {\supp}(d)$}\\
d(z) & \text{if $z\in \supp(d)$}.\end{cases}$$  Thus,
$\alpha_s(d)=d$, and likewise, $\alpha_t(d)=d$.  Therefore, 
$$(d\emb_s)( d\emb_t)=d^2\emb_{st}\neq d^2\emb_{ts}=(d\emb_t)
(d\emb_s),$$ so $\D^c$ is not abelian.

\end{proof}

\begin{proposition}\label{L4count} Let $(X,\Gamma)$ be a discrete  
  dynamical system such that for each $x\in X$, the germ isotropy
  group $H^x$ is abelian.   Let 
  $\Gamma_1\subseteq \Gamma$ be a subgroup of $\Gamma$, set
  $$\C_1:=\D\rtimes_r\Gamma_1,\dstext{and let}\D_1=\{x\in \C_1: xd=dx \text{ for
    all } d\in \D\}.$$  Then $(\C_1,\D_1)$ is a regular MASA
  inclusion and $\C_1\cap \L(\C,\D^c)\subseteq \L(\C_1,\D_1)$.  
\end{proposition}
\begin{proof}
As $\D_1\subseteq \D^c$, $\D_1$ is abelian, and as $\D_1$ is the
relative commutant of $\D$ in $\C_1$, $(\C_1,\D_1)$ is a regular
MASA inclusion.  Let $\epsilon:\C_1\rightarrow \C$ be the inclusion
map.   Notice that each map in the diagram,
$$\xymatrix{
(\C_1,\D_1)\ar[r]^\epsilon& (\C,\D_1)\ar[r]^{\text{id}}& (\C,\D^c)}$$ is a
regular map.  The first is clearly regular, while the regularity of
the second follows from the fact that the relative commutant of
$\D_1$ in $\C$ is $\D^c$ and an application of Lemma~\ref{abelcom}.
Therefore, $\epsilon:(\C_1,\D_1)\rightarrow (\C,\D^c)$ is a regular
$*$-monomorphism. 
An application of Corollary~\ref{LRegSub} completes the proof.
\end{proof}

\begin{remark}{Notation} When $G$ is an abelian group with dual group
$\hat{G}$, we use the notation $\innerprod{g,\gamma}$ to denote the
value of $\gamma\in \hat{G}$ at $g\in G$.   Also, we will 
identify $C^*(G)$ with $C(\hat{G})$; lastly, for $\gamma\in \hat{G}$
and $a\in C^*(G)$, we will write $\gamma(a)$ instead of
$\hat{a}(\gamma)$. 
\end{remark}

\begin{theorem}\label{L4DCP}  Suppose that $(X,\Gamma)$ is a discrete
  dynamical system such that for each $x\in X$, the germ isotropy
  group $H^x$ is abelian.  Then $\L(\C, \D^c)=(0).$
\end{theorem}
\begin{proof}
First assume that $\Gamma$ is a countable discrete group.
Let $$P=
\prod_{x\in X}\widehat{H^x}$$ be the Cartesian product of the dual
groups.  Denote by $p(x)$ the ``$x$-th component'' of $p\in P$. 
  For $(x,p)\in X\times P$,   
define a state  $\ha{x,p}$ on $\C$ by 
$$\ha{x,p}(a)=p(x)\left(\Phi_{x}(a)\right) \quad\text{(here $a\in\C$)}, \dstext{and let} 
A:=\{\ha{x,p}: (x,p)\in X\times P\}.$$
Corollary~\ref{comhom} shows that the restriction of $\ha{x,p}$ to
$\D^c$ is a multiplicative linear functional, so in particular,
$A\subseteq \Mod(\C,\D^c)$.  

For each $s\in \Gamma$, let 
$$X_s:=
\left(X\setminus F_s\right)\cup F_s^\circ.$$  Then $X_s$  is a
dense, open subset of $X$.  Set
$$Y:=\bigcap_{s\in \Gamma}X_s\dstext{and} B:=\{\ha{y,p}: (y,p)\in
Y\times P\}.$$ 
Our goal is to
show that 
\begin{equation}\label{Bcompat}
\overline{B}\subseteq \fS_s(\C,\D^c).
\end{equation}

Fix  $(y,p)\in Y\times P$, and suppose that $\tau\in\Mod(\C,\D^c)$
satisfies $\ha{y,p}|_{\D^c}=\tau|_{\D^c}$.  We claim that $\ha{y,p}=\tau$.
To see this, it suffices to show that for each $s\in \Gamma$,
$\ha{y,p}(\emb_s)=\tau(\emb_s)$.  Given $s\in \Gamma$, if $sy\neq y$, we
may choose $d\in\D$ so that $d(sy)=1$ and $d(y)=0$.  Using~\eqref{genD}, 
$$\ha{y,p}(d)=p(y)(\Phi_y(d))=0\dstext{and}\ha{y,p}(\emb_s^*d\emb_s) =
\ha{y,p}(\alpha_{s^{-1}}(d))=p(y)(d(sy)I)=1.$$ 
Then
$$
\ha{y,p}(\emb_s)= \ha{y,p}(\emb_s)\ha{y,p}(\emb_s^*d\emb_s)=
\ha{y,p}(\emb_s(\emb_s^*d\emb_s))=\ha{y,p}(d)\ha{y,p}(\emb_s)=0.$$  Likewise,
$\tau(\emb_s)=0$, so $\tau(\emb_s)=\ha{y,p}(\emb_s)=0$ when $y\notin F_s$.

  On the other hand, if $sy=y$, then
as $y\in X_s$, we have $y\in F_s^\circ$, so $s\in H^y$.  Choose $d\in
\D$ so that $\hat{d}(y)=1$ and $\supp\hat{d}\subseteq F_s^\circ$.
Then $d\emb_s\in \D^c$, so that 
$$\ha{p,y}(\emb_s)=\ha{y,p}(d\emb_s)=\tau(d\emb_s)=\tau(\emb_s).$$   Therefore,
$\ha{p,y}=\tau$. 

Let $\unistex(\C,\D^c)=\{\tau\in \Mod(\C,\D^c): \tau|_{\D^c}\text{
  extends uniquely to } \C\}$.  The previous paragraph shows that
$B\subseteq \unistex(\C,\D^c)$.  By Theorem~\ref{minonto},
$\overline{B}\subseteq\overline{\unistex(\C,\D^c)}=\fS_s(\C,\D^c)$,
so~\eqref{Bcompat} holds.

Suppose now that $a\in \L(\C,\D^c)$.  
Then for every $\rho\in
\fS_s(\C,\D^c)$, we have $\rho(a^*a)=0$.  In particular, for each
$(y,p)\in Y\times P$, 
$$0=\ha{y,p}(a^*a)=p(y)\left(\Phi_{y}(a^*a)\right).$$   
Now $\widehat{H^y}=\{p(y):p\in P\}$, so holding $y$ fixed and varying
$p$, yields $\Phi_{y}(a^*a)=0$.
Hence, we have $\coexp_e(a^*a)(y)=\coexp(a^*a)(y)=0$ for every $y\in
Y$. 
By
Baire's theorem,  $Y$ is dense in $X$, so that $\coexp(a^*a)=0$.
Since $\coexp$ is faithful, $a=0$.  This gives the theorem in the case
when $\Gamma$ is countable.

We turn now to the general case.  Let $\Gamma$ be any discrete group
and suppose $a\in \L(\C,\D^c)$.  Then there exists a countable subgroup
$\Gamma_1\subseteq \Gamma$ such that $a\in \D\rtimes_r\Gamma_1$. Put
$\C_1=\D\rtimes_r\Gamma_1$ and let $\D_1=\{x\in \C_1: dx=xd \text{ for all }
d\in \D\}$ be the relative commutant of $\D$ in $\C_1$.  By
Proposition~\ref{L4count}, we have $a\in \L(\C_1,\D_1)=(0).$  This
completes the proof.
\end{proof}

We collect  the main results of this section into a main theorem.
\begin{theorem}\label{embedcrossedprod}  Let $X$ be a compact Hausdorff
  space and let $\Gamma$ be a discrete group acting as homeomorphisms
  on $X$.  Let $\C=C(X)\rtimes_r\Gamma$ and $\D=C(X)$.  
The following statements are equivalent.
\begin{enumerate}
\item[a)]  For every $x\in X$, the germ isotropy group $H^x$ is
  abelian;
\item[b)] The relative commutant, $\D^c$, of $\D$ in $\C$ is abelian;
\item[c)]  $\L(\C,\D^c)=(0)$;
\item[d)] $(\C,\D)$ regularly embeds into a \cstardiag.
\end{enumerate}
\end{theorem}
\begin{proof} Theorem~\ref{abelHxabelcom} gives the equivalence of (a)
  and (b) and Theorem~\ref{L4DCP} shows that (a) implies (c).   

Suppose (c) holds.  Since
  $\rad(\C,\D^c)\subseteq \L(\C,\D^c)$, Theorem~\ref{embedCdiag} shows
  that $(\C,\D^c)$ regularly embeds into a \cstardiag.
  Lemma~\ref{abelcom} shows that the inclusion map of $(\C,\D)$ into
  $(\C,\D^c)$ is a regular $*$-monomorphism.  Composing the embedding
  of $(\C,\D^c)$ into a \cstardiag\ with the inclusion map shows that
  (c) implies (d).  

Finally, if (d) holds, Theorem~\ref{regemMASA} shows that $\D^c$ is
abelian, so (d) implies (b).
\end{proof}

\section{A Description of $\mathbf{\fS(\C,\D)}$ for a Regular MASA
  Inclusion}\label{desCSregMASA}

For a regular inclusion, $(\C,\D)$, the $\D$-radical, $\rad(\C,\D)$ 
 is the
intersection of the left kernels of compatible states, and when
$(\C,\D)$ is a regular MASA inclusion, 
$\L(\C,\D)$ is the intersection of the left kernels of strongly
compatible states.  Question~\ref{radeqL} asks
 whether it is possible for these
ideals to be distinct.   In order to make progress on this question,
it seems likely that a description of $\fS(\C,\D)$ will be useful. 
The purpose of this section is to provide this description.

The description is in terms of groups which are determined locally by
the action of $\N(\C,\D)$ and certain
positive definite forms on these groups.

We begin with some generalities on $\bbT$-groups,
and describe a class of positive-definite functions on $\bbT$-groups 
which behave like compatible states.  
The following is more-or-less standard.
\begin{definition}
Let $G$ be a locally compact group with identity element $1$, and let
$U$ be the connected component of the identity.  We say
that $G$ is a \textit{$\bbT$-group} if $U$  is
clopen, isomorphic and homeomorphic to $\bbT$, 
and contained in the center of $G$.  A subgroup $H$ of $G$ is a
\textit{$\bbT$-subgroup of $G$} if $H$ contains $U$. 
When $G$ is a $\bbT$-group, we will always identify $U$ with $\bbT$
(and so will write $G/\bbT$ instead of $G/U$).     

Equivalently, a $\bbT$-group is a central extension of $\bbT$ by a
discrete group $K$,  
$$1\rightarrow \bbT\hookrightarrow G\overset{q}{\twoheadrightarrow} 
K\rightarrow 1.$$
If $f:G\rightarrow \bbT$ is a continuous homomorphism, we define the
\textit{index} of $f$ to be the unique integer $n$ for which
$f(\lambda)=\lambda^n$ for every $\lambda\in\bbT$.

As a set, $G$ may
be identified with $\bbT\times K$, and the topology on $G$ is the product of
the usual topology on $\bbT$ with the discrete topology on $K$.  Also,
the Haar measure on $G$ is the product of Haar measure on $\bbT$ with
the counting measure on $K$.

The $\bbT$-subgroups of $G$ are in one-to-one
correspondence with the subgroups of $K$: if $H$ is a $\bbT$-subgroup
of $G$, then $q(H)$ is a subgroup of $K$ and for any subgroup
$\Gamma$ of $K$, $q^{-1}(\Gamma)$ is a $\bbT$-subgroup of $G$.

We also recall that a function $f:G\rightarrow \bbC$ is
\textit{positive definite} if $f$ is continuous, and 
if for every $n\in\bbN$ and 
 $g_1,\dots,g_n\in
G$, the $n\times n$ complex matrix, $A:=(f(g_i^{-1}g_j))_{i,j}$ satisfies
$A\geq 0$.  
\end{definition}

\begin{proposition}\label{posdef}  Let $G$ be a $\bbT$-group with identity
  $1$. 
\begin{enumerate}
 \item\label{posdef1} Let 
  $f:G\rightarrow \bbC$ be a  positive-definite function such that
  $f(1)=1$ and which  satisfies $|f(g)|\in \{0,1\}$
  for every $g\in G$.  Set $$H:=\{g\in G: f(g)\neq 0\}.$$
 Then
\begin{enumerate}
\item[a)]
 $f(g_1g_2)=f(g_1)f(g_2)$ for any $g_1, g_2\in G$ such that
 $H\cap \{g_1,g_2\}\neq \emptyset$; and  
\item[b)]  $H$ is a $\bbT$-subgroup of $G$ and $f|_H$
is a continuous homomorphism of $H$ onto $\bbT$.   
\end{enumerate}
\item \label{posdef2}
Let $H\subseteq G$ be a $\bbT$-subgroup and suppose $\phi:H\rightarrow
\bbT$ is a continuous homomorphism. Define $f: G\rightarrow
\bbC$  by
$$f(g)=\begin{cases} \phi(h) & \text{if $h\in H$}\\
0&\textit{if $h\notin H$}.\end{cases}$$  Then $f$ is a positive
definite function on $G$ such that $f(1)=1$ and for every $g\in G$,
$|f(g)|\in\{0,1\}$. 
\end{enumerate}
\end{proposition}
\begin{proof}
Suppose $f$ satisfies the hypotheses in~\eqref{posdef1}.
  Since $f(1)=1$ and $f$ is positive definite,
we have $\overline{f(g)}=f(g^{-1})$ for every $g\in G$.   Continuity
of $f$ and connectedness of $\bbT$ yield $\bbT\subseteq H$.  

Recall that
if $\H_1$ and $\H_2$ are Hilbert spaces, $A\in\B(\H_1)$,
$C\in\B(\H_2)$, and $B\in\B(\H_2,\H_1)$ with $A$ invertible, then
\begin{equation}\label{posmatcr}
\begin{pmatrix} A&B\\ B^*&C\end{pmatrix}\geq 0\dstext{if and only if $A\geq
0$, $C\geq 0$ and $C-B^*A^{-1}B\geq0$.}
\end{equation}

Let $g_1, g_2\in G$.   Then the positive definiteness of $f$ (using the group elements $h_1=1, h_2=g_1^{-1}$, and $h_3=g_2$) implies that 
$$\begin{pmatrix}
1& f(g_1^{-1}) & f(g_2)\\
 f(g_1)& 1& f(g_1g_2)\\
f(g_2^{-1})&f(g_2^{-1}g_1^{-1})&1\end{pmatrix} \geq 0.$$
Applying~\eqref{posmatcr} with $A=(1)$, $B=\begin{pmatrix} f(g_1^{-1}) &
  f(g_2)\end{pmatrix}$,  and $C=\begin{pmatrix} 1& f(g_1g_2)\\
  f(g_2^{-1}g_1^{-1})&1\end{pmatrix}$ gives
$$
0\leq \begin{pmatrix} 1-f(g_1)f(g_1^{-1})& f(g_1g_2)-f(g_1)f(g_2)\\
f(g_2^{-1} g_1^{-1})-f(g_2^{-1})f(g_1^{-1})&
1-f(g_2^{-1})f(g_2)\end{pmatrix}=:M.$$

Suppose now that $g_2\in H$, that is, 
 $f(g_2)\neq 0$.  Then
$f(g_2)f(g_2^{-1})=1$, so that   $$M=\begin{pmatrix}
1-f(g_1^{-1})f(g_1)&f(g_1g_2)-f(g_1)f(g_2)\\ f(g_2^{-1}
g_1^{-1})-f(g_2^{-1})f(g_1^{-1})&0\end{pmatrix}.$$  
But then $0\leq \det(M)=-|f(g_1g_2)-f(g_1)f(g_2)|^2$, so
$f(g_1g_2)=f(g_1)f(g_2)$.

The case when $f(g_1)\neq 0$ is the same.  Thus when
$H\cap \{g_1,g_2\}\neq \emptyset$ we obtain,
$$f(g_1g_2)=f(g_1)f(g_2).$$

The facts that $H$ is a $\bbT$-subgroup of $G$ and $f|_H: H\rightarrow
\bbT$ is a continuous homomorphism are now apparent.

Turning now to statement~\eqref{posdef2}, 
let $H$ be a $\bbT$-subgroup of $G$, and
$\phi$ a continuous homomorphism of $H$ into $\bbT$.  Let
$f:G\rightarrow \bbT$ be given
as in the statement.  The continuity of $f$ is clear, as is the fact
that $f(1)=1$ and $|f(g)|\in\{0,1\}$ for every $g\in G$.   To show $f$
is positive definite, let $g_1,\dots, g_n\in G$.  Since
$\phi(h)=\overline{\phi(h^{-1})}$, it follows that 
$f(g_i^{-1}g_j)=\overline{f(g_j^{-1}g_i)}$.   Put $X=\{1,\dots,
n\}$.  Define an equivalence relation $R$ on $X$ by $(i,j)\in R$ if
and only if $g_i^{-1}g_j\in H$, and let $X/R$ be the set of
equivalence classes.  Let $q: X\rightarrow X/R$ be the
map which sends $j\in X$ to its equivalence class, and let 
$u: X/R\rightarrow X$ be a section for $q$. 
Let $\delta_{x,y}$ be the Kronecker delta function on $X/R$ 
and for $x\in X/R$,  define
$$c_x:=\begin{pmatrix} f(g_{u(q(1))}^{-1} g_1)\delta_{q(1),x} & \dots & f(
   g_{u(q(n))}^{-1} g_n)\delta_{q(n),x} \end{pmatrix}.$$  
Then $c_x^*c_x$ is an $n\times n$ matrix whose $i,j$-th entry is
$$f(g_i^{-1}g_{u(q(i))})f(g_{u(q(j))}^{-1}g_j)
\delta_{q(i),x}\delta_{q(j),x}=\begin{cases} f(g_i^{-1}g_j)&\text{if
    $q(i)=q(j)=x$}\\
 0&\text{otherwise.}\end{cases}$$
Hence the $i,j$-th entry of $\sum_{x\in X/R} c_x^*c_x$ is
$f(g_i^{-1}g_j)$ if $(i,j)\in R$ and $0$ otherwise.  Therefore 
$$(f(g_i^{-1}g_j))_{i,j\in X}=\sum_{x\in X/R} c_x^*c_x\geq 0,$$
as desired.

\end{proof}

\begin{corollary}\label{pdindex}  Let $f$ be a continuous positive
  definite function on the $\bbT$-group $G$ such that
  $|f(g)|\in\{0,1\}$ for every $g\in G$.  Then there exists
  $p\in\bbZ$ such that for every $\lambda\in \bbT$ and $g\in G$,
$$f(\lambda g)=\lambda^pf(g).$$
\end{corollary} 
\begin{proof}
The set $\{g\in G: f(g)\neq 0\}$ contains $\bbT$, and
Proposition~\ref{posdef} shows the restriction of $f$ to $\bbT$ is a
character on $\bbT$.  So there  exists $p\in\bbZ$ such that
$f(\lambda)=\lambda^p$ for every $\lambda\in\bbT$.  The corollary now
follows from another
application of Proposition~\ref{posdef}.
\end{proof}

\begin{definition}  Given a $\bbT$-group $G$, call a 
  positive-definite function $f$ on $G$ satisfying $f(1)=1$ and
  $|f(g)|\in \{0,1\}$ a 
\textit{pre-homomorphism}.

 We  will call the number $p$ appearing in
  Corollary~\ref{pdindex} the \textit{index} of $f$, and will denote
  it by $\ind(f)$.  Finally, the group $H:=\{g\in G: f(g)\neq 0\}$
  will be called the \textit{supporting subgroup} for $f$, and will be
  denoted by $\supp(f)$.
\end{definition}

\begin{remark*}{Notation} Some notation will be useful. 
\begin{enumerate}
\item
Let $(\C,\D)$ be an inclusion.  For
  $\sigma\in \hat{\D}$, let 
\begin{align*}
\fS(\C,\D,\sigma)&:=\{\rho\in\fS(\C,\D): \rho|_\D=\sigma\}, &\text{and}\\
\fS_s(\C,\D,\sigma)&:=\{\rho\in\fS_s(\C,\D): \rho|_\D=\sigma\}.
\end{align*}
\item For any $\bbT$-group $G$, let 
$\tog(G)$ denote the set of all \phom s $f:G\rightarrow
\bbT\cup\{0\}$ with $\ind(f)=1$.
\item Finally, recall the  seminorms, $B_{\rho,\sigma}$ on $\C$ 
from~\cite[Definition~2.4]{DonsigPittsCoSyBoIs}:
for each $\rho,\sigma\in\hat{\D}$, the 
seminorm $B_{\rho,\sigma}$ is defined on $\C$ by
$$B_{\rho,\sigma}(x):=\inf\{\norm{dxe}: d, e\in \D,
  \rho(d)=\sigma(e)=1\}  \qquad (x\in \C).$$
We shall require these seminorms for $\rho=\sigma$; however, instead of
writing $B_{\sigma,\sigma}$ we shall write $B_\sigma$.
\end{enumerate}
\end{remark*}

Associated to each regular inclusion $(\C,\D)$ and $\sigma\in
\hat{\D}$ is a certain $\bbT$-group, denoted $H_\sigma/R_1$,
 which we now construct.  We produce a
 distinguished unitary representation  $T$ of $H_\sigma/R_1$. 
Our goal is to exhibit a bijection between elements of
$$\{f\in \tog(H_\sigma/R_1): f \text{ determines a state on the
  \cstaralg\ generated by $T(H_\sigma/R_1)$}\}$$  and
$\fS(\C,\D,\sigma)$.

\begin{definition}
Let $(\C,\D)$ be an inclusion, and
  let $\sigma\in\hat{\D}$.  
\begin{enumerate}
\item Define 
$$H_\sigma:= 
\{v\in\N(\C,\D): \sigma(v^*dv)=\sigma(d) \text{ for every }d\in\D\}.$$
We remark that $H_\sigma$  is the set which arises when considering the
$\D$-stabilizer of $\rho\in\Mod(\C,\D)$, where $\sigma=\rho|_\D$ (see
Definition~\ref{dstab}). 
However, there our interest was in a particular extension of $\sigma$,
while here we do not wish to specify the extension.

Notice that for $v\in H_\sigma$, we have $\sigma(v^*v)=1$, and that
$H_\sigma$ is closed under the adjoint operation:  replace $d$ by
$v^*dv$ in the definition.  Furthermore, it is easy to see that
$H_\sigma$ is a $*$-semigroup.

\item Let $\Lambda\subseteq\bbT$ be a subgroup, (we write the
  product multiplicatively).  Define 
$$R_\Lambda:=\{(v,w)\in H_\sigma\times H_\sigma: B_\sigma(\lambda
I-w^*v)=0 \text{ for some } \lambda\in
\Lambda\}.$$  When $\Lambda=\{1\}$, we write $R_1$ instead of $R_{\{1\}}$.
\end{enumerate}
\end{definition}

\begin{lemma}\label{Lambda1norm} Let $(\C,\D)$ be an 
  inclusion, let $\sigma\in\hat{\D}$, and let $v,w\in H_\sigma$.
  Then
$$B_\sigma(v)=1\dstext{and} B_\sigma(v-w)=B_\sigma(I-v^*w)=B_\sigma(I-vw^*).$$
\end{lemma}  
\begin{proof}  Note that for any $h\in\D$,
  $|\sigma(h)|=\inf\{\norm{d^*hd}: d\in\D, \sigma(d)=1\}$.  Hence
given $\eps>0$, and $e\in \D$ with $\sigma(e)=1$, we may find $d\in\D$
with $\sigma(d)=1$ and 
$$\left| \norm{d^*v^*e^*evd}-\sigma(v^*e^*ev)\right| <\eps.$$  Since
$\sigma(v^*e^*ev)=\sigma(e^*e)=1$, we  obtain
$1-\eps < \norm{evd}^2<1+\eps$.  Hence $1-\eps<B_\sigma(v)^2<1+\eps,$
and the fact that $B_\sigma(v)=1$ follows.

Next, let $x, d,e\in\D$ with $\sigma(x)=\sigma(d)=\sigma(e)=1$.  Since
$\sigma(xvdv^*)=1$, 
we
have \begin{align*}
B_\sigma(v-w)&\leq \norm{(xvdv^*)(v-w)e} =\norm{xv(dv^*ve-dv^*we)}\\
&\leq \norm{xv}\, \norm{dv^*ve-dv^*we}\\
&\leq \norm{xv}\left[ \norm{dv^*ve-de}+\norm{d(I-v^*w)e}\right].
\end{align*}
It follows that 
$$B_\sigma(v-w)\leq B_\sigma(I-v^*w).$$
A similar argument using multiplication on the right by $w^*ewx$ gives
$B_\sigma(I-v^*w)\leq B_\sigma(w^*-v^*).$  But
$B_\sigma(v^*-w^*)=B_\sigma(v-w)$, so we obtain $B_\sigma(v-w)=B_\sigma(I-v^*w).$
Finally, $B_\sigma(v-w)=B_\sigma(v^*-w^*)=B_\sigma(I-vw^*).$

\end{proof}

\begin{proposition}\label{charer}  Let $(\C,\D)$ be a regular MASA 
inclusion, $\sigma\in\hat{\D}$ and suppose that $v\in H_\sigma$.  Let
$\Lambda\subseteq \bbT$ be a subgroup.  
The following statements are equivalent:
\begin{enumerate}
\item
for some $\lambda\in\Lambda$, $B_\sigma(\lambda I-v)=0$;
\item
there exists $\lambda\in \Lambda$ such that 
$f(v)=\lambda$ whenever $f\in
\Mod(\C,\D,\sigma)$;
\item
there exists $\lambda\in \Lambda$ such that 
$\rho(v)=\lambda$ whenever $\rho\in
\fS_s(\C,\D,\sigma)$;
\item
$\sigma\in (\fix\beta_v)^\circ$ and $\hat{v}(\sigma)\in \Lambda$
(where $\hat{v}$ is as in Remark~\ref{vhat});
\item
there exists $\lambda\in \Lambda$ and  $h,k \in \D$ such that 
$\sigma(h)=1= 
\sigma(k)$
  and $vh=\lambda k$.
\end{enumerate}
\end{proposition}
\begin{proof}
$(1)\Rightarrow (2)$.  If $f\in \Mod(\C,\D,\sigma)$ and $x\in\C$, 
then $|f(x)|=|f(dxe)|\leq \norm{dxe}$ whenever
$d,e\in\D$ and $\sigma(d)=\sigma(e)=1$.  Thus, 
$|f(x)|\leq B_\sigma(x)$ for every $x\in\C$.  The implication
$(1)\Rightarrow (2)$ follows.  

$(2)\Rightarrow (3)$ is trivial.

$(3)\Rightarrow (4)$. 
Suppose $\lambda\in\Lambda$ and that $\rho(v)=\lambda$ for every
$\rho\in\{f\in\fS_s(\C,\D): f|_\D=\sigma\}$.  By Lemma~\ref{vactmod},
$\sigma\in \fix\beta_v$.    
To show that $\sigma\in (\fix\beta_v)^\circ$, we argue by
contradiction.  So suppose that $\sigma\in \fix\beta_v\setminus 
(\fix\beta_v)^\circ$.  Then every
neighborhood of $\sigma$ contains an element in $\hat{\D}\setminus
\fix\beta_v$.   Hence we may find a net $(\sigma_s)$ in $\hat{\D}$ such that
$\sigma_s\rightarrow \sigma$ and such that
$\sigma_s\notin\fix\beta_v$.  By Theorem~\ref{natext}, the
 restriction map $f\mapsto
f|_\D$ from
$\fS_s(\C,\D)$ to $\hat{\D}$ is onto.  Thus we may choose $f_s\in \fS_s(\C,\D)$
such that $f_s|_\D=\sigma_s$.  By passing to a subnet if necessary, we
may assume that $f_s$ converges to a state $f$.  Theorem~\ref{natext} shows
that 
$\fS_s(\C,\D)$ is closed,
so  $f\in \fS_s(\C,\D)$.  Clearly $f|_\D=\sigma$.  Lemma~\ref{vactmod}
gives $f_s(v)=0$ for every $s$, so 
$0\neq \lambda =f(v)=\lim_s f_s(v)=0$.  This is absurd, so we conclude
that $\sigma\in (\fix\beta_v)^\circ$.

Let $\rho\in\fS_s(\C,\D,\sigma)$.   Then for 
$h\in J_v$, we have 
\begin{equation}\label{vopcharer}
\hat{v}(\sigma)\sigma(h)=\sigma(vh)=\rho(vh)=\rho(v)\sigma(h).
\end{equation}
By Proposition~\ref{Jvee}, $\sigma|_{J_v}\neq 0$.   Statement 
$(4)$ now follows from equation~\eqref{vopcharer}.

$(4)\Rightarrow (5)$.  Let $\lambda=\hat{v}(\sigma)$ and choose $h\in
J_v$ 
such that $\sigma(h)=1$.
Then $\sigma(vh)=\hat{v}(\sigma)$.  Put $k=\overline{\hat{v}(\sigma)}vh$.

$(5)\Rightarrow (1)$.  Let $h,k\in \D$ be chosen so that
$\sigma(h)=\sigma(k)=1$ and $vh=\lambda k$.   Then 
$$B_\sigma(\lambda I-v)\leq \inf_{\{d\in\D:\sigma(d)=1\}}
\norm{d(\lambda I-v)hd}=|\sigma(\lambda h-vh)|=0.$$  
Thus (1) holds.
 
\end{proof}

We next observe 
that $B_\sigma$ gives the quotient norm on the quotient of
$\overline{\spn}H_\sigma$ by the ideal generated by $\ker\sigma$.
\begin{proposition}\label{Bsigquot}
Let $(\C,\D)$ be a regular inclusion and suppose $\sigma\in\hat{\D}$.
Let $\C_\sigma=\overline{\spn}H_\sigma$.  Then $(\C_\sigma, \D)$ is a
regular inclusion.  If  $\fI_\sigma$ is the closed, two-sided ideal
of $\C_\sigma$ generated by $\ker\sigma$, then $B_\sigma$ vanishes on
$\fI_\sigma$, and for $x\in \C_\sigma$, 
$$B_\sigma(x)=\inf\{\norm{x+j}: j\in \fI_\sigma\}.$$  

Moreover, the following statements hold.
\begin{enumerate}
\item If
$\rho$ is a state on $\C_\sigma$ which annihilates $\fI_\sigma$, then
$\rho$ extends uniquely to a state $\tilde{\rho}$ on $\C$.  When
$\rho\in\fS(\C_\sigma,\D)$ annihilates $\fI_\sigma$,  $\tilde{\rho}\in\fS(\C,\D)$.
\item The map $H_\sigma\ni u\mapsto u+\fI_\sigma\in
  \C_\sigma/\fI_\sigma$ is a $*$-homomorphism of the $*$-semigroup
  $H_\sigma$ into the unitary group of 
  $\C_\sigma/\fI_\sigma.$ 
\end{enumerate}
\end{proposition} 

\begin{proof}
Since $H_\sigma$ is closed under multiplication and the adjoint map,
we see that $\C_\sigma$ is a \cstaralg.

For $d\in\D$, $d=\lim_{t\rightarrow 0} tI +d$.
But for all sufficiently small $t\neq 0$, $(t+\sigma(d))^{-1}(tI+d)\in
H_\sigma$, so $d$ is a limit of scalar multiples of
elements of $H_\sigma$. So $\D\subseteq \C_\sigma$.  Therefore
$(\C_\sigma,\D)$ is a regular inclusion.

Next, suppose that $x\in\C$ and $d\in\ker\sigma$.  Then
$B_\sigma(xd)=B_\sigma(dx)=0$.   When $x=v\in
H_\sigma$ and $y\in \C_\sigma$, 
$$B_\sigma(vdy)=B_\sigma(vv^*vdy)=B_\sigma(vdv^*vy)= 
|\sigma(vdv^*)|B_\sigma(vy)=0.$$  It follows
that when $x\in\spn H_\sigma$, we have $B_\sigma(xdy)=0$.  Taking
closures we obtain $B_\sigma(xdy)=0$ when $x, y\in \C_\sigma$ and
$d\in\ker\sigma$.  Therefore, $B_\sigma(z)=0$ for every $z\in
\fI_\sigma$.  

This gives $B_\sigma(x)\leq B_\sigma(x+j)+B_\sigma(j)=B_\sigma(x+j) 
\leq \norm{x+j}$ for every
$x\in\C_\sigma$ and $j\in \fI_\sigma$. Thus, $B_\sigma(x)\leq
\dist(x,\fI_\sigma)$.

Notice that if $\rho$ is a state on $\C_\sigma$ which
annihilates $\fI_\sigma$, then $\rho|_\D$ annihilates $\ker\sigma$, so
$\rho|_\D=\sigma$.  Any
extension of $\sigma$ to a state $g$ on $\C$ belongs to
$\Mod(\C,\D)$.  Hence  $|g(x)|\leq B_\sigma(x)$ for every $x\in\C$.  In
particular, $|\rho(x)| \leq B_\sigma(x)$ for every $x\in\C_\sigma$.
So if $x\in\C_\sigma$, we obtain, 
$$\norm{x+\fI_\sigma}=\sup\{|\rho(x)|: \rho \text{ is a state on
}\C_\sigma \text{ and } \rho|_{\fI_\sigma}=0\}\leq B_\sigma(x).$$
Hence $B_\sigma$ gives the quotient norm on $\C_\sigma/\fI_\sigma$.

Turning now to statement $(1)$, 
let  $\rho$ be a state on $\C_\sigma$ which annihilates $\fI_\sigma$,
and suppose for $i=1,2$, that $\tau_i$ are states on $\C$ with
$\tau_i|_{\C_\sigma}=\rho$.   
Let $v\in \N(\C,\D)\setminus \{\lambda u: u\in H_\sigma \text{ and
}\lambda\in\bbC\}$.  

We claim that $\tau_i(v)=0$.  Since $\rho$
annihilates $\fI_\sigma$, we have $\sigma=\rho|_\D=\tau_i|_\D$.   
 Suppose that $\sigma(v^*v)\neq 0$.  By multiplying $v$ by a
 suitable scalar, we may assume that $\sigma(v^*v)=1$.    
Since $v$ is not a scalar multiple of an element of $H_\sigma$, 
we see that $\beta_v(\sigma)\neq \sigma$.  Lemma~\ref{vactmod} shows
that $\tau_i(v)=0$. 
On the other hand, if $\sigma(v^*v)=0$, then $v^*v\in
\fI_\sigma$, so $\tau_i(v)=\lim_{n\rightarrow
  \infty}\tau_i(v(v^*v)^{1/n})=0$.  Thus the claim holds.

    Since $\tau_1(v)=\tau_2(v)$ for
every $v\in H_\sigma$, we have $\tau_1(v)=\tau_2(v)$ for every $v\in
\N(\C,\D)$.  Hence $\tau_1=\tau_2$.   Thus $\rho$ extends uniquely to
a state on
$\C$.  Notice also that if $\rho\in\fS(\C_\sigma,\D)$, then this
argument shows that $\tilde{\rho}\in \fS(\C,\D)$.

To prove statement $(2)$, use the fact that $B_\sigma$ gives the
quotient norm on $\C_\sigma/\fI_\sigma$ and apply 
Lemma~\ref{Lambda1norm} to $u\in
H_\sigma$.
\end{proof}

The following is a corollary of Proposition~\ref{Bsigquot}.
\begin{proposition}\label{Ronegp}  Let $(\C,\D)$ be a regular
  inclusion, $\sigma\in\hat{\D}$ and let $\Lambda$ be a subgroup of
  $\bbT$.  
Then $R_\Lambda$ is an equivalence
  relation on $H_\sigma$.  
Denote the equivalence
  class of $v\in H_\sigma$ by $[v]_\Lambda$.   The product
  $[v]_\Lambda [w]_\Lambda:=[vw]_\Lambda$ is a well-defined product on
  $H_\sigma/R_\Lambda$.  With this product, $[I]_\Lambda$ is the unit
  and for each $v\in H_\sigma$, $[v]_\Lambda^{-1}=[v^*]_\Lambda$.
  Thus $H_\sigma/R_\Lambda$ is a group.

Furthermore, the map $T_\sigma: H_\sigma/R_1\rightarrow
\C_\sigma/\fI_\sigma$ given by $T_\sigma([u]_1)=u+\fI_\sigma$ is a
one-to-one group homomorphism of $H_\sigma/R_1$ into the unitary group of 
$\C_\sigma/\fI_\sigma$, and  $T_\sigma(H_\sigma/\fI_\sigma)$
generates $\C_\sigma/\fI_\sigma$.
\end{proposition}
\begin{proof}  
Let $u,v\in H_\sigma$.   By Proposition~\ref{Bsigquot} and
Lemma~\ref{Lambda1norm}, $(u,v)\in R_\Lambda$ if and only if there
exists $\lambda\in \Lambda$ such that  $u+\fI_\sigma=\lambda
v+\fI_\sigma$.  Routine arguments now show that $H_\sigma/R_\Lambda$
is a group under the indicated operations.  The final statement
follows from Proposition~\ref{Bsigquot}(2).

\end{proof}

\begin{lemma}\label{maptot} Let $(\C,\D)$ be a regular MASA inclusion,
  $\sigma\in\hat{\D}$, and suppose $u,v\in H_\sigma$ are such that 
$(u,v)\in R_1$.  Then the following
  statements hold.
\begin{enumerate}
 \item If 
  $\rho\in \fS(\C,\D,\sigma)$, then
  $$\rho(v)=\rho(u).$$
  If in addition, $0\neq\rho(v)$ then
  $\rho(v)\in\bbT$.
\item $\sigma\in (\fix\beta_v)^\circ$ if and only if
  $\sigma\in(\fix\beta_u)^\circ$, and when this occurs,
  $\hat{v}(\sigma)=\hat{u}(\sigma)\in\bbT.$
\end{enumerate} 
\end{lemma}
\begin{proof}
  Suppose $\rho(v)\neq 0$.  Since $|\rho(x)|\leq B_\sigma(x)$ for
  every $x\in \C$ and $B_\sigma(I-u^*v)=0$, we have $\rho(u^*v)=1$.
  Therefore, by part~\ref{Dext0} of Proposition~\ref{Dextreme},
$$\rho(u)=\rho(u)\rho(u^*v)=
\rho(uu^*v)=\sigma(uu^*)\rho(v)=\rho(v).$$
Likewise, if $\rho(u)\neq 0$, then $\rho(u)=\rho(v).$
Thus, we have $\rho(u)=\rho(v)$ whenever
$\rho\in\fS(\C,\D,\sigma)$.

Next, when $\rho(v)\neq 0$, the fact that $\rho\in\fS(\C,\D)$ gives
$$|\rho(v)|^2=\rho(v^*v)=\sigma(v^*v)=1,$$
so $\rho(v)\in\bbT$.  This completes the proof of the first statement.

We now turn to the second statement.  Since $(u,v)\in R_1$, 
Proposition~\ref{charer} implies $\sigma\in
(\fix\beta_{v^*u})^\circ$.  Thus $\sigma\in (\fix\beta_v)^\circ$ if and only if
  $\sigma\in(\fix\beta_u)^\circ$.   

Next suppose that $\sigma\in (\fix(\beta_v))^\circ$.  
Let $(I(\D),\iota)$ be an injective envelope for $\D$ and let $E$
be the pseudo-expectation for $\iota$.  Let $h\in J_v$ satisfy
  $\sigma(h)=1$ and let $\rho\in \widehat{I(\D)}$ be such that 
 $\rho\circ\iota=\sigma$.  Then $\rho\circ E\in \fS_s(\C,\D)$, and 
 Proposition~\ref{formula} shows that 
 $\rho\in \supp(\widehat{E(v)}).$  Thus
$$\hat{v}(\sigma)=\sigma(vh)=\rho(E(v))=\rho(E(w))=\sigma(wh)=\hat{w}(\sigma).$$
Since $\rho(E(v))\neq 0$, we have $\hat{v}(\sigma)\in\bbT$ by part
(1).
\end{proof}

\begin{remark*}{Remark}
  Lemma~\ref{maptot} shows that for $\rho\in\fS(\C,\D)$ with
  $\rho|_\D=\sigma$, we have a well-defined map
  $\tilde{\rho}:H_\sigma/R_1\rightarrow\bbT\cup\{0\}$ given by
  $\tilde{\rho}([v])=\rho(v).$
\end{remark*}

\begin{theorem}\label{topgp}
Let $(\C,\D)$ be a regular MASA inclusion, and let
$\sigma\in\hat{\D}$.  The function 
$$d([v],[w]):=B_\sigma(v-w)$$ is a well defined metric on
$H_\sigma/R_1$ and makes $H_\sigma/R_1$ into a $\bbT$-group.  More
specifically, the following statements hold.
\begin{enumerate}
\item 
Let $$U=
\{[v]\in H_\sigma/R_1:
\sigma\in(\fix\beta_v)^\circ\}.$$
Then $U$ is clopen and is
the connected component of the identity in $H_\sigma/R_1$. 

\item  The subgroup $U$   is contained in the center of
$H_\sigma/R_1$. 
\item  The map $[v]\in  U\mapsto \hat{v}(\sigma)$ is 
an isomorphism of $U$ onto $\bbT$. 
\item The quotient of
$H_\sigma/R_1$ by $U$ is 
$H_\sigma/R_\bbT$.
\end{enumerate}
\end{theorem}
\begin{proof}
Let $T_\sigma$ be the isomorphism of $H_\sigma/R_1$ onto a subgroup of
the unitary group of $\C_\sigma/\fI_\sigma$ defined in
Proposition~\ref{Ronegp}.  Then 
$$d([v],[w])=\norm{T_\sigma([v])-T_\sigma([w])}_{\C_\sigma/\fI_\sigma}.$$  
It follows that $d$ is a well-defined metric which makes
$H_\sigma/R_1$ into a topological group.

We now show that $U$ is an open set.
Let $u \in H_\sigma$ be such that $[u]\in U$ and suppose that $v\in
H_\sigma$  satisfies
$d([u],[v])<1/2.$   We will show that $\sigma\in(\fix\beta_v)^\circ$.
To do this we modify the proof of the implication (3) $\Rightarrow$ (4)
in
 Proposition~\ref{charer} slightly.  
Since $B_\sigma(u-v)<1/2$, for every $\rho\in\fS_s(\C,\D)$, we have
$|\rho(u)-\rho(v)|<1/2.$   

Suppose, to obtain a contradiction, that $\sigma\notin (\fix\beta_v)^\circ$.  
Then we may find find a directed set $S$ and a net $(\sigma_s)_{s\in S}$ such that 
 $\sigma_s\notin \fix\beta_v$ for every $s$ and such that
$\sigma_s\rightarrow \sigma.$ 
As usual,
let $(I(\D),\iota)$ be an injective envelope for $\D$ and let $E$
be the pseudo-expectation for $\iota$.
 For each $s$, choose $\tau_s\in
\widehat{I(\D)}$ such that $\tau_s\circ
\iota=\sigma_s$.  Passing to a subnet if necessary, we may
assume that $\tau_s$ converges to $\tau\in \widehat{I(\D)}$.  
Then $\tau\circ E\in \fS_s(\C,\D)$ and $\tau\circ E|_\D=\sigma$.   
Notice that $\tau(E(v))\neq 0$ because $|\tau(E(v))-\tau(E(u))|<1/2$ and
$\tau(E(u))\in\bbT$.   
   Since $\tau_s\circ E|_\D=\sigma_s\notin\fix\beta_v$, Lemma~\ref{vactmod} shows that
   $\tau_s(E(v))=0$ for every $s\in S$.    Then
 \begin{align*}
\tau(E(v))&=\lim_s\tau_s(E(v)) =0,
\end{align*}
contradicting the fact that $\tau(E(v))\neq 0$.
Thus $\sigma\in (\fix\beta_v)^\circ.$
 Therefore $[v]\in U$, so $U$ is an open subset of $H_\sigma/R_1$.

Similarly, the complement of  $U$ is open in
$H_\sigma/R_1$, so $U$ is also closed.

Let $\gamma: U\rightarrow\bbT$ be the map
$\gamma([u])=\hat{u}(\sigma).$ 
Suppose that $[u], [v]\in U$ and $\hat{u}(\sigma)=\hat{v}(\sigma)$.
Then $$d([u],[v])=B_\sigma(u-v)\leq
B_\sigma(u-\hat{u}(\sigma)I)+B_\sigma(\hat{v}(\sigma)I-v)=0,$$ so
$\gamma$ is one-to-one.  Since $\gamma([\lambda
u])=\lambda\gamma([u])$ for any $\lambda\in \bbT$, $\gamma$ is onto.  

Let $\rho\in\fS(\C,\D,\sigma)$.
Since
$$|\tilde{\rho}([v])-\tilde{\rho}([w]) |=|\rho(v-w)|\leq
B_\sigma(v-w)=d([v],[w]),$$ we see that $\tilde{\rho}$ is a continuous
map on $H_\sigma/R_1$.  By Lemma~\ref{maptot}, for $[v]\in U,$
$\tilde{\rho}([v])=\gamma([v])$.  So $\gamma$ is also continuous.
The map $\lambda\in\bbT\mapsto [\lambda I]$ is the inverse of
$\gamma$, and we see that $\gamma$ is a homeomorphism.  In
particular, $U$ is connected, and hence $U$ is the connected component
of the identity in $H_\sigma/R_1$.      
 
To see that $U$ is contained in the center of $H_\sigma/R_1,$ observe
that for $[u]\in U$, we have $[u]=[\gamma([u])I]$, which evidently
belongs to the center of $H_\sigma/R_1.$

Since the connected component of the identity is compact,
$H_\sigma/R_1$ is a locally compact group.  
Finally, $[v]=[w] \mod
U$ if and only if $[v^*w]\in U$.  Therefore,  $[v]=[w] \mod
U$ if and only if 
$B_\sigma(\gamma([v^*w]) I -v^*w)=0$.  Hence the  
quotient of $H_\sigma/R_1$ by $U$ is $H_\sigma/R_\bbT$.

\end{proof}

We now are prepared to exhibit a bijection between $\fS(\C,\D,\sigma)$ 
and  a class of \phom s
on $H_\sigma/R_1$.  We pause for some notation.

Let $q:\C_\sigma\rightarrow \C_\sigma/\fI_\sigma$ be the quotient map.  
Define a $*$-homomorphism $\theta: C_c(H_\sigma/R_1)\rightarrow
\C_\sigma/\fI_\sigma$ by 
$$ \theta(\phi)=\int_{H_\sigma/R_1} \phi(s) T_\sigma(s)\, ds,$$ where $ds$ is
Haar measure on $H_\sigma/R_1$.  Then the  image of
$\theta$ is dense in $\C_\sigma/\fI_\sigma$.

\begin{definition}  We will say that a positive definite
  function $f$ on $H_\sigma/R_1$  
is \textit{dominated by $B_\sigma$} if 
for every $\phi\in C_c(G)$,
$$\left|\int_{H_\sigma/R_1} \phi(t) f(t)\, dt\right| 
\leq \norm{\int_{H_\sigma/R_1} \phi(t) T_\sigma(t)\,
  dt}_{\C_\sigma/\fI_\sigma}.$$
\end{definition}

\begin{theorem}\label{compatstpdf}  Let $(\C,\D)$ be a regular MASA
  inclusion and let $\sigma\in\hat{\D}$.  Let 
$$M_\sigma:=\{f\in\tog(H_\sigma/R_1): f \text{ is dominated by
  $B_\sigma$}\}.$$
For $\tau\in\fS(\C,\D,\sigma)$ 
the map $\tilde{\tau}:H_\sigma/R_1\rightarrow \bbC$ given by
$$\tilde{\tau}([v]_1)=\tau(v)$$ is well-defined, and $\tilde{\tau}\in
M_\sigma$. 
Moreover, the map $\tau\mapsto\tilde{\tau}$ is a bijection between
$\fS(\C,\D,\sigma)$  and $M_\sigma$.  
\end{theorem}

\begin{proof}
If $\tau\in\fS(\C,\D,\sigma)$, then
  $|\tau(x)|\leq B_\sigma(x)$.  Hence $\tau$ annihilates $\fI_\sigma$,
  so that $\tau$ determines a state
  $\tau'$ on $\C_\sigma/\fI_\sigma$ such that
$$\tau|_{\C_\sigma}=\tau'\circ q.$$  Then 
$\tilde{\tau}=\tau'\circ T_\sigma$, so $\tilde{\tau}$ is well-defined.  

Set 
$$G:=\{[v]\in H_\sigma/R_1: \tilde{\tau}([v])\neq 0\}.$$
Proposition~\ref{Dextreme}\eqref{Dext0} implies that $G$ is closed
under products.  For $v\in H_\sigma$ with  $[v]\in G$, we have
$\tau(v)\in \bbT$, and as $\tau$ is a state, 
 $\tau(v^*)=\overline{\tau(v)}$.  Therefore, $[v]^{-1}=[v^*]\in  G$,
 so $G$  is
closed under inverses.  It follows that  $G$ is a $\bbT$-subgroup of
$H_\sigma/R_1$.   Since $\tau\in \fS(\C,\D)$, $|\tau([v])|\in\{0,1\}$
for every $[v] \in H_\sigma/R_1$.  
Proposition~\ref{posdef} implies that $\tilde{\tau}$
is a \phom\  on $H_\sigma/R_1$.  
Since $\tau$ is linear, for $\lambda\in\bbT$ and $v\in H_\sigma$, we have
$\tilde{\tau}([\lambda v]_1)=\tau(\lambda
v)=\lambda\tilde{\tau}([v]_1)$.  So $\ind(\tilde{\tau})=1.$

For $\phi\in C_c(H_\sigma/R_1)$ we have
$$\left|\int_{H_\sigma/R_1} \phi(s) \tilde{\tau}(s)\, ds\right|=
\left|\tau'\left(\int_{H_\sigma/R_1} \phi(s) T_\sigma(s)\, ds\right)\right|
\leq \norm{\int_{H_\sigma/R_1} \phi(s) T_\sigma(s)\,
  ds}_{\C_\sigma/\fI_\sigma}.  $$  Thus, $\tilde{\tau}$ is dominated by
$B_\sigma$.   Therefore, $\tilde{\tau}\in M_\sigma$.

Next we show that the map $\tau\mapsto \tilde{\tau}$ is surjective.
So suppose that $f:H_\sigma/R_1\rightarrow \{0\}\cup \bbT$ is a
\phom\  dominated by $B_\sigma$ and $\ind(f)=1$.  
Then the map $F_0:\theta(C_c(H_\sigma/R_1))\rightarrow \bbC$
given by 
$$F_0\left(\int_{H_\sigma/R_1} \phi(t) T_\sigma(t)\,
  dt\right)=\int_{H_\sigma/R_1} \phi(t) f(t)\,
  dt\qquad (\phi\in C_c(H_\sigma/R_1))$$
extends by continuity to a bounded linear functional $F$ on $\C_\sigma/\fI_\sigma$.
Since $\theta$ is a $*$-homomorphism and $f$ is a positive definite
function, $F$ is a positive linear functional.  Clearly $\norm{F}\leq
1$.   

 We next show that 
 $F(v+\fI_\sigma)=f([v])$ for every $v\in H_\sigma$.  
   Since $H_\sigma/R_1$ is
a $\bbT$-group, the connected component of the identity is $\bbT$. 
Since $f$ is continuous, $\ind(f)=1$  and $f(1)=1$, we have $f(\lambda)=\lambda$ for
every $\lambda\in\bbT$.  Given $v\in H_\sigma$, 
 let $\phi\in C_c(H_\sigma/R_1)$ be the function given by 
$$\phi(t)=\begin{cases}
\overline{\lambda}  &\text{if $t=[\lambda v]_1$ for some $\lambda\in\bbT$}\\
0&\text{otherwise.}\end{cases}$$
Now for $t\in \{[\lambda v]_1: \lambda\in \bbT\}\subseteq H_\sigma/R_1$, we have $T_\sigma(t)=\lambda q(v)$, so 
$$\theta(\phi)=\int_{H_\sigma/R_1} \phi(t)T_\sigma(t)\, dt =\int_\bbT  \,
\overline{\lambda}q(\lambda v)dt=q(v).$$  Thus, 
$$F(q(v))=\int_{H_\sigma/R_1} \phi(t)f(t)\, dt=\int_\bbT
\overline{\lambda}f([\lambda v]_1) \,
dt=f([v]_1).$$    It follows that $\norm{F}=1$, so $F$ is a state on
$\C_\sigma/\fI_\sigma$.  

As $\N(\C_\sigma,\D)\subseteq \N(\C,\D)$, we see that if $w\in
\N(\C_\sigma,\D)$, then $w$ is a scalar multiple of an element of
$H_\sigma$.   Since $|F\circ q(v)|=|f([v]_1)|\in \{0,1\}$ for each
$v\in H_\sigma$, we find  $F\circ q\in\fS(\C_\sigma,\D)$.
Proposition~\ref{Bsigquot} shows that $F\circ q$ extends uniquely to
an element $\tau\in \fS(\C,\D,\sigma)$.   As $\tilde{\tau}=f$, we find that
the map $\tau\mapsto\tilde{\tau}$ is onto.  

To show that $\tau\mapsto \tilde{\tau}$ is one-to-one, suppose $\tau$
and $\tau_1$ belong to $\fS(\C,\D,\sigma)$ and 
$\tilde{\tau}=\tilde{\tau_1}$.   Then $\tau(v)=\tau_1(v)$ for every $v\in
H_\sigma$, so that $\tau|_{\C_\sigma}=\tau_1|_{\C_\sigma}$.
Proposition~\ref{Bsigquot} then shows $\tau=\tau_1$.  

Thus the mapping
$\tau\mapsto \tilde{\tau}$ is indeed a bijection. 
\end{proof}

We conclude this section with a pair of very closely related questions
and two examples.

\begin{remark}{Notation}\label{allB}  Let $(\C,\D)$ be a regular
  inclusion and let $$\bist(\C,\D):=\{x\in \C: B_\sigma(x)=0 \text{
    for all }\sigma\in\hat{\D}\}.$$ 
By Proposition~\ref{Bsigquot}, $\bist(\C,\D)=\{x\in \C: \rho(x^*x)=0
\text{ for all }\rho\in\Mod(\C,\D)\},$ and by
Proposition~\ref{invideal},
 $\bist(\C,\D)$ is an ideal
of $\C$ which satisfies
\begin{equation}\label{bistincl}
\bist(\C,\D)\subseteq \rad(\C,\D).
\end{equation}
\end{remark}

\begin{question}\label{bistradcont}
Let $(\C,\D)$ be a regular
MASA   inclusion.
\begin{enumerate}
\item Let $\sigma\in\hat{\D}$.  For $x\in \C$, is
  $B_\sigma(x)=\sup_{\rho\in \fS(\C,\D)} \rho(x^*x)^{1/2}?$
\item Is $\bist(\C,\D)=\rad(\C,\D)$?  
\end{enumerate}
\end{question}

When $\rad(\C,\D)=(0)$,
the answer to both questions is yes.

\begin{example}\label{disGrp}
This example applies the previous results to a special case of reduced
crossed products to produce a Cartan inclusion $(\C,\D)$ and a pure
state $\tau$ on $\C$ such that $\tau|_\D\in\hat{\D}$, yet
$\tau\notin\fS(\C,\D)$.  

 Let $\Gamma$ be an infinite discrete group, let
$X:=\Gamma\cup\{\infty\}$ be the one-point compactification of
$\Gamma$ and let $\Gamma$ act on $X$ by extending the left regular
representation to $X$:  for $s\in
\Gamma$ and $x\in X$, let 
$$sx:=\begin{cases}
sx&\text{if $x\in \Gamma$;}\\
x&\text{if $x=\infty$.}
\end{cases}$$

We now use the notation from Section~\ref{secCP}: let
$\C=C(X)\rtimes_r\Gamma$ and $\D$ be the cannonical image of $C(X)$ in
$\C$.  We identify $\hat{\D}$ with $X$.  The action of $\Gamma$ on $X$ is topologically free, so
$(\C,\D)$ is a regular MASA inclusion.  Moreover, the conditional
expectation $\coexp:\C\rightarrow
\D$ is the  pseudo-expectation.

Let $\sigma\in \hat{D}$ be the map $\sigma(f)= f(\infty)$.  We claim
that $H_\sigma/R_\bbT=\Gamma$ and $H_\sigma/R_1=\bbT\times \Gamma$.
Observe first that $\fS(\C,\D,\sigma)=\{\sigma\circ\coexp\}$ and 
the map $\theta:\Gamma\rightarrow H_\sigma/R_\bbT$ given by
$\theta(s)= [\emb_s]_\bbT$  is a
group homomorphism.   
Next, suppose $v\in H_\sigma$.  Then there exists $t\in \Gamma$, so 
that $\bbE_t(v)=\bbE(v \emb_{t^{-1}})\neq 0.$  Since both $v$ and
$\emb_{t^{-1}}\in H_\sigma$ we have $v\emb_{t^{-1}}\in H_\sigma$, so
$\sigma(\coexp(v\emb_{t^{-1}}))\neq 0$.  By Proposition~\ref{charer}, 
$(v,\emb_t)\in R_\bbT$.   Thus $\theta$ is
surjective.  Proposition~\ref{charer} also implies $\theta$ is one-to-one: 
if $[\emb_s]_\bbT=[\emb_t]_\bbT$
then $\sigma(\coexp(\emb_{st^{-1}}))\neq 0$, so that
$\coexp(w_{st^{-1}})\neq 0$.  Thus $s=t$.  Therefore
$\theta$ is an isomorphism of $H_\sigma/R_\bbT$ onto $\Gamma$, and we
use $\theta$ to identify $\Gamma$ with $H_\sigma/R_\bbT$.  The map $\Gamma\ni s\mapsto [\emb_s]_1\in H_\sigma/R_1$ 
is a group homomorphism and
also a section for the quotient
map of $H_\sigma/R_1$ onto $H_\sigma/R_\bbT$.  It follows that
$H_\sigma/R_1$ is isomorphic to $\bbT\times \Gamma$.    

Let $\rho=\sigma\circ\coexp$ and let $(\pi_\rho,\H_\rho)$ be the GNS
representation of $\C$ associated to $\rho$.
Proposition~\ref{Bsigquot} implies that
$\fI_\sigma\subseteq \ker\pi_\rho$, so $\pi_\rho$ induces a
representation, again denoted $\pi_\rho$, of $\C_\sigma/\fI_\sigma$ on
$\H_\sigma$.  

 We shall show that the image of
$\C_\sigma/\fI_\sigma$ under $\pi_\rho$ is isomorphic to
$C^*_r(\Gamma)$.  To  do this, first observe that for 
 $v,w\in H_\sigma$,  
  $v+L_\rho=w+L_\rho$ if and only if $(v,w)\in R_1$.  Indeed, since
  $v,w\in H_\sigma$, $\rho(w^*w)=\rho(v^*v)=1$, so
  $\rho((v-w)^*(v-w))= 2-2\Re(\rho(v^*w))$.  Since $|\rho(v^*w)|\in \{0,1\},$
 we get $v+L_\rho=w+L_\rho$ if and only if $\rho(v^*w)=1$, which
by Proposition~\ref{charer}  gives the observation.
This observation and regularity of $\C_\sigma$ implies that 
$$\{\emb_s+L_\rho: s\in \Gamma\}$$ is an orthonormal basis for $\H_\rho$.
Thus there is a unitary operator $U:\ell^2(\Gamma)$
onto $\H_\rho$ which carries the basis element $\delta_s\in
\ell^2(\Gamma)$ to $\emb_s+L_\rho\in \H_\rho$.

Now let $\lambda:  \Gamma\rightarrow \B(\ell^2(\Gamma))$ be the left regular
representation.   For $s,t\in \Gamma$ we have   
$U\lambda(s)\delta_t=\emb_{st}+L_\rho=\pi_\rho(\emb_s)U\delta_t,$ so 
\begin{equation}\label{lrrueq} 
U\lambda(s)=\pi_\rho(\emb_s)U.
\end{equation}
It follows from Proposition~\ref{Ronegp} that the set
$\{\emb_s+\fI_\sigma:s\in\Gamma\}$ generates $\C_\sigma/\fI_\sigma$,
so~\eqref{lrrueq} shows that $\pi_\rho(\C_\sigma/\fI_\sigma)$ is
isomorphic to $C^*_r(\Gamma)$.  

Thus there is a surjective $*$-homomorphism $\Psi:\C_\sigma\rightarrow
C^*_r(\Gamma)$ which annihilates $\fI_\sigma$.  The composition of
$\Psi$ with any pure state $f$ on $C^*_r(\Gamma)$ yields a pure state
on $\C_\sigma$, which in turn may be extended to a pure state
$\tau\in\Mod(\C,\D,\sigma)$.  Apply this process when $\Gamma=\bbF^2$
is the free group on 2-generators $u_1$ and $u_2$.  By
\cite[Theorem~2.6 and Remark~3.4]{PaschkePuEiSuGeFrGp}, there
exists a pure state $f$ on $C^*_r(\bbF^2)$ such that
$|f(u_1)|\notin\{0,1\}$.  It follows that there exists a pure state
$\tau$ on $\C$ such that $\tau|_\D=\sigma$, yet
$\tau\notin\fS(\C,\D)$.
\end{example}

\begin{example}\label{ToeplitzAlg}
Denote by 
   $\{e_n\}_{n\in\bbN}$ the standard orthonormal basis for $\H:=\ell^2(\bbN)$.
Consider the inclusion $(\C,\D)$, where 
 $\C$ is the Toeplitz algebra
  (the \cstaralg\ generated by the unilateral shift $S$ acting on
  $\H$) and 
 $\D\subseteq\bh$ is the \cstaralg\
  generated by $\{S^kS^{*k}: k\geq 0\}$.  Then
  $(\C,\D)$ is a regular MASA inclusion and  $\hat{\D}$ is homeomorphic
  to the one-point compactification of $\bbN$,  $\bbN\cup\{\infty\}$.
  We identify $\hat{\D}$ with this space.

Here the pseudo-expectation is the conditional expectation
  $E:\C\rightarrow \D$ which takes $T\in \C$ to the operator $E(T)$
  which acts on
   basis elements via  $E(T)e_n=\innerprod{Te_n,e_n} e_n$.  

  We shall do the following:
\begin{enumerate}
\item give a description of $\fS(\C,\D)$;
\item show that not every element of $\fS(\C,\D)$ is a pure state of
  $\C$ and identify the pure states in $\fS(\C,\D)$.
\end{enumerate}

The strongly compatible states are easy to identify.
Let $\rho_n$ and $\rho_\infty$ be the states on $\C$ given by
   $\rho_n(X)=\innerprod{Xe_n,e_n}$ and
   $\rho_\infty(X)=\lim_{n\rightarrow \infty}\rho_n(X)$.  Then
   $$\fS_s(\C,\D)=\{\rho_n: n\in \bbN\cup\{\infty\}\}= \{\sigma\circ E:
   \sigma\in\hat{\D}\}.$$
For each $n\in\bbN$, the set $\{n\}$ is clopen in $\hat{\D}$, so
$\rho_n$ is the unique extension of $\rho_n|_\D$ to a state on $\C$.
(This can be proved directly or viewed as a consequence of
Theorem~\ref{denseuep}.)

Thus, to complete a description of $\fS(\C,\D)$, we need only describe
$\fS(\C,\D,\sigma_\infty),$ where
$\sigma_\infty=\rho_\infty|_\D$. 

To do this,  
let $\K=\K(\H)$ be the compact operators and let
  $q:\C\rightarrow \C/K=C(\bbT)$ be the quotient map.  Given
  $z\in\bbT$, we write $\tau_z$ for the state on $\C$ given by
  $\tau_z(T)=q(T)(z)$. 
Also, for $z\in \bbT$, we let $\alpha_z$ be the gauge automorphism on $\C$
determined by $\alpha_z(S)=zS$.  For each $N\in \bbN$, let
$\lambda(N)=\exp(2\pi i/N)$ and define
$\Phi_N:\C\rightarrow\C$ by
$$\Phi_N(T):=\frac{1}{N}\sum_{k=0}^{N-1} \alpha_{\lambda(N)}^k(T).$$
(Note that if  $T=\sum_{k=-p}^p a_kS^k$ is a ``trigonometric
polynomial,'' then 
$\Phi_N(T)=\sum_{k\in N\bbZ} a_kS^k$.) 
  
We claim that 
\begin{equation}\label{cmoin}
\fS(\C,\D,\sigma_\infty)=\{\rho_\infty\}\cup
\{\tau_z\circ\Phi_N: z\in \bbT, N\in \bbN\}.
\end{equation}

Each state of the form
$\tau_z\circ \Phi_1=\tau_z$ is multiplicative on $\C$.  Therefore,
$\{\tau_z\circ\Phi_1: z\in\bbT\}$ is a set of   pure
states and is a subset of  $\fS(\C,\D)$.  Also, we have
$$\rho_\infty=\int_\bbT \tau_z\, dz\dstext{and for $N\geq
  1$,}
\tau_z\circ \Phi_N=\frac{1}{N}\sum_{k=0}^{N-1} \tau_{\lambda(N)^kz},$$
so the only pure states on the right hand side of~\eqref{cmoin} are
those of the form $\tau_z$.  We will show that states of the form
$\tau_z\circ\Phi_N$ are compatible states.   We proceed by first
identifying the elements of $\N(\C,\D)$ with $\sigma_\infty(v^*v)\neq 0$.

Suppose  $v\in\N(\C,\D)$ satisfies  $\sigma_\infty(v^*v)>0.$  
Put $A_v=\{n\in\bbN : \innerprod{v^*v e_n,e_n}\neq 0\}$ and
$B_v=\{n\in \bbN: \innerprod{vv^*e_n,e_n}\neq 0\}$.  Then 
$\beta_v$ induces a bijection $f:A_v\rightarrow B_v$, and there exist
scalars $c_j$ so that  
$$v_{ij}:=\innerprod{ve_j,e_i}=\begin{cases} c_j&\text{if $i=f(j)$;}\\
  0&\text{otherwise.} \end{cases}$$ Moreover, note that $q(v)=cq(S)^m$
for some $c\in \bbC$ and $m\in\bbZ$.  Since $\sigma_\infty(v^*v)\neq
0$, $c\neq 0$, so $v$ is a Fredholm operator.
Let $m$ be the Fredholm index of $v$.  Then $q(v)=\rho_\infty(vS^{-m})q(S)^m$.
Thus $v=S^m d$ for some $d\in \D$ with $\sigma_\infty(d)\neq 0$.  
The fact that each
$\tau_z\circ\Phi_N\in\fS(\C,\D)$, now follows.
It remains to show that we have found all elements of $\fS(\C,\D,\sigma_\infty)$.

We will write $H_\infty$, $B_\infty$, $C_\infty$, and $\fJ_\infty$
rather than the more cumbersome $H_{\sigma_\infty}$,
$B_{\sigma_\infty}$, $\C_{\sigma_\infty}$, and $\fJ_{\sigma_\infty}$.
Then
$$H_\infty=\{S^m d: m\in\bbZ\text{ and } d\in \D,
\sigma_\infty(d)\in\bbT\}.$$
Thus,
 $\fJ_\infty=\K$ and
$\C_\infty=\C$.  By Proposition~\ref{Bsigquot},
$B_\infty$ gives the quotient norm on $\C_\infty/\fJ_\infty=C(\bbT)$.

Next, for $v,w\in H_\infty$, we have $B_\infty(v-w)=0$ if and only if
$\rho_\infty(w^*v)=1$.  In other words, $(v,w)\in R_1$ if and only
if $q(v)=q(w)$.  Therefore, 
$H_\infty/R_1$ is isomorphic to the direct product $\bbT\times \bbZ$.
Observe that $H_\infty/R_\bbT$ is isomorphic to $\bbZ$, and
we obtain the trivial $\bbT$-group extension,
$$1\rightarrow \bbT\rightarrow \bbT\times\bbZ\rightarrow
\bbZ\rightarrow 1.$$
The generators of the subgroups of $\bbZ$ are the non-negative
integers, so the $\bbT$-subgroups of $\bbT\times \bbZ$ are 
$$\{\bbT\times n\bbZ: n\geq 0\}.$$  Let $f$ be a \phom\ of
index 1 on
$H_\infty/R_1$.  Then there exists a non-negative integer $N$ such
that $f$ is a character on $\bbT\times N\bbZ$.   Since $f$ has index
1,  there exists
$\lambda\in \bbZ$ such that for $(z,n)\in \bbT\times \bbZ=H_\infty/R_1,$
$$f(z,n)=\begin{cases}
z\lambda^n&\text{if $n\in N\bbZ$;}\\
0& \text{otherwise.}\end{cases}$$  As $B_\infty$ is the quotient
norm on $\C_\infty$,  $f$ is dominated by $B_\infty$.  Let
$\tau=\tau_z\circ \Phi_N$.  With the notation of 
Theorem~\ref{compatstpdf}, we get $\tilde{\tau}=f$, so  by
Theorem~\ref{compatstpdf},  the compatible state
corresponding to $f$ is $\tau_z\circ \Phi_N.$  This completes the proof
of~\eqref{cmoin}.  

We now identify the topology on $\fS(\C,\D)$.  
Notice that if $M,N$ are positive integers with $N\notin M\bbZ$,
then for any $z\in\bbT$, $\tau_z(\Phi_N(S^N))=\tau_z(S^N)=z^N\neq
0=\tau_z(\Phi_M(S^N)).$  Given distinct positive integers $N,M$,
either $N\notin M\bbZ$ or $M\notin N\bbZ$.  Thus for $N>0$,  
$\{\phi_z\circ\Phi_N: z\in \bbT\}$ is a connected component of
$\fS(\C,\D)$ and  is homeomorphic to $\bbT$.    

We next show that if $G$ is a weak-$*$ open neighborhood of
$\rho_\infty$, then there exists $N\in\bbN$ such that
$\fT_n:=\{\tau_z\circ \Phi_n: z\in \bbZ\}\subseteq
G$ for every $n\geq N$.  To do this, it suffices to show that for
every $a\in\C$ and $\eps>0$, the set $G_{a,\eps}:=\{\phi\in\fS(\C,\D):
|\phi(a)-\rho_\infty(a)|<\eps\}$ contains $\fT_n$ for all sufficiently
large $n$, and this is what we shall do.    First observe that for
every $b\in\C$,
\begin{equation}\label{psiright}
\lim_{n\rightarrow\infty}\norm{\Phi_n(b)-E(b)}=0.\end{equation}  (This
is clear for ``trigonometric polynomials'' in $S$, approximate $b$ in
norm with a trigonometric polynomial 
 to obtain~\eqref{psiright}.)
Fix  $N\in\bbN$  so that $\norm{\Phi_n(a)-E(a)}<\eps$ for
every $n\geq N$.  For any $z\in \bbT$, we have
$\tau_z(E(a))=\rho_\infty(a)$, so we find 
$$|\tau_z(\Phi_n(a))-\rho_\infty(a)|<\eps.$$ So $\fT_n\subseteq
G_{a,\eps}$ for all $n\geq N$.

We conclude that 
$$\fS(\C,\D)=\fS_s(\C,\D)\cup\{\tau_z\circ \Phi_N: z\in \bbT, N\in
\bbN\},$$ which may be viewed as the one-point compactification of
the space $\bbN \cup \left(\bigcup_{N\in\bbN}\bbT\times\{N\}\right) $, with
$\rho_\infty$ corresponding to the point at infinity.  
\end{example}

\section{The Twist of a Regular Inclusion}\label{GpReIn}

Throughout this section, we fix, once and for all, a regular inclusion
$(\C,\D)$ and a closed $\N(\C,\D)$-invariant subset $F\subseteq
\fS(\C,\D)$ such that the restriction map, $f\in F\mapsto f|_\D$, is a
surjection of $F$ onto $\hat{\D}$.  When $(\C,\D)$ is a regular MASA
inclusion, Theorem~\ref{minonto} shows $\fS_s(\C,\D)\subseteq F$.  For
this reason, we have in mind taking $F=\fS_s(\C,\D)$, though other
choices (e.g. $F=\fS(\C,\D)$) may be useful for some purposes.

In this section, we show that associated to this data, there is a twist
$(\Sigma,G)$, which, when $(\C,\D)$ is a \cstar-diagonal (in which
case $F$ is necessarily $\fS_s(\C,\D)$) or a Cartan pair (with
$F=\fS_s(\C,\D)$) gives the twist of the pair as defined by
Kumjian~\cite{KumjianOnC*Di} for \cstar-diagonals, or for Cartan pairs
given by Renault in~\cite{RenaultCaSuC*Al}.  The set $F$ will be used
as the unit space for the \'{e}tale groupoid $G$ associated to the
twist $(\Sigma,G)$.

   Our construction
parallels the constructions by Kumjian and Renault, but with several
differences.  First, in the Renault and Kumjian contexts, a
conditional expectation $E:\C\rightarrow \D$ is present, and since 
$\{\rho\circ E: \rho\in\hat{\D}\}$ is homeomorphic to $\hat{\D}$,
Renault and Kumjian use $\hat{\D}$ as the unit space for the twists
they construct.  
In our context, we need not have a conditional
expectation, so we use the set $F$ as a replacement for $\hat{\D}$.
Next,   
as in the constructions of Kumjian and Renault, we construct a regular
$*$-homomorphism $\theta:(\C,\D)\rightarrow (C^*_r(\Sigma,G),
C(G^{(\circ)}))$, however, $\theta(\C)$ need not equal $C^*_r(\Sigma,G)$,
and the kernel of $\theta$ is not trivial unless the ideal $\K_F=\{x\in \C: f(x^*x)=0
\text{ for all } f\in F\}=(0).$   We note however, that this ideal is
trivial in the cases considered by Kumjian and Renault.

\subsection{Twists and their \cstaralg s}
Before proceeding, it is helpful to recall some generalities on twists
and the (reduced) \cstaralg s associated to them.

\begin{definition}
 The pair $(\Sigma, G)$ is a \textit{twist}
  if $\Sigma$ and $G$ are Hausdorff locally compact topological groupoids, $G$ is an 
  \'{e}tale groupoid and the following hold:
\begin{enumerate}
\item there is a free action of $\bbT$ by homeomorphisms of $\Sigma$ such
that whenever $(\sigma_1,\sigma_2)\in \Sigma^{(2)}$ and
  $z_1,z_2\in\bbT$, we have $(z_1\sigma_1,z_2\sigma_2)\in \Sigma^{(2)}$
  and $(z_1\sigma_1)(z_2\sigma_2)=(z_1z_2)(\sigma_1\sigma_2)$;
\item there is a continuous surjective groupoid homomorphism 
$\gamma:\Sigma\twoheadrightarrow G$ such that for every
$\sigma\in\Sigma$, $\gamma^{-1}(\gamma(\sigma))=\{z\sigma: z\in\bbT\}$;
\item the bundle $(\Sigma,G,\gamma)$ is  locally trivial.
\end{enumerate}

Notice that  $\gamma|_{\unitspace{\Sigma}}: 
\unitspace{\Sigma}\rightarrow \unitspace{G}$ is a homeomorphism of the
unit space of $\Sigma$ onto the unit space of $G$.  We will usually
use this map to 
identify $\unitspace{\Sigma}$ and $\unitspace{G}$.

\end{definition}

Recall that given a twist $\Sigma$ over the \'{e}tale topological
groupoid $G$, one can form the twisted groupoid \cstaralg\ of the pair
$(\Sigma,G)$.  We summarize the construction in our context, for
details, see~\cite[Section~4]{RenaultCaSuC*Al}
and~\cite[Section~2]{KumjianOnC*Di}.  We note that in both
\cite{RenaultCaSuC*Al} and~\cite{KumjianOnC*Di}, there is a blanket
assumption that the \'{e}tale groupoid $G$ is second countable.
However, for what we require here, this hypothesis is not used.  The
reader may also wish to consult Section~3 of~\cite{ExelInSeCoC*Al}.

Let $C_c(\Sigma,G)$ be
the family of all compactly supported continuous complex valued
functions $f$ on $\Sigma$ which are \textit{equivariant}, that is,
which satisfy $f(z\sigma)=zf(\sigma)$ for all  $\sigma\in\Sigma,
z\in\bbT$.  Given $f,g\in C_c(\Sigma,G)$, notice whenever
$\tau,\sigma\in\Sigma$ with $s(\tau)=s(\sigma)$ and $z\in\bbT$, we
have $f(\sigma\tau^{-1})g(\tau)=f(\sigma (z\tau)^{-1})g(z\tau)$.  
For $x\in G$ and $\sigma\in \Sigma$ with $s(x)=s(\sigma)$, let $\tau\in
\gamma^{-1}(x)$. Then 
$$(f\circledast g)(\sigma, x):=f(\sigma\tau^{-1})g(\tau)$$ does not
depend
 on the choice of $\tau\in\gamma^{-1}(x).$
The
product of $f$ with $g$ is defined by
$$(f\star g)(\sigma)=\sum_{\substack{x\in G\\s(x)=s(\sigma)}}
(f\circledast g)(\sigma,x),$$ and the adjoint operation is defined by
$$f^*(\sigma)=\overline{f(\sigma^{-1})}.$$  These operations make
$C_c(\Sigma,G)$ into a $*$-algebra.  

Similarly, notice that for $f\in C_c(\Sigma,G)$, $z\in\bbT$ and
$\sigma\in \Sigma$, $|f(\sigma)|=|f(z\sigma)|$, so for $x\in G$, we
denote by $|f(x)|$ the number $|f(\sigma)|$, where $\sigma\in \gamma^{-1}(x)$. 
In particular, $|f|$ may be viewed as a function on $G$.

One may norm $C_c(\Sigma, G)$ as in~\cite{RenaultCaSuC*Al}
or~\cite{KumjianOnC*Di}.   For convenience, we provide a sketch of an
equivalent, but 
slightly different method.   Given $x\in \unitspace{G}$, let
$\eval_x:C_c(\Sigma,G)\rightarrow \bbC$ by $\eval_x(f)=f(x).$  Then $\eval_x$ is a
positive linear functional in the sense that for each $f\in
C_c(\Sigma, G)$, $\eval_x(f^*\star f)\geq 0$.   Let $\N_x:=\{f\in
C_c(\Sigma,G): \eval_x(f^*\star f)=0\}$ and let $\H_x$ be the
completion of $C_c(\Sigma,G)/\N_x$ with respect to the inner product
$\innerprod{f+N_x,g+N_x}=\eval_x(g^*\star f)$.   

Recall that a \textit{slice} of $G$ is an open set $U\subseteq G$ so
that $r|_U$ and $s|_U$ are one-to-one.
By~\cite[Proposition~3.10]{ExelInSeCoC*Al}, given $f\in
C_c(\Sigma,G)$, there exist $n\in\bbN$ and slices $U_1,\dots, U_{n}$ of
$G$ such that the support of $|f|$ is contained in $\bigcup_{k=1}^{n}
U_k$.  Let $n(f)\geq 0$ be the smallest integer such that there exist
slices $U_1,\dots U_{n(f)}$ with the support of $|f|$ is contained in
$\bigcup_{k=1}^{n(f)}U_k$.     

For $f, g\in C_c(\Sigma,G)$, a calculation shows that 
$$\norm{(f\star g)+\N_x}\leq n(f)\norm{f}_\infty
\norm{g+\N_x}_{\H_x}.$$  Therefore, the map $g+\N_x\mapsto (f\star
g)+\N_x$ extends to a bounded linear operator $\pi_x(f)$ on $\H_x$.
It is easy to see that $\pi_x$ is a $*$-representation of
$C_c(\Sigma,G)$.  Also, if $\pi_x(f)=0$ for every $x\in \unit{G}$,
then $\norm{f+\N_x}_{\H_x}=0$ for each $x\in \unit{G}$.  A calculation
then gives $f=0$.    
Thus, 
$$\norm{f}:=\sup_x\norm{\pi_x(f)}.$$ defines a norm on
$C_c(\Sigma,G)$.  
The \textit{(reduced)
  twisted} \cstaralg, $C^*(\Sigma,G)$, is the completion of
$C_c(\Sigma,G)$ relative to this norm.  Clearly, the representation
$\pi_x$ extend by continuity to a representation, again called
$\pi_x$, of $C^*(\Sigma,G)$.

As observed in the remarks
following~\cite[Proposition~4.1]{RenaultCaSuC*Al}, elements of
$C^*(\Sigma,G)$ may be regarded as equivariant continuous functions on
$\Sigma$, and the formulas defining the product and involution on
$C_c(\Sigma,G)$ remain valid for elements of $C^*(\Sigma,G)$.  Also,
as in~\cite[Proposition~4.1]{RenaultCaSuC*Al},
for $\sigma\in\Sigma$, and $f\in C^*(\Sigma,G)$, 
$$|f(\sigma)|\leq \norm{f}.$$  
\begin{definition}
We shall call the smallest topology on $C^*(\Sigma,G)$ such that for
every $\sigma\in \Sigma$, the point evaluation functional,
$C^*(\Sigma,G)\ni f\mapsto f(\sigma)$ is 
continuous, the \textit{$\unit{G}$-compatible topology} on $C^*(\Sigma,G)$.
Clearly this topology is Hausdorff.
\end{definition}

The open support of $f\in C^*(\Sigma,G)$ is
$\supp(f)=\{\sigma\in\Sigma: f(\sigma)\neq 0\}.$ Then
$C_0(G^{(\circ)})$ may be identified with $$\{f\in C^*(\Sigma,G):
\supp(f)\subseteq G^{(\circ)}\}.$$

In order to remain within the unital context, we now assume that the
unit space of $G$ is compact.  In this case $C^*(\Sigma,G)$ is unital,
and $C(G^{(\circ)})\subseteq C^*(\Sigma,G)$, so that $(C^*(\Sigma,G), C(G^{(\circ)}))$ is an
inclusion.  We wish to show that it is a regular inclusion.

Recall (see~\cite[Section~3]{ExelInSeCoC*Al}) that a
\textit{slice} (or $G$-set) of $G$ is an open subset $S\subseteq G$
such that the restrictions of the range and source maps to $S$ are
one-to-one.   We will say that an element $f\in C^*(\Sigma,G)$ is
\textit{supported in the slice $S$} if $\gamma(\supp(f))\subseteq S$.

If $f\in C^*(\Sigma,G)$ is supported in a slice $U$, then a
computation (see~\cite[Proposition~4.8]{RenaultCaSuC*Al}) shows that
$f\in \N(C^*(\Sigma,G),C(G^{(\circ)}))$, and because the collection
of slices forms a basis for the topology of $G$
(\cite[Proposition~3.5]{ExelInSeCoC*Al}), it follows (as
in~\cite[Corollary~4.9]{RenaultCaSuC*Al}) that $(C^*(\Sigma,G),
C(G^{(\circ)}))$ is a regular inclusion.

\begin{proposition}\label{trivrad4twists} Let $\Sigma$ be a twist
  over the Hausdorff \'{e}tale groupoid $G$.  Assume that the unit space $X$ of
  $G$ is compact.  Then there is a faithful conditional expectation
  $E:C^*(\Sigma,G)\rightarrow C(G^{(\circ)})$, the inclusion $(C^*
  (\Sigma,G),C(G^{(\circ)}))$ is regular.  If in addition,
   is $C(\unit{G})$ is a MASA, then $\rad(C^* (\Sigma,G),C(G^{(\circ)}))=(0)$.
\end{proposition}
\begin{remark}{Remark}  The condition that $C(\unit{G})$ is a MASA is
  satisfied when $\unit{G}$ is
second countable and $G$ is essentially principal, that is, when the
interior of the isotropy bundle for $G$ is $\unit{G}$,
see~\cite[Proposition~4.2]{RenaultCaSuC*Al}.  We expect that it is
possible to remove the hypothesis of second countability here, but we
have not verified this.
\end{remark} 
\begin{proof}[Proof of Proposition~\ref{trivrad4twists}]
The existence of the conditional expectation is proved as
in~\cite[Proposition~4.3]{RenaultCaSuC*Al} or
\cite[Proposition~II.4.8]{RenaultGrApC*Al}, and we have already
observed that the inclusion is regular.  

When $C(\unit{G})$ is a MASA in $C^*(\Sigma,G)$,
the triviality of the radical follows from Proposition~\ref{homobehav}.
\end{proof}

\subsection{Compatible Eigenfunctionals and the 
  Twist for $\mathbf{(\C,\D)}$}

We turn next to a discussion of eigenfunctionals, for a certain class of
eigenfunctionals will yield our twist.
Recall (see~\cite{DonsigPittsCoSyBoIs}) that an \textit{eigenfunctional}
is a non-zero element $\phi\in\dual{\C}$ which is an eigenvector for both the
left and right actions of $\D$ on $\dual{\C}$; when this occurs, there
exist unique elements $\rho,\sigma\in\hat{\D}$ so that whenever
$d_1,d_2\in\D$ and $x\in \C$, we have
$\phi(d_1xd_2)=\rho(d_1)\phi(x)\sigma(d_2)$.  We
write
$$s(\phi):=\sigma\dstext{and} r(\phi):=\rho.$$

\begin{definition} A \textit{compatible eigenfunctional} is a
  eigenfunctional $\phi$ such that for every $v\in
 \N(\C,\D)$,
\begin{equation}\label{cedef}
|\phi(v)|^2\in\{0,s(\phi)(v^*v)\}.
\end{equation}
Let $\ce(\C,\D)$ denote the set consisting of the zero functional
together with the set of all compatible eigenfunctionals, and let
$\ceo(\C,\D)$ be the set of compatible eigenfunctionals which have
unit norm.  Equip both $\ce(\C,\D)$ and $\ceo(\C,\D)$ with the
relative $\sigma(\dual{\C},\C)$ topology.

\end{definition}
\begin{remark*}{Remark} Notice that when $\phi$ is an eigenfunctional
  and $v\in \N(\C,\D)$ is such that $\phi(v)\neq 0$, then for every
  $d\in\D$, 
 \begin{equation}\label{sreig}
\frac{s(\phi)(v^*dv)}{s(\phi)(v^*v)}=r(\phi)(d).
\end{equation} Indeed,
$\phi(v)s(\phi)(v^*dv)
=\phi(vv^*dv)=\phi(dvv^*v)=r(\phi)(d)\phi(v)s(\phi)(v^*v)$. 
Thus, taking $d=1$, the condition in~\eqref{cedef} is equivalent to 
\begin{equation}\label{cedefalt} |\phi(v)|^2\in\{0,r(\phi)(vv^*)\}.
\end{equation}
\end{remark*}

We now show that associated with each $\phi\in \ceo(\C,\D)$ is a pair
$f,g\in \fS(\C,\D)$ which extend $r(\phi)$ and $s(\phi)$.  
Note that regularity of the inclusion
$(\C,\D)$ ensures the existence of $v\in \N(\C,\D)$ such that $\phi(v)>0$.
\begin{proposition}\label{cefprop}
Let $\phi\in\ceo(\C,\D)$, and
let $v\in \N(\C,\D)$ satisfy $\phi(v)>0$. 
Define elements $f,g\in\dual{\C}$ by
$$f(x)=\frac{\phi(xv)}{\phi(v)}\dstext{and}g(x)=\frac{\phi(vx)}{\phi(v)}.$$ 
Then the following statements hold.
\begin{enumerate}
\item[i)] $f,g\in \fS(\C,\D)$, $r(\phi)=f|_\D$ and $s(\phi)=g|_\D$.
\item[ii)] For every $x\in
  \C$, $$\phi(x)=\frac{g(v^*x)}{g(v^*v)^{1/2}}=
  \frac{f(xv^*)}{f(vv^*)^{1/2}}.$$
\item[iii)] For every $x\in\C$, $g(v^*xv)=g(v^*v)f(x)$ and $f(vxv^*)=f(vv^*)g(x).$ 
\end{enumerate}
\end{proposition}
\begin{proof}  
  The definitions show  $f|_\D=r(\phi)$ and $g|_\D=s(\phi)$.  
  We next claim that $\norm{f}=\norm{g}=1.$
  For any $d\in\D$ with $s(\phi)(d)=1$, replacing $v$ by $vd$ in the
  definition of $f$ does not
  change $f$. Thus,  if $x\in\C$ and $\norm{x}\leq 1$, we have
  $|f(x)|\leq \inf\{\frac{\norm{vd}}{\phi(v)}: d\in \D,
  s(\phi)(d)=1\}=1$ (because $d$ may be chosen so that
  $\norm{vd}=\norm{d^*v^*vd}^{1/2}$ is as close to
  $s(\phi)(v^*v)^{1/2}$ as desired).  This shows
  $\norm{f}=1$. Likewise $\norm{g}=1$.  As $f(1)=g(1)=1$, 
  both $f$ and $g$ are states
  on $\C$.

  If $w\in \N(\C,\D)$ and $f(w)\neq 0$, we have (using~\eqref{sreig}) 
  $$|f(w)|^2=
  \left|\frac{\phi(wv)^2}{\phi(v)}\right|^2=
  \frac{s(\phi)(v^*w^*wv)}{s(\phi)(v^*v)} =r(\phi)(w^*w)=f(w^*w),$$
  and it follows that $f\in \fS(\C,\D)$.  Likewise, $g\in
  \fS(\C,\D)$.

  Statements (ii) and (iii) are  calculations using~\eqref{cedef}
  and~\eqref{cedefalt}  whose verification is left to the reader.  
\end{proof}

\begin{remark*}{Notation} For $v\in\N(\C,\D)$ and $f\in \fS(\C,\D)$
  such that $f(v^*v)>0$, let $[v,f]\in\dual{\C}$ be defined by
  $$[v,f](x):=\frac{f(v^*x)}{f(v^*v)^{1/2}}= 
\innerprod{x+L_f,\frac{v+L_f}{\norm{v+L_f}_{\H_f}}}_{\H_f}.$$  
  (This notation is borrowed from Kumjian~\cite{KumjianOnC*Di}.
  There, Kumjian works in the context of \cstar-diagonals and uses
  states on $\C$ of the form $\sigma\circ E$ with $\sigma\in\hat{\D}$.
  As we assume no conditional expectation here, we replace functionals
  of the form $\sigma\circ E$, with elements from $\fS(\C,\D)$.  See
  also \cite{RenaultCaSuC*Al}.)

\end{remark*}

We have the following.
\begin{lemma}\label{cefform}  If $v\in \N(\C,\D)$ and $f\in\fS(\C,\D)$
  with $f(v^*v)>0$, then $[v,f]\in \ceo(\C,\D)$ and the following statements hold.
\begin{enumerate}
\item[i)]
 $s([v,f])=f|_\D$ and $r([v,f])=\beta_v(f|_\D)$.
\item[ii)] $[v,f]=[w,g]$ if and only if $f=g$ and $f(v^*w)>0$.
\end{enumerate}
\end{lemma}
\begin{proof}
  Suppose that $f\in\fS(\C,\D)$, $v\in\N(\C,\D)$ and $f(v^*v)\neq 0$.
  Let $\phi=[v,f]$. A calculation shows that $\phi$ is a norm-one
  eigenfunctional and that statement (i) holds.  

If $w\in \N(\C,\D)$ and $\phi(w)\neq 0$, then
  $$|\phi(w)|^2=\frac{|f(v^*w)|^2}{f(v^*v)}=\frac{f(v^*ww^*v)}{f(v^*v)}=
  \beta_v(s(\phi))(ww^*)=r(\phi)(ww^*),$$ so $\phi$ belongs to
  $\ceo(\C,\D)$ by~\eqref{cedefalt}.  

  Turning now to part (ii), suppose that
   $\phi=[v,f]=[w,g].$ Then we have $s(\phi)=f|_\D=g|_\D$.
  For every $x\in\C$, Proposition~\ref{cefprop} gives
  $$ f(x)=\frac{\phi(vx)}{\phi(v)}\dstext{and}
  g(x)=\frac{\phi(wx)}{\phi(w)}.$$ Since
  $\frac{g(w^*v)}{g(w^*w)^{1/2}}=\phi(v)=f(v^*v)^{1/2}$, we obtain
 $$ g(w^*v)=f(v^*v)^{1/2}g(w^*w)^{1/2}>0.$$  Likewise, $f(v^*w)>0.$   Also,
$$f(x)=\frac{\phi(vx)}{\phi(v)}=\frac{[w,g](vx)}{[v,f](v)}=
\frac{g(w^*vx)}{f(v^*v)^{1/2}g(w^*w)^{1/2}}=
\frac{g(w^*v)g(x)}{g(w^*v)}=g(x),$$ where the fourth equality follows
from  Proposition~\ref{Dextreme}.

Conversely, if $f\in\fS(\C,\D)$ and $v,w\in\N(\C,\D)$ with
$f(v^*w)>0$, Proposition~\ref{Dextreme} shows that
$f(v^*w)^2=f(w^*w)f(v^*v)$, so that in the GNS Hilbert space $\H_f$,
we have $\innerprod{v+L_f,w+L_f}=\norm{v+L_f}\norm{w+L_f}$.  By the
Cauchy-Schwartz inequality, there exists a positive real number $t$ so
that $v+L_f=tw+L_f$.  But then for any $x\in \C$,
$$[v,f](x)=\frac{\innerprod{x+L_f,v+L_f}}{\norm{v+L_f}}=
\frac{\innerprod{x+L_f,tw+L_f}}{\norm{tw+L_f}}=[w,f](x).$$
\end{proof}

Combining Proposition~\ref{cefprop} and Lemma~\ref{cefform} we obtain
the following.
\begin{theorem}\label{ceffromthm}
If $\phi\in \ceo(\C,\D)$, then there exist unique elements $\fs(\phi),
\fr(\phi)\in \fS(\C,\D)$ such that whenever $v\in\N(\C,\D)$ satisfies
$\phi(v)\neq 0$ and $x\in\C$,
$$\phi(vx)=\phi(v)\, \fs(\phi)(x) \dstext{and} \phi(xv)=\fr(\phi)(x)\,
\phi(v).$$
If $v\in \N(\C,\D)$ satisfies $\phi(v)>0$, then 
$\phi=[v,\fs(\phi)].$
Moreover, 
$$\ceo(\C,\D)=\{[v,f]: v\in \N(\C,\D), f\in \fS(\C,\D)
  \text{ and } f(v^*v)\neq 0\}.$$
\end{theorem}
\begin{proof}
Suppose $v,w\in\N(\C,\D)$ are such that $\phi(v)>0$ and $\phi(w)>0$.  For
$x\in\C$, set
$$f(x):=\frac{\phi(vx)}{\phi(v)}\dstext{and}
g(x):=\frac{\phi(wx)}{\phi(w)}.$$  
Proposition~\ref{cefprop} shows
that $\phi=[v,f]=[w,g]$.  Lemma~\ref{cefform} yields $f=g$.  
Another application of Proposition~\ref{cefprop} shows that for any
$x\in\C$, 
$$\frac{\phi(xv)}{\phi(v)}=\frac{\phi(xw)}{\phi(w)}.$$ 
Then
taking $\fs(\phi)=f$, and $\fr(\phi)= \frac{\phi(xv)}{\phi(v)}$,  we
obtain the result. 
\end{proof}

Notice that for $\phi\in\ceo(\C,\D)$, we have $\fs(\phi)\in F$ if and
only if $\fr(\phi)\in F$.  
\begin{definition}   Let $\ceoF(\C,\D):=\{\phi\in\ceo(\C,\D):
  \fs(\phi)\in F\}$.  We shall call  $\phi\in\ceoF(\C,\D)$ an
  \textit{$F$-compatible eigenfunctional}.  Notice that 
$$\ceoF(\C,\D)=\{[v,f]: f\in F\text{ and } f(v^*v)\neq 0\}.$$
\end{definition}

With these preparations in hand, we can show that $\ceoF(\C,\D)$ forms
a topological groupoid.  The topology has already been defined, so we
need to define the source and range maps, composition and inverses.

\begin{definition} Given $\phi\in\ceoF(\C,\D)$, let $v\in\N(\C,\D)$ be
  such that $\phi(v)> 0$.  We make the following definitions.
\begin{enumerate} 
 \item We say that $\fs(\phi)$ and $\fr(\phi)$
  are  the
  \textit{source} and \textit{range} of $\phi$ respectively.
  
\item Define the \textit{inverse}, $\phi^{-1}$ by the formula,
$$\phi^{-1}(x):=\overline{\phi(x^*)}.$$

If $\phi\in\ceo(\C,\D)$ and $v\in\N(\C,\D)$ is such that
  $\phi(v)>0$, (so that $\phi=[v,\fs(\phi)]$), then a calculation
  shows that
  $\phi^{-1}=[v^*,\fr(\phi)].$  The fact that $F$ is
  $\N(\C,\D)$-invariant 
ensures that $\phi^{-1}\in \ceoF(\C,\D)$.  Thus, our definition of $\phi^{-1}$ is
  consistent with the definition of inverse in the definition of the
  twist of a \cstardiag\ 
  arising in \cite{KumjianOnC*Di}
  and the twist of a Cartan MASA from \cite{RenaultCaSuC*Al}.

\item For $i=1,2$, let $\phi_i\in\ceoF(\C,\D)$.  We say
  that the pair $(\phi_1,\phi_2)$ is a \textit{composable pair} if
  $\fs(\phi_2)=\fr(\phi_1)$.  As is customary, we write
  $\ceoF(\C,\D)^{(2)}$ for the set of composable pairs.   
 To define the composition, choose $v_i\in
  \N(\C,\D)$ with $\phi_i(v_i)>0$, so that
  $\phi_i=[v_i,\fs(\phi_i)]$.   By Proposition~\ref{cefprop}(iii), we
  have
  $$\fs(\phi_2)(v_2^*v_1^*v_1v_2)=\fr(\phi_2)(v_1^*v_1)\fs(\phi_2)(v_2^*v_2)
  = \fs(\phi_1)(v_1^*v_1)\fs(\phi_2)(v_2^*v_2)> 0,$$ so that
  $[v_1v_2,\fs(\phi_2)]$ is defined.
  The product is then defined to be  
   $\phi_1\phi_2:=[v_1v_2,\fs(\phi_2)]$.

   We show now that this product is well defined.  Suppose that
   $(\phi_1,\phi_2)\in\ceoF(\C,\D)^{(2)}$, $f=\fs(\phi_2),$
   $\fr(\phi_2)=g=\fs(\phi_1),$ and that for $i=1,2$,
   $v_i,w_i\in\N(\C,\D)$ are such that $\phi_1=[v_1,g]=[w_1,g]$ and
   $[v_2,f]=[w_2,f].$ Then using Lemma~\ref{cefform}, we have
   $g(w_1^*v_1)>0$ and $f(v_2^*w_2)>0$, so, as $f\in\fS(\C,\D)$, there
   exists a positive scalar $t$ such that $v_2+L_f=tw_2+L_f$.  Hence,
\begin{align*}
 f((w_1w_2)^*(v_1v_2))&=\innerprod{\pi_f(v_1)(v_2+L_f),\pi_f(w_1)(w_2+L_f)}\\
&= 
t\innerprod{\pi_f(v_1)(w_2+L_f),\pi_f(w_1)(w_2+L_f)}\\
&=tf(w_2^*(w_1^*v_1)w_2)\\
&=tf(w_2^*w_2) \fr(\phi_2)(w_1^*v_1) \\ 
&=
tf(w_2^*w_2) \fs(\phi_1)(w_1^*v_1)\\
&=tf(w_2^*w_2) g(w_1^*v_1)>0.
\end{align*} By Lemma~\ref{cefform},   
   $[v_1v_2,f]=[w_1w_2,f]$, so that the product is well defined.

\item For $\phi\in\ceoF(\C,\D)$, denote the map $\C\ni x\mapsto
  |\phi(x)|$ by $|\phi|$.  Observe that for $\phi,
  \psi\in\ceoF(\C,\D)$, $|\phi|=|\psi|$ if and only if there exists
  $z\in\bbT$ such that $\phi=z\psi$.  Let
$\fRF(\C,\D):=\{|\phi|: \phi\in\ceoF(\C,\D)\}$.  
We now define source and range maps, along with inverse and product maps
on $\fRF(\C,\D)$.

Since a state on $\C$ is
determined by its values on the positive elements of $\C$,  we
identify $f\in\fS(\C,\D)$ with $|f|\in\fRF(\C,\D)$.  Define
$\fs(|\phi|)=\fs(\phi)$ and $\fr(|\phi|)=\fr(\phi)$.  Next we define
inversion in $\fRF(\C,\D)$ by $|\phi|^{-1}=|\phi^{-1}|$, and composable
pairs by
$\fRF(\C,\D)^{(2)}:=\{(|\phi|,|\psi|):(\phi,\psi)\in\ceo(\C,\D)^{(2)}\}$,
and the product by $\fRF(\C,\D)^{(2)}\ni (|\phi|,|\psi|)\mapsto
|\phi\psi|.$
Topologize $\fRF(\C,\D)$ with the topology of point-wise convergence:
$|\phi_\lambda|\rightarrow |\phi|$ if and only if
$|\phi_\lambda|(x)\rightarrow |\phi|(x)$ for every $x\in\C$.  
We call  $\fRF(\C,\D)$ the \textit{spectral
     groupoid over $F$} of $(\C,\D)$.

\item Define an action of $\bbT$ on $\ceoF(\C,\D)$ by $\bbT\times
  \ceoF(\C,\D)\ni (z,\phi)\mapsto z\phi$, where $(z\phi)(x)= \phi(zx)$.
Notice that if $\phi$ is written as 
$\phi=[v,f]$, where $v\in \N(\C,\D)$ and $f\in F$, then
$z\phi=[\overline{z}v,f]$. 

\end{enumerate}
\end{definition}

We have the following fact, whose proof is essentially the same 
as that of~\cite[Proposition~2.3]{DonsigPittsCoSyBoIs} (the continuity
of the range and source maps follows from their definition).
\begin{proposition}\label{loccmpt} The set 
  $\ceoF(\C,\D)\cup\{0\}$ is a weak-$*$ compact subset of
  $\dual{\C}$, and the maps $\fs,\fr:\ceoF(\C,\D)\rightarrow
  \fS(\C,\D)$ are weak-$*$--weak-$*$ continuous.
\end{proposition}

\begin{theorem}\label{cefgroupoid} Let $\ceoF(\C,\D)$ and $\fRF(\C,\D)$ be as above. 
  Then $\ceoF(\C,\D)$ and $\fRF(\C,\D)$ are locally compact Hausdorff
  topological groupoids and $\fRF(\C,\D)$ is an \'{e}tale groupoid.
  Their unit spaces are $\ceoF(\C,\D)^{(\circ)}
  =\fRF(\C,\D)^{(\circ)}=F$.  Moreover, $\ceoF(\C,\D)$ is a locally
  trivial topological twist over $\fRF(\C,\D)$.
\end{theorem}
\begin{proof}

That inversion on $\ceoF(\C,\D)$ is continuous follows readily from the
definition of inverse map and the weak-$*$ topology.  Suppose
$(\phi_\lambda)_{\lambda\in\Lambda}$ and
$(\psi_\lambda)_{\lambda\in\Lambda}$ are nets in $\ceoF(\C,\D)$ 
converging to $\phi, \psi\in \ceoF(\C,\D)$ respectively, and such that
$(\phi_\lambda,\psi_\lambda)\in \ceoF(\C,\D)^{(2)}$ for all $\lambda$. Since the
 source and range maps are continuous, we find that
 $\fs(\phi)=\lim_\lambda
 \fs(\phi_\lambda)=\lim_\lambda\fr(\psi_\lambda)=\fr(\psi)$, so
 $(\phi,\psi)\in\ceoF(\C,\D)^{(2)}.$  Let
$v,w\in\N(\C,\D)$ be such that $\phi(v)>0$ and $\psi(w)>0$.  There exists
 $\lambda_0$, so that $\lambda\geq \lambda_0$ implies 
 $\phi_\lambda(v)$ and $\psi_\lambda(w)$ are non-zero.  For each
 $\lambda\geq \lambda_0$, there exists scalars
 $\xi_\lambda,\eta_\lambda\in\bbT$ such that
 $\phi_\lambda(v)=\xi_\lambda[v,\fs(\phi_\lambda)]$ and
 $\psi_\lambda=\eta_\lambda[v,\fs(\psi_\lambda)]$.  Since
 $$\lim_\lambda\phi_\lambda(v)=\phi(v)=\lim_\lambda
 [v,\fs(\phi_\lambda)](v)
\dstext{and} 
 \lim_\lambda\psi_\lambda(v)=\psi(v)=\lim_\lambda [v,\fs(\psi_\lambda)](v),$$ we
 conclude that $\lim\eta_\lambda=1=\lim\xi_\lambda$.  So for any $x\in\C$,
\begin{align*}
(\phi\psi)(x)&=\frac{\fs(\psi)((vw)^*x)}{(\fs(\psi)((vw)^*(vw)))^{1/2}}=
\lim_\lambda
\frac{\fs(\psi_\lambda)((vw)^*x)}{(\fs(\psi_\lambda)((vw)^*(vw)))^{1/2}}
=\lim_\lambda [v,\fs(\phi_\lambda)][w,\fs(\psi_\lambda)]\\
&=\lim_\lambda(\phi_\lambda\psi_\lambda)(x),
\end{align*} giving  continuity of multiplication.
Notice that for $\phi\in\ceoF(\C,\D)$, $\fs(\phi)=\phi^{-1}\phi$ and
  $\fr(\phi)=\phi\phi^{-1}$, and $F\subseteq \ceoF(\C,\D)$.
Thus, $\ceoF(\C,\D)$ is a
locally compact Hausdorff topological groupoid with unit space
$F$.

The definitions show that $\fRF(\C,\D)$ is a groupoid.  By
construction, the map $q$ defined by $\phi\mapsto |\phi|$ is continuous and is a
surjective groupoid homomorphism.  The topology on $\fRF(\C,\D)$ is
clearly Hausdorff.  If $\phi\in \ceoF(\C,\D)$, and $v\in\N(\C,\D)$ is
such that $\phi(v)\neq 0$, then $W:=\{\alpha\in\fRF(\C,\D): \alpha(v)>
|\phi(v)|/2\}$ has compact closure so $\fRF(\C,\D)$ is locally compact.
Also, if $\alpha_1, \alpha_2\in W$ and $\fr(\alpha_1)=\fr(\alpha_2)=f$,
then writing $\alpha_i=|\psi_i|$ for $\psi_i\in \ceoF(\C,\D)$, we see
that $\psi_i(v)\neq 0$, so there exist $z_1, z_2\in\bbT$ so that for
$i=1,2$ and every $x\in\C$,
$\psi_i(x)= z_if(xv^*)f(v)^{-1}$.  Hence $\alpha_1=\alpha_2$ showing
that the range map is locally injective.  We already know that the
range map is continuous, so by local compactness, the range map is a
local homeomorphism.
 
Note that convergent nets in $\fRF(\C,\D)$ can be lifted to convergent
nets in $\ceoF(\C,\D)$.  Indeed, if $q(\phi_\lambda)\rightarrow
q(\phi)$ for some net $(\phi_\lambda)$ and $\phi$ in $\ceoF(\C,\D)$,
choose $v\in\N(\C,\D)$ so that $\phi(v)>0$.  Then for large enough
$\lambda$, $\phi_\lambda(v)\neq 0$, and we have
$[\phi_\lambda,\fs(\phi_\lambda)]\rightarrow \phi$.  Also,
$|[\phi_\lambda,\fs(\phi_\lambda)] |\rightarrow |\phi|.$ The fact that
the groupoid operations on $\fRF(\C,\D)$ are continuous now follows easily from the
continuity of the groupoid operations on $\ceoF(\C,\D)$.  Thus
$\fRF(\C,\D)$ is a locally compact Hausdorff \'{e}tale groupoid.

Finally, $q(\phi_1)=q(\phi_2)$ if and only if there exist $z\in\bbT$
so that $\phi_1=z\phi_2$.  Moreover, for each $v\in\N(\C,\D)$, the map
$f\mapsto [v,f]$, where $f\in\{g\in \fS(\C,\D): g(v^*v)>0\}$ is a
continuous section for $q$.  Also, the action of $\bbT$ on
$\ceoF(\C,\D)$. given above makes $\ceoF(\C,\D)$ into a $\bbT$-groupoid.
So $\ceoF(\C,\D)$ is a twist over $\fRF(\C,\D).$

\end{proof}

\begin{remark}{Notation}
We now let 
$$\Sigma=\ceoF(\C,\D)\text{ and } G=\fRF(\C,\D), \dstext{so that}\unit{G}=F.$$  For $a\in\C$, define
$\hat{a}:\ceoF(\C,\D)\rightarrow \bbC$ to be the `Gelfand' map:
for $\phi\in \ceoF(\C,\D)$, $\hat{a}(\phi)=\phi(a)$. 
Then $\hat{a}$ is a continuous equivariant
function on $\ceoF(\C,\D)$.
\end{remark}

 Note that if $w\in\N(\C,\D)$, then
$\hat{w}$ is compactly supported.  Indeed, for
$\phi=[v,f]\in\ceoF(\C,\D)$, $\phi\in\supp(\hat{w})$ if and only if
$[v,f](w)\neq 0$, which occurs exactly when $f(v^*w)\neq 0$.
Proposition~\ref{Dextreme} shows this occurs precisely when
$f(w^*w)\neq 0$.  Hence,
$$\supp{\hat{w}}=\{\phi\in\ceoF(\C,\D): \fs(\phi)(w^*w)\neq 0\},$$ and
it follows that $\hat{w}$ has compact support.  Moreover, $\hat{w}$ is
supported on a slice, so that we find $\hat{w}\in \N(C^*(\Sigma,G),C(G^{(\circ)})).$

Before stating the main result of this section, recall that
Proposition~\ref{invideal} shows that
$\K_F=\{x\in\C:
f(x^*x)=0\text{ for all } f\in F\}$ is an ideal of $\C$ whose
intersection with $\D$ is  trivial.  
\begin{theorem}\label{inctoid}  Let $(\C,\D)$ be a regular inclusion,
  and let $G:=\fRF(\C,\D)$ and $\Sigma:=\ceoF(\C,\D)$.  
The map sending $w\in \N(\C,\D)$  to $\hat{w}\in C^*(\Sigma,G)$ extends
uniquely to
a regular $*$-homomorphism $\theta:(\C,\D)\rightarrow
(C^*(\Sigma,G), C(G^{(\circ)}))$ with $\ker\theta= \K_F$.
Furthermore,  $\theta(\C)$ is dense in $C^*(\Sigma,G)$ in the
$\unit{G}$-compatible topology.  
\end{theorem}
\begin{remark*}{Remark}  In general, $\theta(\C)$ and $\theta(\D)$ may
  be proper subsets of $C^*(\Sigma,G)$
  and $C(G^{(\circ)})$ respectively.
\end{remark*}

 \begin{proof}  The point is that the norms on
   $\C/\K_F$ and $C^*(\Sigma,G)$ both arise from the left regular
   representation on appropriate spaces.  Here are the details.
 
We have already observed that the map $w\mapsto \hat{w}$ sends
normalizers to normalizers.  Let $\C_0=\spn{\N(\C,\D)}.$  Then for any
$a\in\C_0$, $\hat{a}\in C_c(\Sigma,G)$.  

Let $f\in F=G^{(\circ)}.$ Then $f$ can be regarded as either a state
on $\C$ or as determining a state on $C^*(\Sigma, G)$ via evaluation
at $f$.  We write $f_\C$ when viewing $f$ as a state on $\C$, and
$f_\Sigma$ when viewing $f$ as a state on $C^*(\Sigma, G)$.

Let $(\pi_{\C,f},\H_{\C,f})$ be the GNS representation of $\C$ arising
from $f_\C$, and let $(\pi_{\Sigma,f},\H_{\Sigma,f})$ be the GNS
representation of $C^*(\Sigma,G)$ determined by $f_\Sigma$.   (Writing
$x=f$, the restriction of $\pi_{\Sigma,f}$ to $C_c(\Sigma,G)$ is the
representation $\pi_x$ discussed above when defining the norm on
$\C_c(\Sigma,G)$.)

For typographical reasons, when the particular $f$ is understood, we
will drop the extra $f$ in the notation: 
write $\pi_{\C}$ or $\pi_\Sigma$ instead of $\pi_{\C,f}$  or
$\pi_{\Sigma,f}$.

Now fix $f\in \unit{G}$.  For $a\in \C_0$, we claim that
\begin{equation}\label{un}
\norm{\hat{a}+\N_f}_{\H_{\Sigma}}=\norm{a+L_f}_{\H_{\C}}.
\end{equation}
Let $T\subseteq \Lambda_f$ be chosen so that
$f(w^*w)=1$ for every $w\in T$ and so that $T$ contains exactly one
element from each $\sim_f$ equivalence class.  
Proposition~\ref{cefform} shows that when 
$w_1,w_2\in\Lambda_f$, we have $|[w_1,f]|=|[w_2,f]|$ if and only if
$w_1\sim_f w_2.$   
Proposition~\ref{compatpure} shows that $\{w+L_f:
w\in T\}$ is an orthonormal basis for $\H_{\C}$.
 Writing $\phi=[w,f]$,
we have,
\begin{align*}
\norm{\hat{a}+\N_f}_{\H_{\Sigma}}^2&=\sum_{\substack{|\phi|\in G\\ \fs(|\phi|)=f}}
|\hat{a}(|\phi|)|^2= \sum_{w\in T} |[w,f](a)|^2= \sum_{w\in T}
|f(w^*a)|^2\\ 
&=\sum_{w\in T} |\innerprod{a+L_f,w+L_f}|^2
=\norm{a+L_f}^2_{\H_{\C}}.
\end{align*}
It follows that the map $a+L_f\mapsto \hat{a}+\N_f$ extends to an
isometry $W_f:\H_{\C}\rightarrow\H_{\Sigma}$.

To see that $W_f$ is a unitary operator, fix $\xi\in C_c(\Sigma,G)$. 
Since
$\xi$ is compactly supported, the set $$S_\xi:=\{ |\phi|\in G:
\fs(|\phi|)=f \text{ and } \xi(|\phi|)\neq 0\}$$ is a finite set.  
Let $|\phi_1|,\dots, |\phi_n|$ be the elements of $S_\xi$. 
For $1\leq j\leq n$, we may find $v_j\in
\N(\C,\D)$ such that $|\phi_j|=|[v_j,f]|$.  By
Proposition~\ref{cefform}, we may assume each $v_j$ belongs to the set
$T$. 
Let $z_j=\xi(\phi_j)$ and set 
$a=\sum_{j=1}^n z_jv_j$.   Clearly $a\in\C_0$.   Using the
fact that $f(w^*w)=1$ for each $w\in T$ and the fact that
$f(w_1^*w_2)=0$ for distinct elements $w_1, w_2\in T$, we find that
for 
$1\leq k\leq n$,
$$[v_k,f](a)=z_k.$$ 
Then 
\begin{align*}
\norm{(\hat{a}-\xi)+\N_f}^2_{\H_\Sigma}
&=f_\Sigma\left((\hat{a}-\xi)^*\star (\hat{a}-\xi)\right) =\sum_{j=1}^n
\left|(\hat{a}-\xi)(|\phi_j|)\right|^2\\ 
&=\sum_{j=1}^n
\overline{(\hat{a}(\phi_j) - \xi(\phi_j))} (\hat{a}(\phi_j) -
\xi(\phi_j))  
=\sum_{j=1}^n
\left| ([v_j,f](a) -
\xi(\phi_j))\right|^2 \\
&=\sum_{j=1}^n (z_j-\xi(\phi_j))^2 =0.
\end{align*}
Therefore, $\{\hat{a}+\N_f: a\in \C_0\}=\{\xi +\N_f: \xi\in
C_c(\Sigma,G)\}$.  As this set is dense in $\H_\Sigma$, 
$W_f$ is a unitary operator.

Next, we show that for $a\in \C_0$, we have
\begin{equation}\label{rightnorm}
\pi_{\Sigma}(\hat{a})W_f=W_f\pi_{\C}(a).\end{equation}
To do this, it suffices to show that for each $v\in T$, 
$$\pi_{\Sigma}(\hat{a})W_f(v+L_f)=W_f\pi_{\C}(a)(v+L_f).$$
Letting $\phi=[w_1,f]$ and $y=[w_2,f]$ we find $\phi
y^{-1}=[w_1w_2,\beta_{w_2}(f)]$, and a computation yields,
$$\hat{a}(\phi y^{-1})\hat{v}(y)=\frac{f(w_1^*aw_2)
  f(w_2^*v)}{f(w_1w_1)^{1/2} f(w_2^*w_2)}=\begin{cases}0&\text{if
    $w_2\not\sim_f v$}\\ \ds \frac{f(w_1^*av)}{f(w_1^*w_1)^{1/2}} &\text{if
    $w_2\sim_f v$}\end{cases}  
=\begin{cases}0&\text{if
    $w_2\not\sim_f v$}\\ \widehat{av}(\phi)&\text{if
    $w_2\sim_f v.$}\end{cases}$$
Therefore, 
when $\fs(\phi)=f$, we have
$$(\hat{a}\star\hat{v})(\phi)=\sum_{\substack{|y|\in G\\
    \fs(|y|)=f}}(\hat{a}\circledast\hat{v})(\phi,|y|) 
=\widehat{av}(\phi).$$   So for $|y|\in G$ with $\fs(|y|)=f$,
$|(\hat{a}\star\hat{v}-\widehat{av})(|y|)| =0$.
Then,
$$\hat{a}\star\hat{v}+\N_f=\widehat{av}+\N_f$$ because
\begin{align*}
\norm{((\hat{a}\star\hat{v})-\widehat{av})+\N_f}^2_{\H_\sigma} &= 
\sum_{\substack{|y|\in G\\ s(|y|)=f}}
|(\hat{a}\star\hat{v}-\widehat{av})(|y|)|^2 
=0.
\end{align*}

Hence, 
$$\pi_{\Sigma}(\hat{a})W_f(v+L_f)=\pi_{\Sigma}(\hat{a})(\hat{v}+\N_f)=
\widehat{av}+\N_f 
=W_f(\pi_{\C}(a))(v+L_f),$$ which gives \eqref{rightnorm}.

The definition of the norm on $C^*(\Sigma,G)$ and~\eqref{rightnorm}
imply that for $a,b\in \C_0$,
$\norm{\widehat{ab}-\hat{a}\hat{b}}_{C^*(\Sigma,G)} =0$.  Therefore,
the map $a\in \C_0\mapsto \hat{a}$ is multiplicative and 
$$\norm{\hat{a}}_{C^*(\Sigma,G)}=\sup_{f\in
  \fS(\C,\D)}\norm{\pi_{\C,f}(a)}.$$   The existence of $\theta$ now
follows from continuity, and the fact that $\ker\theta=\K_F$
follows from Proposition~\ref{describerad}.  For $v\in\N(\C,\D)$,
$\theta(v)=\hat{v}$ belongs to $\N(C^*(\Sigma,G), C(\unit{G}))$, so
$\theta$ is a regular $*$-homomorphism.  

Finally, we turn to showing the $\unit{G}$-compatible  
density of $\theta(\C)$ in $C^*(\Sigma,G)$.  
Let $\M\subseteq\dual{C^*(\Sigma,G)}$ 
be the linear span of the evaluation functionals $\xi\mapsto
\xi(\sigma)$ where $\xi\in C^*(\Sigma,G)$ and $\sigma\in \Sigma$.  
Suppose $\mu\in \M$ annihilates $\theta(\C)$.  Then there exists
$n\in\bbN$, scalars $\lambda_1,\dots, \lambda_n$, elements $v_1,\dots,
v_n\in\N(\C,\D)$ and $f_1,\dots, f_n\in F$ such that  for any $\xi\in
C^*(\Sigma,G)$, 
$$\mu(\xi)=\sum_{k=1}^n \lambda_k\xi([v_k,f_k]).$$   Without loss of
generality we may assume that $[v_i,f_i]\neq [v_j,f_j]$ if $i\neq j$.
 Since $\mu$ annilhlates
$\theta(\C)$, for every $a\in \C_0$, 
$$0=\mu(\hat{a})=\sum_{k=1}^n \lambda_k[v_k,f_k](a).$$  Fix $1\leq
j\leq n$, and let $d,e\in \D$ be such that
$\beta_{v_j}(f_j)(d)=f_j(e)=1$.  For  $i\neq j$, since $[v_i,f_i]\neq
[v_j,f_j]$, either $f_j\neq f_i$ or $\beta_{v_i}(f_i)\neq
\beta_{v_j}(f_j)$.  Hence we assume that  $d$ and $e$ have been chosen
so
that if $i\neq j$, then $[v_i,f_i](dv_je)=0$.  Then 
$$\mu(\hat{v_j})=\lambda_jf_j(v_j^*v_j)^{1/2}=0.$$  As
$f_j(v_j^*v_j)\neq 0$, we obtain $\lambda_j=0$.  It follows that
$\mu=0$.  Since the dual of $C^*(\Sigma,G)$ equipped with the
$\unit{G}$-compatible topology is $\M$, we conclude that $\theta(\C)$
is dense in the $\unit{G}$-compatible topology on $C^*(\Sigma,G).$      
This completes the proof.

\end{proof}

\section{Applications}\label{appl}
In this section we give some applications which apply to regular MASA
inclusions with $\L(\C,\D)=(0)$.  
Theorem~\ref{embedcrossedprod} gives a very large class of such
inclusions.

Here is an application of our work to norming algebras.
We begin with a definition.  Recall that  $\N(\C,\D)$ is a closed 
$*$-semigroup containing $\D$.

\begin{definition}
A $*$-subsemigroup $\F\subseteq \N(\C,\D)$ with $\D\subseteq \F$
is \textit{countably generated over $\D$} if there exists a countable
set $F\subseteq \F$ so that the smallest $*$-subsemigroup of
$\N(\C,\D)$ containing
$F\cup \D$ is $\F$.    The set $F$ will be called a
\textit{generating set} for $\F$.  

We will say that the inclusion $(\C,\D)$ is \textit{countably regular}
if  there exists a $*$-subsemigroup $\F\subseteq \N(\C,\D)$ such that
$\F$ is countably 
generated over $\D$ and $\C=\overline{\spn}(\F)$.
\end{definition}

The following result
generalizes~\cite[Lemma~2.15]{PittsNoAlAuCoBoIsOpAl} 
and gives a large
class of norming algebras.   In particular, notice that the result
holds for Cartan inclusions.

\begin{theorem}\label{norming}  Suppose $(\C,\D)$ is a regular MASA
  inclusion such that  $\L(\C,\D)=(0).$ 
Then $\D$ norms $\C$.
\end{theorem}
\begin{proof}
Let $\F\subseteq \N(\C,\D)$ be a $*$-subsemigroup which is countably
generated over $\D$ by the (countable) set $F$.  Let $\C_\F\subseteq \C$ be the
\cstar-subalgebra  generated by $\F$.  (Notice that $\C_\F$ is simply
the closed linear span of $\F$.)   Then $(\C_\F,\D)$ is a countably regular MASA
inclusion.

We will show that $\D$ norms $\C_\F$.   
Let $$Y:=\{\sigma\in\hat{\D}:
\sigma \text{ has a unique state extention to $\C_\F$}\}.$$
Theorem~\ref{denseuep} shows that $Y$ is dense in $\hat{\D}$.  
For each element $\sigma\in Y$, let $\sigma'$  denote
the unique extension of $\sigma$ to all of $\C_\F$.  Notice that if
$\rho\in \widehat{I(\D)}$ and $\rho\circ\iota=\sigma$, then 
$\sigma'=\rho\circ E$ because $\sigma'|_\D=\sigma=\rho\circ E|_\D$.  

For $\sigma\in Y$, let $\pi_\sigma$ be the GNS representation for $\sigma'$.
Proposition~\ref{equivMASA} shows that 
$\pi_\sigma(\D)''$ is a MASA in $\B(\H_\sigma)$.   

Define an equivalence relation $R$ on $Y$ by $\sigma_1\sim\sigma_2$ if and
only if there exists $v\in\F$ such that
$\sigma_2=\beta_v(\sigma_1)$.  (Since $\F$ is a $*$-semigroup,
this is an equivalence relation.)  

We claim that if $\pi_{\sigma_1}$
is unitarily equivalent to $\pi_{\sigma_2}$, then
$\sigma_1\sim\sigma_2$.    To see this, we use a modification of the
argument 
in~\cite[Lemma~5.8]{DonsigPittsCoSyBoIs}.
Let $U\in \B(\H_{\sigma_2},\H_{\sigma_1})$ be a unitary operator such
that $$U^*\pi_{\sigma_1} U=\pi_{\sigma_2}.$$  Let $L_{\sigma_i}$ be
the left kernel of $\sigma_i'$.    Since $\pi_{\sigma_i}$ is
irreducible, 
$\C/L_{\sigma_i}=\H_{\sigma_i}$.  Hence  we may find $X\in \C$ such
that $U(I+L_{\sigma_2})=X+L_{\sigma_1}$.  Then for every $x\in \C$,
$$\sigma_2'(x)
=\innerprod{\pi_{\sigma_2}(x)(I+L_{\sigma_2}),(I+L_{\sigma_2})}
=\sigma_1'(X^*xX).$$  
Fix $\rho_i\in \widehat{I(\D)}$ such that
$\rho_i\circ\iota=\sigma_i$.  

The map $\C\ni x\mapsto \sigma_1'(X^*x)$ is a non-zero linear
bounded linear functional on $\C$.  Since $\spn(\F)$ is dense in
$\C$, there exists $v\in \F$ so that $\sigma_1'(X^*v)\neq 0.$  
The
Cauchy-Schwartz inequality for completely positive maps shows that for
any $d\in \D$, 
$$|\sigma_1'(X^*vd)|^2=\rho_1(E(X^*vd)E(d^*v^*X))\leq
\rho_1(E(X^*vdd^*v^*X))=\sigma_1'(X^*vdd^*v^*X)=\sigma_2(vdd^*v^*).$$
When $d\in\D$ and $\sigma_1(d)\neq 0$, we have
$\sigma_1'(X^*vd)=\sigma_1'(X^*v)\sigma_1(d)\neq 0$.  
 Therefore, when $d\in \D$ satisfies $\sigma_1(d)\neq 0$, 
\begin{equation}\label{cscp1}
0< \sigma_2(vdd^*v^*).
\end{equation}
In particular, $\sigma_2(vv^*)\neq 0$.   For any $d\in
\D$, we have
$$\beta_{v^*}(\sigma_2)(d)=\frac{\sigma_2(vdv^*)}{\sigma_2(vv^*)}.$$
If $\beta_{v^*}(\sigma_2)\neq \sigma_1$, then there exists $d\in \D$
with $\sigma_1(dd^*)\neq 0$ and $\beta_{v^*}(dd^*)=0$.  But this is
impossible by ~\eqref{cscp1}.  So 
$\beta_{v^*}(\sigma_2)=\sigma_1$.  Hence $\sigma_1\sim\sigma_2$ as
claimed.

Thus, if $\sigma_1\not\sim\sigma_2$, then $\pi_{\sigma_1}$ and
$\pi_{\sigma_2}$ are disjoint representations (as they are
both irreducible).

Let $\Y\subseteq Y$ be chosen so that $Y$ contains exactly one element
from each equivalence class of $Y$.   Put $$\pi=\bigoplus_{\sigma\in \Y}\pi_\sigma.$$
Then 
\begin{equation}\label{kerpi}
\ker\pi=\bigcap_{\sigma\in \Y}\ker\pi_\sigma=\bigcap_{\sigma\in Y}
\ker\pi_\sigma = \{x\in\C_\F: \sigma'(z^*x^*xz)=0 \text{ for all $\sigma\in
  Y$ and all $ z\in
\C_\F$}\}.
\end{equation}  

We next prove that 
\begin{equation}\label{Ydense}
\L(\C_\F,\D)\supseteq \ker\pi.
\end{equation}  
Suppose to obtain a contradiction, that $x\in \ker\pi$ and  that $E(x^*x)$ is a
non-zero element of $I(\D)$.  Let $$L:=\overline{\{\rho\in \widehat{I(\D)}:
\rho(E(x^*x))> \norm{E(x^*x)}/2\}}.$$   Then $L$ is a clopen set.  By
Lemma~\ref{pcover}, $\iota^*(L)=\{\sigma\in \hat{\D}: \sigma=\rho\circ\iota \text{
  for some }\rho\in L\}$ has non-empty interior.   Since $Y$ is dense
in $\hat{\D}$, we may find $\sigma\in Y$ and $\rho\in L$ such that
$\rho\circ \iota=\sigma$.  Then $\sigma'(x^*x)=\rho(E(x^*x))\neq 0,$ so
$x\notin\ker\pi_\rho$, contradicting~\eqref{kerpi}.  Hence \eqref{Ydense} holds.

Since $\L(\C,\D)\supseteq \L(\C_\F,\D)\supseteq \ker\pi$, we see that $\pi$ is a
faithful representation of $\C_\F$.

Since the representations in the definition of $\pi$ are disjoint and
each $\pi_\sigma(\D)''$ is a MASA in $\B(\H_\sigma)$, 
$\pi(\D)''$ is an atomic MASA in $\B(\H_\pi)$.  Therefore, $\pi(\D)''$ is
locally cyclic (see \cite[p.~173]{PopSinclairSmithNoC*Al}) for
$\B(\H_\pi)$.  By~\cite[Theorem~2.7 and
Lemma~2.3]{PopSinclairSmithNoC*Al} $\pi(\D)$ norms $\B(\H_\pi)$.  But
then $\pi(\D)$ norms $\pi(\C_\F)$.  Since $\pi$ is faithful, $\D$
norms $\C_\F$.

Finally, suppose that $k\in \bbN$ and that  $x=(x_{ij})\in M_n(\C)$.  
For each $n\in\bbN$ and $i,j\in \{1,\dots, k\}$, we may find a finite set
$F_{n,i,j}\subseteq \N(\C,\D)$ so that $\norm{x_{ij}-\sum_{v\in F_{n,i,j}} v}<1/n. $
Let $$F=\cup \{F_{n,i,j}: n\in \bbN, i,j\in \{1,\dots, k\}\}.$$  Then $F$ is
countable.  Let $\F$ be the closed $*$-subsemigroup of $\N(\C,\D)$ 
generated by $F$ and $\D$.  Then for $i,j\in \{1,\dots, k\}$,
$x_{ij}\in \C_\F$.  
Since $\D$ norms
$\C_\F$, we conclude that 
$$\norm{x}_{M_k(\C)}=\norm{x}_{M_k(\C_\F)}=\sup\{\norm{RxC}: R\in M_{1,n}(\D), C\in
M_{n,1}(\D), \norm{R}\leq 1, \norm{C}\leq 1\}.$$ 
Hence $\D$ norms $\C$.

\end{proof}

For any norm closed subalgebra $\A$ of the \cstaralg\ $\C$, 
let $C^*(\A)$ be the \cstar-subalgebra of $\C$ generated by $\A$, and
let $C^*_e(\A)$ be the \cstar-envelope of $\A$. (There are a number of
references which discuss \cstar-envelopes; see  
\cite{BlecherLeMerdyOpAlThMo,EffrosRuanOpSp,PaulsenCoBoMaOpAl}.) 

The following result is a significant
generalization of~\cite[Theorem~4.21]{DonsigPittsCoSyBoIs}.
Theorem~\ref{Vrej} was 
observed by Vrej Zarikian, who has kindly consented to its inclusion
here.

\begin{theorem}\label{Vrej}  Let $\C$ and $\D$ be 
  \cstaralg s, with $\D\subseteq \C$ ($\D$ is not assumed
  abelian). Let $(I(\D),\iota)$ be an injective envelope for $\D$.
Suppose there exists a unique unital completely
  positive map $\Phi:\C\rightarrow I(\D)$ such that $\Phi|_\D=\iota$,
  and 
 assume also that  $\Phi$ is
  faithful.  Let $\A$ be a norm-closed (not necessarily self-adjoint)
  subalgebra of $\C$ such that $\D\subseteq \A\subseteq \C$.  Then the
  \cstar-subalgebra of $\C$ generated by $\A$ is the \cstar-envelope
  of $\A$.
\end{theorem}
\begin{proof}
 Let
$\theta:\A\rightarrow C^*_e(\A)$ be a unital completely isometric
(unital) homomorphism such that the \cstaralg\ generated by the image of
$\theta$ is $C^*_e(\A)$.    
  Then there exists a
unique $*$-epimorphism $q:C^*(\A)\rightarrow C^*_e(\A)$ such that
$q|_\A=\theta$.  Our task is to show that $q$ is one-to-one.

Since $I(\D)$ is injective in the category of operator systems and
completely contractive maps, there exists a unital completely
contractive map $\Phi_e:C^*_e(\A)\rightarrow I(\D)$ such that
$\Phi_e\circ\theta=\iota.$   
Also, there exists a unital completely contractive map
$\Delta:\C\rightarrow I(\D)$ so that $\Delta|_{C^*(\A)}=\Phi_e\circ
q$.  Then for $d\in\D$, we have $\theta(d)=q(d)$, so 
$\iota(d)=\Phi_e(\theta(d))=\Phi_e(q(d))=\Delta(d)$.  The uniqueness
of $\Phi$ gives $\Delta=\Phi$.  Then if $x\in C^*(\A)$ and $q(x)=0$,
we have $\Phi(q(x^*x))=0$, so $q(x^*x)=0$ by the faithfulness of
$\Phi$.  
Thus $q$ is one-to-one,
and the proof is complete.   
\end{proof}

We now obtain the following generalization
of~\cite[Theorem~2.16]{PittsNoAlAuCoBoIsOpAl}.  While the outline of
the proof is the same as the proof
of~\cite[Theorem~2.16]{PittsNoAlAuCoBoIsOpAl}, the details in
obtaining norming subalgebras are different.  
\begin{theorem}\label{niceiso}  For $i=1,2$, suppose that
  $(\C_i,\D_i)$ are regular MASA inclusions such that
  $\L(\C_i,\D_i)=(0)$  and that $\A_i\subseteq \C_i$ are norm closed
  subalgebras such that $\D_i\subseteq \A_i\subseteq \C_i$.  Let
  $C^*(\A_i)$ be the \cstar-subalgebra of $\C_i$ generated by $\A_i$.
  If
  $u:\A_1\rightarrow \A_2$ is an isometric isomorphism, then $u$ 
  extends uniquely to a $*$-isomorphism of $C^*(\A_1)$ onto
  $C^*(\A_2)$.  
\end{theorem}
\begin{proof} 
Theorem~\ref{norming} implies that $\D_i$ norms $\C_i$.  
Taken together, Theorem~\ref{uniquecpmap} and Theorem~\ref{Vrej} imply
that $C^*(\A_i)$ is the \cstar-envelope of $\A_i$.   
Finally, an application of \cite[Corollary~1.5]{PittsNoAlAuCoBoIsOpAl}
completes the proof.
\end{proof}

\bibliographystyle{amsplain} 
\def\cprime{$'$}
\providecommand{\bysame}{\leavevmode\hbox to3em{\hrulefill}\thinspace}
\providecommand{\MR}{\relax\ifhmode\unskip\space\fi MR }
\providecommand{\MRhref}[2]{%
  \href{http://www.ams.org/mathscinet-getitem?mr=#1}{#2}
}
\providecommand{\href}[2]{#2}

\end{document}